\documentclass[12pt]{amsart}
\usepackage[margin=1in]{geometry}
\pdfoutput=1
\usepackage[T1]{fontenc}
\usepackage[utf8]{inputenc}
\usepackage[pagebackref=true,destlabel=true,colorlinks]{hyperref}
\usepackage[dvipsnames, table]{xcolor}
\hypersetup{%
  linkcolor = Purple,
  citecolor = Blue,
  urlcolor = Blue,
}
\usepackage{%
  amsmath,
  amssymb,
  amsthm,
  mathrsfs,
  mathtools,
  manfnt,
  nicematrix,
  tikz-cd,
  booktabs,
  fourier-orns,
  soul
}
\usetikzlibrary{decorations.pathreplacing}
\usepackage[mathscr]{euscript}
\usepackage{ytableau}
\ytableausetup{centertableaux}

\usepackage{enumitem}
\setlist[enumerate]{label=(\alph*)}

\theoremstyle{definition}
\numberwithin{equation}{subsection}

\theoremstyle{plain}
\newtheorem{thm}[equation]{Theorem}
\newtheorem*{thm*}{Theorem}
\newtheorem*{mainthm*}{Main Theorem}
\newtheorem{lem}[equation]{Lemma}
\newtheorem{cor}[equation]{Corollary}
\newtheorem{conj}[equation]{Conjecture}
\newtheorem{prop}[equation]{Proposition}

\theoremstyle{definition}
\newtheorem{defn}[equation]{Definition}
\newtheorem{eg}[equation]{Example}

\theoremstyle{remark}
\newtheorem{rem}[equation]{Remark}
\newtheorem*{note}{Note}

\newcommand{\stacks}[1]{%
  \cite%
  [{\href{https://stacks.math.columbia.edu/tag/#1}{\textsc{#1}}}]%
  {stacks}}

\makeatletter
\@addtoreset{equation}{section}
\renewcommand{\thesubsection}{%
  \ifnum\c@subsection<1 \@arabic\c@section
  \else \thesection.\@arabic\c@subsection
  \fi
}
\makeatother

\colorlet{hlcol}{Yellow!25}
\sethlcolor{hlcol}

\newcommand\sorry{%
  \relax
  \PackageWarning{unresolved sorry: }{missing detail!!}
  \ifmmode
    {\text{\color{BrickRed}\raisebox{2pt}{\Large\warning}}}
  \else
    {{\color{BrickRed}\Large\warning}}
  \fi}

\renewcommand{\emptyset}{\varnothing}

\newcommand{\injects}{\hookrightarrow}

\def\ZZ{\mathbf{Z}}

\def\CC{\mathbf{C}}

\DeclareMathOperator{\Hom}{Hom}

\DeclareMathOperator{\Tor}{Tor}

\DeclareMathOperator{\grade}{grade}
\DeclareMathOperator{\rank}{rank}
  
\DeclareMathOperator{\gr}{gr} 

\DeclareMathOperator{\Spec}{Spec}

\DeclareMathOperator{\Jac}{\mathbf{J}}
\DeclareMathOperator{\codim}{codim} 
\DeclareMathOperator{\Gr}{\mathbf{Gr}} 
\DeclareMathOperator{\Flag}{\mathbf{Flag}} 

\DeclareMathOperator{\Sym}{Sym}

\DeclareMathOperator{\GL}{\mathbf{GL}}

\newcommand{\fA}{\mathfrak{A}}

\newcommand{\sA}{\mathscr{A}}

\newcommand{\bC}{\mathbf{C}}

\newcommand{\sC}{\mathscr{C}}

\newcommand{\sD}{\mathscr{D}}

\newcommand{\rE}{\mathrm{E}}
\newcommand{\sE}{\mathscr{E}}
\newcommand{\bF}{\mathbf{F}}

\newcommand{\bG}{\mathbf{G}}

\newcommand{\rH}{\mathrm{H}}

\newcommand{\sJ}{\mathscr{J}}
\newcommand{\bK}{\mathbf{K}}

\newcommand{\bL}{\mathbf{L}}

\newcommand{\sO}{\mathscr{O}}
\newcommand{\bP}{\mathbf{P}}

\newcommand{\sQ}{\mathscr{Q}}

\newcommand{\rR}{\mathrm{R}}
\newcommand{\sR}{\mathscr{R}}

\newcommand{\fS}{\mathfrak{S}}
\newcommand{\rS}{\mathrm{S}}
\newcommand{\sS}{\mathscr{S}}

\newcommand{\sT}{\mathscr{T}}

\newcommand{\sU}{\mathscr{U}}

\newcommand{\sV}{\mathscr{V}}

\newcommand{\fW}{\mathfrak{W}}

\newcommand{\sX}{\mathscr{X}}

\newcommand{\sY}{\mathscr{Y}}
\newcommand{\bZ}{\mathbf{Z}}

\newcommand{\sZ}{\mathscr{Z}}

\newcommand{\bd}{\mathbf{d}}

\newcommand{\fg}{\mathfrak{g}}

\newcommand{\fh}{\mathfrak{h}}

\newcommand{\bp}{\mathbf{p}}
\newcommand{\fp}{\mathfrak{p}}

\newcommand{\bq}{\mathbf{q}}

\newcommand{\rs}{\mathrm{s}}

\def\Schur{\mathbf{S}}

\def\gl{\mathfrak{gl}}
\def\g{\mathfrak{g}}

\DeclareMathOperator\sdim{sdim}

\DeclareMathOperator\sign{sgn}
\DeclareMathOperator\Fact{Fact}
\DeclareMathOperator\sFact{\mathbf{Fact}}
\DeclareMathOperator\Split{Split}
\DeclareMathOperator\sSplit{\mathbf{Split}}

\newcommand{\sslash}{\mathbin{/\mkern-6mu/}}

\newcommand{\gnot}{\g_{\bar 0}}
\newcommand{\gone}{\g_{\bar 1}}
\newcommand\ord[1]{{#1}_{\mathrm{bos}}}

\newcommand\ftyp[1]{\ensuremath{\mathtt{#1}}}
\renewcommand\ul[1]{\underline{#1}}

\NiceMatrixOptions{%
  margin,last-col,first-row,
  columns-width=0.8em,
  xdots={horizontal-labels,line-style=|-|}
}

\begin{document}
\title[Flag Superschemes and Compositional Varieties]{Cohomology of Flag Superschemes and Syzygies of Compositional Varieties}
\author{Abhik Pal}
\email{apal@ucsd.edu}
\address{%
  Department of Mathematics \\
  University of California San Diego \\
  La Jolla CA 92093 \\
  USA}

\begin{abstract}
  We study the coherent cohomology of partial flag supervarieties and calculate the sheaf cohomology of the structure sheaves of four new infinite families of flag supervarieties.
  We introduce compositional varieties, a generalization of determinantal varieties obtained by imposing rank conditions on compositions of a pair of maps.
  We study geometric properties of compositional varieties and, in cases of interest, compute their Tor-groups. 
  For each of the four families of flag supervarieties, we show that the graded sheaf cohomology is isomorphic to the tensor product of the singular cohomology ring of an ordinary partial flag variety and the graded Tor-groups of a compositional variety.
\end{abstract}
\maketitle


\section{Introduction}

Algebraic supergeometry is a generalization of algebraic geometry that first originated from supersymmetric field theories formulated by mathematical physicists in the 1970s.
Their main motivation was to introduce a geometric theory that encompasses bosonic (even) and fermionic (odd) variables.
In essence, algebraic supergeometry is the study of superschemes, geometric objects that locally resemble spectra of \(\mathbf Z_2\)-graded commutative rings.
Many basic objects from ordinary algebraic geometry have counterparts in algebraic supergeometry; however, these generalized objects are still poorly understood.
For instance, there exists no systematic treatment of cohomology groups of natural vector bundles on super analogues of homogeneous spaces.
The goal of the present paper is to calculate the sheaf cohomology of structure sheaves of four infinite families of flag supervarities.
Our results generalize analogous calculations for the supergrassmannian carried out by Sam--Snowden in \cite{ss2024:cohomology, ss2024:cohomology2} and by Sam in \cite{sam2026:borel}.

Our work is motivated by the open problem of understanding generalizations of the celebrated Borel--Weil--Bott theorem to algebraic supergeometry.
The problem was first posed by Manin in the 1980s, see \cite{pen1990:borelweilbott}, and approaches to a generalization have coalesced around two main perspectives.
The first perspective is from the viewpoint of the representation theory of Lie superalgebras and Penkov \cite{pen1990:borelweilbott}, Penkov--Serganova \cite{ps1997:characters}, Gruson--Serganova \cite{gs2010:cohomology}, and Coulembier \cite{cou2016:bott} have worked on techniques to handle so-called typical weight representations of algebraic supergroups.
As far as we are aware, with the exception of the projective superspace, the representation-theoretic approach also doesn't work well on other natural homogeneous bundles as they result in atypical weights.
The weights corresponding to the structure sheaves, for example, are trivial and therefore maximally atypical; thus, these cases cannot be understood using the representation theoretic approach.

The second perspective on this problem is from invariant theory and commutative algebra and was first explored by Sam--Snowden in \cite{ss2024:cohomology}; they calculated the sheaf cohomology of the structure sheaf of the supergrassmannian in the general linear case \cite{ss2024:cohomology} and in the periplectic case \cite{ss2024:cohomology2}.
In \cite{sam2026:borel} this approach was used to determine the sheaf cohomology of a range of Schur functors applied to the tautological bundles on the supergrassmannian.
The main goal of the current work is to follow this approach and determine the sheaf cohomology of the structure sheaf of partial flag supervarieties.

Let \(V = V_0|V_1\) be a complex vector superspace, that is \(V = V_0 \oplus V_1\) is \(\bZ_2\)-graded and \(V_0\) and \(V_1\) are, respectively, the even and the odd graded parts and \(\dim{V_0} = m\), \(\dim{V_1} = n\).
Let \(0 = p_0 \leq p_1 \leq \dots \leq p_t \leq p_{t + 1} = m\) and \(0 = q_0 \leq q_1 \leq \dots \leq q_t \leq q_{t + 1} = n\) be sequences of integers and consider the \(t\)-step partial flag supervariety
\begin{equation*}
  X = \Flag(p_1|q_1, \dots, p_t|q_t;\; V)
\end{equation*}
consisting of super subspaces \(R_1|S_1 \subseteq \dots \subseteq R_t|S_t \subseteq V_0|V_1\) where \(\dim{R_i} = p_i\) and \(\dim{S_i} = q_i\).
The set \(X\) can be given the structure of a superscheme, that is, it is a locally ringed space \((X, \sO_X)\) where \(\sO_X\) is a sheaf of \(\bZ_2\)-graded superrings compatible with the topology of \(X\).
The underlying topological space of \((X, \sO_X)\) is the product \(\Flag(p_1, \dots, p_1;\; V_0) \times \Flag(q_1, \dots, q_t;\;V_1)\) of two ordinary (that is, not super) partial flag varieties; we call \(\sO_X\) the structure sheaf of \(X\).

\begin{mainthm*}[Theorem~\ref{main-thm}]
  Let \(X\) be a \(t\)-step partial flag supervariety.
  Set \(\delta_i = p_i - q_i\), \(\delta = \delta_{t + 1}\), \(r_k = \sum_{i = 1}^k \min(q_i - q_{i - 1}, p_i - p_{i - 1})\), and \(r = r_{t + 1}\).
  Let \(\delta_{i + 1} \geq \delta_i\) for \(i \geq 3\) and \(Y \subseteq \Hom(V_0, V_1) \times \Hom(V_1, V_0)\) consist of pairs \((f, g)\) subject to the rank conditions specified below:
  \begin{equation*}
    \begin{array}{rlllll}
      \text{Configuration of } \delta_i
      & r_1 & r_2 & r_{i \geq 3} & r & \text{Conditions on } (f, g)\\
      \midrule
      \delta_3 \geq \delta_2 \geq \delta_1 \geq 0
      & q_1 & q_2    & q_i    & n & \varnothing
      \\
      \delta_3 \geq \delta_2 \ge 0 \geq \delta_1
      & p_1 & q_2+\delta_1 & q_i+\delta_1 & n+\delta_1 & \rank{g} \leq n + \delta_1
      \\
      \delta_3 \ge 0 \geq \delta_2 \geq \delta_1
      & p_1 & q_2+\delta_1 & q_i+\delta_1 & n+\delta_1 & \rank{g} \leq n + \delta_1
      \\
      \delta_3 \ge 0 \geq \delta_1 \geq \delta_2
      & p_1 & p_2    & q_i+\delta_2 & n+\delta_2 & \rank{g} \leq n + \delta_2
      \\
      \delta_3 \ge \delta_1 \geq \delta_2 \geq 0
      & q_1 & p_2-\delta_1 & q_i+\delta_2-\delta_1 & n-\delta_1+\delta_2 & \rank{fg} \leq n - \delta_1 + \delta_2
      \\
      \bottomrule
    \end{array}
  \end{equation*}
  In all the above cases, there is an isomorphism of graded algebras
  \begin{equation*}
    \rH^\bullet(X, \sO_X) = \rH^\bullet_{\mathrm{sing}}(\Flag(r_1, \dots, r_t;\; r), \bC) \otimes E^\bullet
  \end{equation*}
  where \(E^\bullet\) consists of the graded \(\Tor\)-groups 
  \begin{equation*}
    E^i = \bigoplus_{j \geq 0} \Tor^{A}_j(\sO_Y, \bC)_{i + j}
  \end{equation*}
  of \(\sO_Y\) over the coordinate ring \(A = \Sym(V_0 \otimes V_1^\ast) \otimes \Sym(V_1 \otimes V_0^\ast)\)  of \(\Hom(V_0, V_1) \times \Hom(V_1, V_0)\).
\end{mainthm*}

The general approach to this calculation involves a lot of moving parts and we spend the rest of introduction giving a high-level sketch our setup.

\subsection{General approach}
The primary geometric objects we are interested in are superschemes, that is, locally ringed spaces  \((X, \sO_X)\) where \(\sO_X\) is a sheaf of superrings.
The odd elements of \(\sO_X\) generate the ideal sheaf \(\sJ \subseteq \sO_X\) which determines the underlying ordinary scheme \((\ord{X}, \sO_X/\sJ)\)  of \(X\), we will call this the bosonic reduction \(\ord{X}\) of \(X\).
We first note that the \(\sJ\)-adic filtration on \(\sO_X\) determines a filtration on any super vector bundle \(\sE\) on \(X\) and so there is a spectral sequence
\begin{equation}
  \label{eq:i:main-ss}
  \rE^{p, q}_1 = \rH^{p + q}(\ord{X}, \gr_p\sE) \implies \rH^{p + q}(X, \sE).
\end{equation}
that effectively reduces the sheaf cohomology calculation of \(\sE\) to a calculation in ordinary algebraic geometry.

If \(\bG\) is an algebraic supergroup with parabolic subgroup \(\bP\), a flag supervariety is the quotient \(X = \bG / \bP\) and its bosonic reduction is a product of two ordinary flag varieties.
We take \(\bG = \GL(V)\), the general linear supergroup consisting of \(\bZ_2\)-graded automorphisms of \(V\), and note that its Lie superalgebra \(\g = \gl(m|n)\) has a natural \(\ZZ_2\)-grading where
\begin{equation*}
  \gnot = \Hom(V_0, V_0) \times \Hom(V_1, V_1), \quad
  \gone =  \Hom(V_0, V_1) \times \Hom(V_1, V_0),
\end{equation*}
consist, respectively, of all the even and odd endomorphisms of \(V\), respectively.
We fix the parabolic \(\bP\) and let \(\fp\) denote its Lie superalgebra; in this case, the tangent space at any given \(x = g\bP \in X\) is \(\fg/\fp^x\) where \(\fp^x = g \fp g^{-1}\).

The ideal sheaf \(\sJ\) defines \(\ord{X}\) as a closed subscheme of \(X\), so the quotient \(\sJ / \sJ^2\) is the conormal bundle associated to \(\ord{X}\) and its fibers at \(x\) are given by \((\gone / \fp_{\bar 1}^x)^\ast\).
We thus obtain a short exact sequence
\begin{equation}
  \label{eq:i:main-ses}
  0 \to \sJ / \sJ^2 \to \gone^\ast \otimes \sO_{\ord{X}} \to \eta \to 0
\end{equation}
of ordinary vector bundles on \(\ord{X}\) where the fiber of \(\eta\)  at \(x\) is \((\fp_{\bar 1}^x)^\ast\).
The total space \(Z\) of \(\eta^\ast\) is a subbundle of \(\ord{X} \times \gone\).
We see from \eqref{eq:i:main-ses} that as a module over \(\sO_{\ord{X} \times \gone}\), \(\sO_Z\) is resolved by the Koszul complex \(\bK(\sJ / \sJ^2)_\bullet = \bigwedge^\bullet (\sJ / \sJ^2)\).
Consider the diagram 
\begin{equation}
  \label{eq:i:basic-diag}
  \begin{tikzcd}
    & Z \ar["\pi'"]{dd} \ar["\varphi"]{dl} \ar["\subseteq"]{rr}
    &
    & \ord{X} \times \gone \ar["\pi = \pi_{\gone}"]{dd} \ar["\pi_{\ord{X}}"]{rr}
    &
    & \ord{X}\\
    \widetilde{Y} \ar["\widetilde\varphi"]{dr} & \\
    & Y \ar["\subseteq"]{rr}
    &
    & \gone
  \end{tikzcd}
\end{equation}
where \(Y = \pi'(Z)\) and \(\pi' = \widetilde \varphi \circ \varphi\) is the Stein factorization of \(\pi'\).
In particular, \(\widetilde Y\) is the affinization of \(Z\), \(\sO_{\widetilde Y} = \Gamma(Z, \sO_Z)\), and \(\widetilde \varphi \colon \widetilde Y \to Y\) is a finite morphism.
If \(\rH^i(Z, \sO_Z) = 0\) for all \(i > 0\),  then, by cohomology and base change, one can show
\begin{equation}
  \label{eq:i:tor-ss}
  \rH^{q}(\ord{X}, \bigwedge^{p + q} \sJ/\sJ^2) = \Tor_p(\sO_{\widetilde Y}, \CC)_{p + q}.
\end{equation}
Using finiteness of \(\widetilde\varphi\), the \(\Tor\)-groups of \(\widetilde Y\) can be related to the \(\Tor\)-groups of \(Y\).
Our setup is essentially a slight modification of the Kempf--Weyman geometric technique that relates the cohomology of bundles on a projective variety to syzygies, see \cite[Theorem~5.1.2]{Wey:CohomVect}.

The precise relationship between \(Z\), \(Y\), and \(\widetilde{Y}\) is obtained by considering an analogue of the classical Grothendieck--Springer resolution (cf. \cite[Definition~3.2.4]{CG:RepresentationTheory}) and an ensemble of varieties related to the ring of invariants \(\CC[\gnot]^{\ord{\bG}}\).
If \(\fh \subseteq \fp_{\bar 0}\) is a Cartan subalgebra and \(\fW\) is its Weyl group, then the ensemble is related to the diagram
\begin{equation}
  \label{eq:i:ges-cons}
  \begin{tikzcd}[column sep=large]
    &
    \fp_{\bar 1}^x \ar{r} \ar{d} &
    \fp_{\bar 0}^x \ar{r} \ar{d} &
    \fp_{\bar 0}^x \sslash \ord{\bP} \ar["\cong"]{r} \ar{d} &
    \fh \sslash \fW_\bP \ar{d}
    \\
    Y \ar["\subseteq"]{r} & 
    \gone \ar["{v \mapsto [v, v]}"]{r}&
    \gnot \ar["\substack{\text{GIT}\\\text{quotient}}"]{r} &
    \gnot \sslash \ord{\bG} \ar["\cong"]{r}  &
    \fh \sslash \fW
  \end{tikzcd}
\end{equation}
where the rightmost isomorphisms are obtained by the Chevalley restriction theorem \cite[Theorem~3.1.38]{CG:RepresentationTheory} and \(\fW_\bP \subseteq \fW\) is the Weyl group associated to \(\ord{\bP}\).
Since each element of \(Z\) belongs to some \(\fp_{\bar 1}^x\), there is a map \(Z \to \fh \sslash \fW_{\bP}\).
In fact, as the target is affine, this map factors through \(\widetilde Y\).
We will call varieties arising from \eqref{eq:i:basic-diag} and \eqref{eq:i:ges-cons} the Grothendieck--Springer ensemble (GSE) associated to the flag supervariety \(X\).
We will analyze the GSE to obtain a description of the \(\Tor\)-groups of \(\widetilde Y\) in terms of the \(\Tor\)-groups of \(Y\).

In the case of the supergrassmannian, a key observation in \cite{ss2024:cohomology} was that the space  \(Y\) is a determinantal variety and therefore its syzygies are given by the well-known Lascoux resolution, see \cite{las1978:syzygies} or \cite[Proposition~6.1.2]{Wey:CohomVect}.
We show that for the families of partial flag supervarieties we consider, \(Y \subseteq \Hom(V_0, V_1) \times \Hom(V_1, V_0)\) is a generalization of a determinantal variety called a compositional variety.
Given  a pair of sequences \(r^0_1, r^0_2, \dots\) and \(r^1_1, r^1_2, \dots\) of nonnegative integers, the compositional variety associated to this data is 
\begin{equation*}
  Y
  = \left\{ (f, g) \middle\vert
    \begin{array}{cc}
      \rank(f) \leq r_1^0, & \rank(g) \leq r_1^1, \\
      \rank(gf)^{i - 1} \leq r_{i}^0, & \rank(fg)^{i - 1} \leq r_{i}^1
    \end{array}
  \right\}
\end{equation*}
We show that for all the flag supervarieties we study, the image of \(\pi'\) in \eqref{eq:i:basic-diag} is a compositional variety where the ranks on the compositions are functions of the parameters \(p_i\) and \(q_i\).
We establish geometric properties for these varieties and describe their \(\Tor\)-groups.

\subsection{Some questions and relation to other work}
\begin{enumerate}
\item
  We suspect that for all flag supervarieties the variety \(Y\) in the GSE is always compositional with bounds on the ranks depending on the configuration of the superdimensions; see Conjecture~\ref{conj:Y-always-comp}.
  
\item
  When the compositional variety is not determinantal, our calculation for the \(\Tor\)-groups in Theorem~\ref{simp-2:res} of \(Y\)  relies on using  Eagon--Northcott generic perfection to relate the \(\Tor\)-groups to those of a determinantal variety.
  This method doesn't generalize easily to give \(\Tor\)-groups of other compositional varieties.

\item
  Note that the bundle \(Z \subseteq \ord{X} \times \gone\) is equivariant with respect to the group \(\ord{\bG} = \GL(V_0) \times \GL(V_1)\) and so the space of generators \(\Tor_1^{A}(\sO_{Y}, \bC)_\bullet\) of \(\sO_Y\) as an \(A\)-module is a rational representation of \(\ord{\bG}\).
  The construction of the infinite families in our main theorem relies on first describing \(Y\) for a fixed flag supervariety on \(V = V_0 | V_1\) and then extending it to a longer flag supervariety on \(W = W_0 | W_1\) where \(V \subseteq W\).
  Using the inclusions \(\GL(V_0) \times \GL(V_1) \to \GL(W_0) \times \GL(W_1)\), one obtains a corresponding sequence of representations on \(\Tor_1(\sO_Y, \bC)\).
  Is this sequence representation stable in the sense of \cite{ss2016:grobner, ps2017:representation}?
  Also see Remark~\ref{rem:stability}.

\item
  A natural extension of our work is to consider partial flag supervarieties with respect to the complex periplectic group.
  In \cite{ss2024:cohomology2} it is shown that in the case of the supergrassmannian, the cohomology of the structure sheaf of the periplectic supergrassmannian  is related to the singular cohomology of a Grassmannian for either the symplectic or orthogonal group and syzygies of (skew-)symmetric determinantal varieties.
  It is conceivable that the natural generalization of compositional varieties to the (skew-)symmetric case plays a role in the cohomology of the structure sheaf of partial periplectic flag supervarieties.
  In a similar vein, another natural generalization is to consider other equivariant bundles on flag supervarieties; see \cite{sam2026:borel}.
\end{enumerate}

\null 

\subsection{Outline}
\begin{enumerate}
\item[\S\ref{s:prelim-i}]
  We establish preliminaries on algebraic supervarieties, relationship between syzygies and \(\Tor\)-groups, and the basic ingredients for determining cohomology of flag supervarieties.
  This section mostly establishes our notation for the various basic objects.
  The rest of the preliminaries appear in \S\ref{s:prelim-ii}.

\item[\S\ref{s:flags-and-types}]
  We introduce flag supervarieites and define their basic invariants.
  Based on these invariants we define the \emph{type} of a flag supervariety and describe a classification of flag supervarieties based on type.
  We show that the types exhibit a natural \(\bZ_2 \times \bZ_2\)-symmetry and establish that, as far as the sheaf cohomology calculation is concerned, it is enough to consider the equivalence classes with respect to this action.

\item[\S\ref{s:comp}]
  We introduce compositional varieties and determine their basic geometric properties.
  In the cases relevant to us, we calculate their graded \(\Tor\)-groups.

\item[\S\ref{s:prelim-ii}]
  We state preliminaries required in the analysis of the Grothendieck--Springer ensemble associated to a flag supervariety.
  This section mainly summarizes known results about splitting and factorization rings and sets the relevant notation.

\item[\S\ref{s:ges}]
  We analyze the Grothendieck--Springer ensemble of a flag supervariety.
  This section is the most technically involved as we study properties of all the objects appearing in \eqref{eq:i:basic-diag} and \eqref{eq:i:ges-cons}.
  In \S\ref{ss:gse-comp} we first identify the compositional varieties in the case of \(2\)-step partial flag supervarieites and describe a procedure for lifting results to flags of arbitrary length.
  In \S\ref{ss:gse-split} we analyze the splitting rings associated to the compositional varieties; this analysis is used in \S\ref{ss:gse-fact} to study the spectra of factorization rings, that is, spaces of invariants of splitting rings that appear as the space \(\widetilde Y\) in \eqref{eq:i:basic-diag}.
  The final result is a description of the \(\Tor\)-groups of \(\widetilde Y\) in terms of \(\Tor\)-groups of \(Y\) and the singular cohomology of ordinary partial flag varieties.

\item[\S\ref{s:cohomology}]
  We state Theorem~\ref{main-thm}, the main result of this paper.
  In \S\ref{ss:cohom-ss} we analyze the spectral sequence \eqref{eq:i:main-ss} and show that it degenerates.
  We then combine all the results to prove our main theorem in \S\ref{ss:mainthm}.
  We also derive the super Euler characteristic of the structure sheaf as a corollary.
\end{enumerate}

\subsection{Leitfaden}
The results in various subsections depend on each other in the following way:
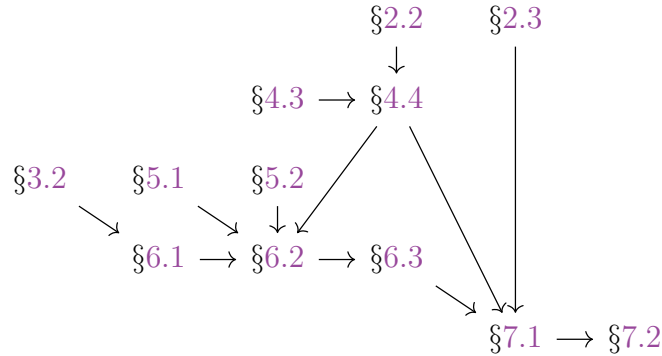
\begin{equation*}
  \begin{tikzcd}[row sep=small, column sep=small]
    &
    &
    & \S\ref{ss:geom-tech} \ar{d} 
    & \S\ref{ss:supercohom} \ar[crossing over]{dddd}
    &
    \\
    &
    & \S\ref{ss:simp-2-aux} \ar{r}
    & \S\ref{ss:simp-2-comp} \ar{ddl} \ar{dddr}
    &
    & \\
    \S\ref{ss:types} \ar{dr}
    & \S\ref{ss:norm-forms} \ar{dr}
    & \S\ref{ss:split-fact} \ar{d} 
    & & & \\
    & \S\ref{ss:gse-comp} \ar{r} 
    & \S\ref{ss:gse-split} \ar{r}
    & \S\ref{ss:gse-fact} \ar{dr} 
    & & \\
    & & &
    & \S\ref{ss:cohom-ss} \ar{r}
    & \S\ref{ss:mainthm}
  \end{tikzcd}
\end{equation*}

\subsection*{Acknowledgments}
We thank Steven Sam for introducing us to this problem and for many helpful discussions. The author was partially supported by NSF grant DMS 2302149.

\section{Preliminaries I}
\label{s:prelim-i}

Unless noted otherwise, we work over an algebraically closed field of characteristic zero.
Without loss of generality, we take this field to be \(\CC\).
We denote the quotient \(\ZZ/2\ZZ\) by \(\ZZ_2\).

\subsection{Superschemes and cohomology}
\label{ss:superschemes}

In this section we collect some basic definitions that appear in algebraic supergeometry.
Further details on the constructions presented here are available in \cite{Man:GaugeField,brp2023:AlgSuperGeom}.

Let \(R\) be a \(\ZZ_2\)-graded associative unital \(\CC\)-algebra and let \(J \subseteq R\) be the ideal generated by the odd, that is degree \(1\), elements.
A (commutative) superalgebra is such a ring \(R\) where \(J\) is finitely generated and \(xy = (-1)^{|x||y|}yx\) holds for homogeneous elements \(x, y \in R\) with degrees \(|x|, |y| \in \ZZ_2\).
Ideals, morphism, modules, etc are defined in the usual way with respect to the $\mathbf Z_2$-grading.
Note that the quotient \(\ord{R} = R / J\) is an ordinary algebra which we will call the bosonic reduction of \(R\).

The superspectrum of a commutative superring \(R\)  is the topological space \(\Spec{(\ord{R})}\) equipped with a sheaf of superalgebras derived from localizations of \(R\) in the usual way \cite[\S2.2]{brp2023:AlgSuperGeom}.
A locally ringed superspace is a pair \(X = (X, \sO_X)\) where \(X\) is a topological space, \(\sO_X\) is a sheaf of \(\ZZ_2\)-graded commutative rings such that for all \(x \in X\), the stalks \(\sO_{X, x}\) are local superrings.
A superscheme is a locally ringed superspace that is locally isomorphic to the superspectrum of a superring \cite[Definition~2.3]{brp2023:AlgSuperGeom}.
The odd elements of \(\sO_X\) generate a homogeneous ideal $\sJ = \sJ_X$ and the quotient \(\sO_X/\sJ_X\) is a sheaf of ordinary algebras; we call the ordinary locally ringed space \(\ord{X} = (X, \sO_X/\sJ_X)\)  the bosonic reduction of \(X\).
Note that \(\sO_X \to \sO_X / \sJ_X\) induces a closed immersion \(\ord{X} \injects X\) and underlying topological spaces of \(X\) and \(\ord{X}\) are identical.

There is a natural \(\sJ\)-adic filtration on \(\sO_X\) with the associated graded object given by 
\begin{equation}
  \label{eq:grO}
  \gr{\sO_X}
  = \bigoplus \gr_i \sO_X
  = \bigoplus \sJ^i / \sJ^{i + 1}
\end{equation}
If \(\sE\) is a sheaf of \(\sO_X\)-modules, in particular when \(\sE\) is a vector bundle on \(X\), we have \(\gr_i \sE = \sE \sJ^i / \sE \sJ^{i + 1}\).
There is an isomorphism \(\gr_i\mathscr E = \gr_0\mathscr E \otimes \bigwedge^i(\sJ / \sJ^2)\) over the ordinary sheaf of rings \(\sO_{\ord{X}}\) \cite[Chapter~4 \S3 Proposition~14(1)]{Man:GaugeField}.
A superscheme \(X\) is called smooth if the natural map \(\bigwedge (\sJ / \sJ^2) \to \gr \sO_X\) is an isomorphism.
We will often use the term ``supervariety'' to mean a smooth superscheme.

\subsection{Kempf--Weyman geometric method}
\label{ss:geom-tech}
We present here an overview of the Kempf--Weyman's geometric method for calculating syzygies of ordinary projective varieties.
All objects in this subsection are objects in ordinary algebraic geometry.
Details and proofs are available in \cite[Chapter~5]{Wey:CohomVect}.

Let \(X\) be a projective variety and \(E\) an affine space.
Picking a subbundle \(Z \subseteq X \times E\) defines a subvariety \(Y \subseteq E\) by restricting the projection \(\pi = \pi_E \colon X \times E \to E\) to \(Z\), see \eqref{eq:basic-diag}.
We consider \(X \times E\) as the total space of the trivial bundle \(\sE\) on \(X\) and consider the subvariety \(Z\) as the total space of a subbundle \(\sS \subseteq \sE\).
Let \(\xi = (\sE / \sS)^\ast\) and \(\eta = \sS^\ast\) be the dual bundles.
Suppose \(A = \Sym(E^\ast)\) is the coordinate ring of \(E\) with grading given by \(\deg{E^\ast} = 1\).
For every \(p \in \bZ\) define the graded \(A\)-modules
\begin{equation*}
  \bF_p =  \bigoplus_{j \geq 0} \rH^q(X, \bigwedge^{p + q}\xi) \otimes A(- p - q).
\end{equation*}

\begin{thm}[The Geometric Method]
  \label{t:geom-method}
  \begin{enumerate}
  \item
    There exists minimal differentials $d_i \colon \bF_i \to \bF_{i - 1}$ of degree \(0\) such that $\bF_\bullet$ is a complex of graded free $A$-modules and
    \begin{equation*}
      \rH_{-i}(\bF_\bullet) = \rR^i\pi'_\ast\sO_Z.
    \end{equation*}
    The complex \(\bF_\bullet\) is exact in positive degrees.
  \item
    If $\pi'$ is birational and $\rR^i\pi'_\ast\sO_Z = 0$ for all $i > 0$, then $\bF_\bullet$ is a minimal $A$-free resolution of the normalization of $Y$.
    Moreover, the normalization of $Y$ has rational singularities.
  \end{enumerate}
\end{thm}
\begin{proof}
  The claims follow from \cite[Theorem~5.1.2]{Wey:CohomVect} and
  \cite[Theorem~5.1.3]{Wey:CohomVect} respectively.
\end{proof}

When \(\pi'\) fails to be birational, we may use its Stein factorization \(\pi' = \widetilde{\varphi} \circ \varphi\) where \(\varphi \colon Z \to \widetilde{Y}\) is proper and has connected fibers, \(\widetilde{\varphi} \colon \widetilde{Y} \to Y\) is finite, and the space \(\widetilde{Y}\) is the affinization of \(Z\) with coordinate ring \(\sO_{\widetilde{Y}} = \Gamma(Z, \sO_Z)\) \stacks{03H0}.
In particular, \(\sO_{\widetilde{Y}}\) is a finite \(A\)-module.
In this case, the Tor groups of \(\sO_{\widetilde{Y}}\) are related to the resolution \(\bF_\bullet\) by \cite[Proposition~2.2]{ss2024:cohomology}.

We will refer to the following diagram as the basic diagram of the geometric method:
\begin{equation}
  \label{eq:basic-diag}
  \begin{tikzcd}
    & Z \ar["\pi'"]{dd} \ar["\varphi"]{dl} \ar["\subseteq"]{rr}
    &
    & X \times E \ar["\pi = \pi_E"]{dd} \ar["\pi_X"]{rr}
    &
    & X\\
    \widetilde{Y} \ar["\widetilde\varphi"]{dr} & \\
    & Y \ar["\subseteq"]{rr}
    &
    & E
  \end{tikzcd}
\end{equation}
In discussions involving this diagram, unless noted otherwise, we will use the notation for spaces and morphisms as shown above.

\subsection{Cohomology of superschemes}
\label{ss:supercohom}

If \(X\) is a supervariety, taking \(\xi = \sJ / \sJ^2\)  allows us to relate  the  cohomology groups \(\rH^i(X, \sO_X)\) to the syzygies of an appropriate ordinary variety appearing in the basic diagram \eqref{eq:basic-diag} of the geometric method.
\cite[Proposition~2.1]{ss2024:cohomology},  \cite[Proposition~2.2]{ss2024:cohomology}, and \cite[Theorem~2.4]{ss2024:cohomology} together imply the following important result:
\begin{thm}
  \label{t:super-geom-method}
  Let \(X\) be a smooth supervariety with \(\ord{X}\) projective.
  Suppose there is an exact sequence
  \begin{equation*}
    0 \to \sJ / \sJ^2 \to \epsilon \to \eta \to 0
  \end{equation*}
  of locally free coherent sheaves on \(\ord{X}\) with \(\epsilon = \sO_{\ord{X}} \otimes E^\ast\) globally free.
  Let \(A = \Sym(E^\ast)\), \(Z = \Spec(\Sym(\eta))\), and  \(\widetilde{Y}\) be the affinization of \(Z\).
  If \(\rH^i(Z, \sO_Z) = 0\) for all \(i > 0\), then there is a natural isomorphism
  \begin{equation*}
    \Tor^A_p(\sO_{\widetilde{Y}}, \CC)_{p + q}
    \cong
    \rH^q(\ord{X}, \bigwedge^{p + q} \sJ / \sJ^2)
  \end{equation*}
  and a spectral sequence
  \begin{equation*}
    \rE_1^{p, q} = \Tor^A_{-q}(\sO_{\widetilde{Y}}, \CC)_p
    \implies
    \rH^{p + q}(X, \sO_X).
  \end{equation*}
\end{thm}

\section{Flag supervarieties and types}
\label{s:flags-and-types}

\subsection{Flag superschemes}
\label{ss:superflags}
A flag superscheme is a quotient \(\bG / \bP\) of a linear algebraic supergroup \(\bG\) by an appropriate parabolic supersubgroup \(\bP\).
We will, however, work with the functor of points as it is better suited to our calculations.
Let \(V = V_0 | V_1\) be a vector superspace with \(\dim(V) = (m \mid n)\) that is, \(V = V_0 \oplus V_1\) and \(\dim{V_0} = m, \dim{V_1} = n\).
The superdimension of \(V\) is the difference \(\sdim{V} = \dim{V_0} - \dim{V_1} = m - n\).
Let \(\bp = (p_1, \dots, p_t)\) and \(\bq = (q_1, \dots, q_t)\) be weakly increasing sequences such that \(p_t \leq p_{t + 1} = m\), \(q_t \leq q_{t + 1} = n\) and for all \(i \leq t\), one of the inequalities \(p_i \leq p_{i + 1}\) and \(q_i \leq q_{i + 1}\) is strict.
The flag superscheme
\begin{equation*}
  X = \Flag(\bp|\bq;\, m | n) = \Flag(p_1|q_1, \dots, p_t|q_t;\, V)
\end{equation*}
is the superscheme representing the functor that attaches to every superalgebra \(T\) a chain of submodules \(R_1|S_1 \subseteq \dots \subseteq R_t|S_t\) of \(T \otimes V\) such that \(R_i|S_i\) is locally a summand of rank \((p_i \mid q_i)\).
Representability  and smoothness of this functor is shown in  \cite[Chapter~4 \S3]{Man:GaugeField}.
The underlying bosonic reduction of the flag supervariety is a product
\begin{equation*}
  \ord{X} = \Flag(p_1, \dots, p_t; V_0) \times \Flag(q_1, \dots, q_t; V_1)
\end{equation*}
of two ordinary partial flag varieties.
We set \(p_0 = 0\) and \(q_0 = 0\) and denote the superdimensions  of the subspaces appearing at points of \(X\) by \(\delta_i = p_i - q_i\) and \(\delta = \delta_{t + 1} = m - n\).

\subsection{Type classes and \(\bZ_2 \times \bZ_2\)-symmetry}
\label{ss:types}
Theorem~\ref{t:super-geom-method} implies that if two flag supervarieties have isomorphic bosonic reductions then they will result in isomorphic spectral sequences and thus have isomorphic cohomology groups.
As such, we classify flag supervarieties based on isomorphisms of their underlying bosonic reductions; as far as the cohomology calculation is concerned, we can reduce to equivalence classes up to these isomorphisms.
We formalize these isomorphisms by introducing a notion of types for flag supervarieties.

\begin{defn}[Type of a flag supervariety]
  The type of a flag supervariety \(X = \Flag(\bp|\bq;\; V)\) is the order relations on the tuple of superdimensions of its flags.
  That is, the type of a flag supervariety is a permutation \(\tau = (\delta_{i_0}, \dots, \delta_{i_{t + 1}})\) of the tuple \((\delta_0 = 0, \dots, \delta_{t + 1} = \delta)\) of superdimensions such that \(\delta_{i_0} \geq \delta_{i_1} \geq \dots \geq \delta_{i_{t + 1}}\) and \(i_j \geq i_{j + 1}\) if \(\delta_{i_j} = \delta_{i_{j + 1}}\).
  If \(X\) has type \(\tau\), we write \(X \colon \tau\).
\end{defn}

\begin{rem}
  We will often denote the type \( (\delta_{i_0}, \dots, \delta_{i_{t + 1}})\) by the string \(i_0 i_1\dots i_{t + 1}\) where we use the symbol \(\ast\) instead of \(i_k\) whenever \(i_k = t + 1\).
  For example, the type \((\delta, \delta_1, \delta_0, \delta_2)\) will be represented by the string \(\ftyp{\ast 1 0 2}\).
\end{rem}

To reduce the number of cases we need to study, we now determine isomorphism classes of types of superflags.
Our main approach is to pick suitable involutions and isomorphisms of \(V\) that result in isomorphisms of flags on \(V\).

Let \(P \colon V \to V\) be the parity change involution that swaps the even and the odd components that is \(P(V)_0 = V_1\) and \(P(V)_1 = V_0\).
At the level of flags,
\begin{equation*}
  P \left(R_1|S_1 \subseteq R_2|S_2 \subseteq \dots \subseteq R_t|S_t\right)
  =
  S_1|R_1 \subseteq S_2|R_2 \subseteq \dots \subseteq S_t|R_t
\end{equation*}
and therefore
\begin{equation}
  \label{eq:flag-iso-p}
  \Flag(p_1|q_1, \dots, p_t|q_t; V)
  \simeq
  \Flag(q_1|p_1, \dots, q_t|p_t; P(V)).
\end{equation}
Similarly, fixing an involution \(D \colon V \to \Hom(V, \CC) = V^\ast\), we see
\begin{equation*}
  D \left(R_1|S_1 \subseteq R_2|S_2 \subseteq \dots \subseteq R_t|S_t\right)
  =
  (V_0^\ast / R_t^\ast)| (V_1^\ast / S_t^\ast) \subseteq \dots \subseteq (V_0^\ast / R_1^\ast) | (V_1^\ast / S_1^\ast)
\end{equation*}
and thus obtain the isomorphism
\begin{equation}
  \label{eq:flag-iso-d}
  \Flag(p_1|q_1, \dots, p_t|q_t; V)
  \simeq
  \Flag(m - p_t|n - q_t, \dots, m - p_1|n - q_1; \Hom(V, \CC))
\end{equation}
of flag varieties.

\begin{prop}
  \label{p:z2-z2-symmetry}
  The isomorphisms \(P\) and \(D\)  from \eqref{eq:flag-iso-p} and \eqref{eq:flag-iso-d} generate a \(\ZZ_2 \times \ZZ_2\)-action on the set of types. In particular, they identify the types
  \begin{align*}
    P &\colon (\delta_{i_0}, \dots, \delta_{i_{t + 1}}) \mapsto (\delta_{i_{t + 1}}, \dots, \delta_{i_0}), \\
    D &\colon (\delta_{i_0}, \dots, \delta_{i_{t + 1}}) \mapsto (\delta_{t + 1 - i_{t + 1}}, \dots, \delta_{t + 1 - i_0}).
  \end{align*}
\end{prop}
\begin{proof}
  The proof effectively follows by observing that \(P\) and \(D\) are involutions and then comparing superdimensions of the flags obtained after applying \(P\) and \(D\).

  Let \(X = \Flag(p_1|q_1, \dots, p_t|q_t; V)\) and let \(\delta_i\) be the superdimensions calculated with respect to \(X\). If \(\delta'_i\) are the superdimensions computed with respect to \(P(X)\), then \(\delta'_i = q_i - p_i = -\delta_i\).
  Since
  \begin{equation*}
    \delta_{i_0} \leq \dots \leq \delta_{i_{t + 1}} \iff
    -\delta_{i_{t+1}} \leq \dots \leq -\delta_{i_0} \iff
    \delta'_{i_{t+1}} \leq \dots \leq \delta'_{i_0}.
  \end{equation*}
  The involution \(P\) identifies the corresponding types. Now suppose \(\delta'_i\) are the superdimensions computed with respect to \(D(X)\), then \(\delta'_i = \delta - \delta_{t + 1 - i}\) and thus
  \begin{equation*}
    \delta_{i_0} \leq \dots \leq \delta_{i_{t + 1}} \iff
    \delta -\delta_{i_{t+1}} \leq \dots \leq \delta -\delta_{i_0} \iff
    \delta'_{t + 1 - i_{t + 1}} \leq \dots \leq \delta'_{t + 1 - i_0}.
    \qedhere
  \end{equation*}
\end{proof}

\begin{rem}
  The \(\ZZ_2 \times \ZZ_2\)-action can be thought of as an action on the set of permutations of the string \(\ftyp{\ast} \ftyp{t} \dots \ftyp{0}\) where \(D\) reverses the string and \(P\) reverses the string and swaps the symbol \(i\) with the symbol \(t + 1 - i\).
\end{rem}

For purposes of the cohomology calculation, the discussion earlier implies that it is enough to consider equivalence classes with respect to this \(\ZZ_2\times\ZZ_2\)-action.
For instance when \(t = 1\), there are just two equivalence classes: \([\ftyp{\ast 1 0}]\) and \([\ftyp{1 \ast 0}]\).
This explains, \emph{a posteriori}, why the cohomology calculation for the supergrassmannian \(\Gr(p|q;\; n|m)\) in \cite[Theorem~1.2]{ss2024:cohomology} splits into two distinct sets of answers: one for \(n - m \geq p - q \geq 0\) (type \([\ftyp{\ast 1 0}]\))  and another for \(p - q \geq 0\) and \(p - q \geq n - m\) (type \([\ftyp{1 \ast 0}]\)):
\begin{thm}[{\cite[Theorem~1.2]{ss2024:cohomology}}]
  Let \(V = V_0|V_1\) be a vector superspace of dimension \((m \mid n)\) and \(X = \Gr(p|q;\; V) = \Flag(p|q;\; V)\).
  \begin{enumerate}
  \item If \(X \colon [\ftyp{\ast 1 0}]\) then there is an isomorphism
    \begin{equation*}
      \rH^\bullet(X, \sO_X) \cong \rH^\bullet_{\mathrm{sing}}(\Gr(q;\; \CC^n), \CC)
    \end{equation*}
    of graded algebras and the \(\GL(V)\)-action on \(\rH^\bullet(X, \sO_X)\) is trivial.

  \item If \(X \colon [\ftyp{1 \ast 0}]\) then there is an isomorphism
    \begin{equation*}
      \rH^\bullet(X, \sO_X)^{\GL(V)} \cong \rH^\bullet_{\mathrm{sing}}(\Gr(q;\; \CC^{m - \delta_1})) = A^\bullet
    \end{equation*}
    of graded algebras.
    There is a graded \(\GL(V)\)-representation \(E^\bullet\) such that \(\rH^\bullet(X, \sO_X) \cong A^\bullet \otimes E^\bullet\) as a graded \(\GL(V)\)-equivariant \(A^\bullet\)-module.
    The summand \(E^i\) of \(E^\bullet\) is the \(i\)-th linear strand of the determinantal variety
    \begin{equation*}
      Y_0 = \left\{ f \in \Hom(V_0, V_1) \mid \dim\ker{f} \geq \delta_1 = p - q \right\}.
    \end{equation*}
    As representations of \(\GL(V_0) \times \GL(V_1)\), we have
    \begin{equation*}
      E^i = \bigoplus_{j \geq 0} \Tor^S_j(\sO_{Y_0}, \CC)_{i + j}
    \end{equation*}
    where \(S = \Sym(V_0 \otimes V_1^\ast)\) and we regard \(\sO_{Y_0}\) as a quotient of \(S\).
  \end{enumerate}
\end{thm}

For longer flags, the number of equivalences of types grows quite substantially.
A precise count can be obtained as a corollary of Proposition~\ref{p:z2-z2-symmetry}.

\begin{cor}
  \label{c:num-types}
  Let \(r = \left\lfloor\frac{t + 2}{2}\right\rfloor\).
  There are \(\frac{1}{4} \left[ (t + 2)! + 2^r \cdot r! \right]\) equivalence classes of types of flags of length \(t\)
\end{cor}
\begin{proof}
  It is immediate that there are \((t + 2)!\) types and none of these are fixed by the action induced by \(D\) and by \(PD\).
  A type fixed by \(P\) can be determined by first fixing a permutation of the \(r\) pairs \(\{0, t + 1\}, \dots, \{r, t + 1 - r\} \) and then picking one of \(i\) or \(t + 1 - i\) entries; there are \(2^r r!\) ways of constructing such a string.
  The claim now follows from Proposition~\ref{p:z2-z2-symmetry} and Burnside's lemma.
\end{proof}

\begin{eg}[The eight types of \(2\)-step flag supervarieties]
  \label{eg:2-step-types}
  Let \(t = 2\).
  Then the orbits of \(\ZZ_2 \times \ZZ_2\)-action described in Proposition~\ref{p:z2-z2-symmetry}  are:
  \begin{align*}
    [\ftyp{\ast 2 1 0}]
    &= \left\{ \ftyp{\ast 2 1 0}, \ftyp{0 1 2 \ast} \right\},
    &
      [\ftyp{\ast 2 0 1}]
    &= \left\{
      \ftyp{\ast 2 0 1}, \ftyp{1 0 2 \ast}, \ftyp{2 {\ast} 1 0}, \ftyp{0 1 {\ast} 2}
      \right\},
    \\
    [\ftyp{\ast 1 2 0}]
    &= \left\{ \ftyp{\ast 1 2 0}, \ftyp{0 2 1 \ast}  \right\},
    &
      [\ftyp{\ast 1 0 2}]
    &= \left\{
      \ftyp{\ast 1 0 2}, \ftyp{2 0 1 \ast}, \ftyp{1 {\ast} 2 0}, \ftyp{0 2 {\ast} 1}
      \right\},
    \\
    [\ftyp{2 {\ast} 0 1}]
    &= \left\{ \ftyp{2 {\ast} 0 1}, \ftyp{1 0 {\ast} 2} \right\},
    &
      [\ftyp{\ast 0 2 1}]
    &= \left\{
      \ftyp{\ast 0 2 1}, \ftyp{1 2 0 \ast}, \ftyp{2 1 {\ast} 0}, \ftyp{0 {\ast} 1 2}
      \right\},
    \\
    [\ftyp{1 {\ast} 0 2}]
    &= \left\{ \ftyp{2 0 {\ast} 1}, \ftyp{1 {\ast} 0 2} \right\},
    &
      [\ftyp{\ast 0 1 2}]
    &= \left\{
      \ftyp{\ast 0 1 2}, \ftyp{2 1 0 \ast}, \ftyp{1 2 {\ast} 0}, \ftyp{0 {\ast} 2 1}
      \right\}.
  \end{align*}
\end{eg}

\section{Compositional varieties}
\label{s:comp}

Recall that if \(E\) and \(F\) are vector spaces then the determinantal variety
\begin{equation*}
  Y_r = \left\{ f \in \Hom(E, F) \ \middle|\ \rank{f} \leq r \right\}
\end{equation*}
is defined by the vanishing of \((r + 1) \times (r + 1)\) minors of a matrix of indeterminates in \(\Sym(E \otimes F^\ast)\) \cite[Chapter~6]{Wey:CohomVect}.
In this section we generalize this idea to compositional variety, a new family of algebraic sets determined by rank conditions on compositions of maps.
In cases of interest, we determine whether they are reduced, normal, have rational singularities, and compute their free resolutions. 


\subsection{Basic definitions and generators}

\begin{defn}[Compositional Variety]
  \label{def:comp}
  Let \(E\) and \(F\)  be vector spaces, \(\delta = \dim{E} - \dim{F}\), and \(\Delta_\bullet = \Delta_f, \Delta_g, \Delta_{fg}, \Delta_{gf}, \dots\) be a sequence of nonnegative integers satisfying
  \begin{equation}
    \label{eq:comp:Delta-restr}
    \begin{array}{l}
      \max(0, \delta) \leq \Delta_f \leq \Delta_{gf} \leq \Delta_{fgf} \leq \dots \leq \dim{E}, \\
      \max(0, -\delta) \leq \Delta_g \leq \Delta_{fg} \leq \Delta_{gfg} \leq \dots \leq \dim{F}.
    \end{array}
  \end{equation}
  The compositional variety associated to this data is the algebraic subset of \(\Hom(E, F) \times \Hom(F, E)\) defined by
  \begin{equation*}
    Y = 
    Y(\Delta_\bullet) = 
    \left\{ (f, g) \mid \dim \ker{h} \geq \Delta_h \text{ where } h = f, g, fg, gf, \dots \right\}.
  \end{equation*}
\end{defn}

\begin{note}
  The conditions \eqref{eq:comp:Delta-restr} essentially account for the chain of inclusions \(\ker{f} \subseteq \ker{gf} \subseteq \ker{fgf} \subseteq \dots\) and \(\ker{g} \subseteq \ker{fg} \subseteq \ker{gfg} \subseteq \dots\) and the fact that at least one of \(f \colon E \to F\) and \(g \colon F \to E\) will pick up a \(|\delta|\)-dimensional kernel. 
\end{note}

\begin{note}
  A compositional variety can be equivalently defined by imposing upper bounds on ranks of compositions.
  If the \(\Delta_\bullet\) are specified and \(h\) is a \(k\)-fold composition of \(f\) and \(g\), then the lower bound on \(\dim{\ker{h}}\) translates to 
  \begin{equation*}
    \rank{h} \leq
    \begin{cases}
      \dim{E} - \Delta_h &\text{if } h \in \Hom(E, E) \text{ or } h \in \Hom(E, F) \\
      \dim{F} - \Delta_h &\text{if } h \in \Hom(F, E) \text{ or } h \in \Hom(F, F)
    \end{cases}.
  \end{equation*}
\end{note}

We will use the convention that when only a subset of the parameters \(\Delta_\bullet\)  are specified, the rest of the \(\Delta_h\) are determined by making the weak inequality immediately preceding \(\Delta_h\) in \eqref{eq:comp:Delta-restr} an equality.
For example if \(\Delta_f = k_1\) and \(\Delta_{gfgf} = k_2\) then \(\Delta_{gf} = \Delta_{fgf} = k_1\), \(\Delta_{fgfgf} = \Delta_{gfgfgf} = \dots = k_2\), and \(\Delta_g = \Delta_{fg} = \dots = \max(0, -\delta)\).
In fact, Corollary~\ref{comp:finite-Delta-enough} shows that only a finite subset of the parameters from \(\Delta_\bullet\) are required to define \(Y\).

A compositional variety defined by a single parameter \(\Delta_h\) will be called simple.
A compositional variety defined only by parameters corresponding to \(k\)-fold compositions will be called a \(k\)-fold compositional variety. 

\begin{rem}
  \label{rem:comp-elsewhere}
  Special classes of compositional varieties are already known elsewhere in the literature; however, as far as we are aware, there is no general definition that treats all these classes of varieties.
  We list below some well-known varieties that turn out to be closely related to compositional varieties.
  \begin{enumerate}
  \item
    Determinantal varieties and their products are compositional.
    The product of determinantal varieties of maps \(f \colon E \to F\) with \(\rank{f} \leq r_1\) and maps \(g \colon F \to E\) with \(\rank{g} \leq r_2\) is the compositional variety defined by \(\Delta_f = \dim{F} - r_1\) and \(\Delta_g = \dim{E} - r_2\).

  \item
    If \(\mu\) is a partition then varieties of the form 
    \begin{equation*}
      Y_\mu =  \left\{ \varphi \in \Hom(E, E) \mid \dim{\ker{\varphi^i}} \geq \mu^\top_1 + \dots + \mu^\top_i  \right\}
    \end{equation*}
    appear in the study of nilpotent orbit closures \cite[Chapter~8]{Wey:CohomVect}.
    Letting \(E = F\), the variety \(Y_\mu\) can be identified with the intersection of the diagonal and a compositional variety defined by \(\Delta_h = \mu_1^\top + \dots \mu_i^\top\) for each \(i\)-fold composition \(h\).

  \item
    Varieties related to the space
    \begin{equation*}
      Y = \left\{ (f, g) \in \Hom(E, F) \times \Hom(F, E) \mid fg = gf = 0 \right\}
    \end{equation*}
    feature in the context of circular complexes in their relation to \(F\)-splittings \cite{mt1999:variety}.
    The space \(Y\) is the compositional variety corresponding to the parameters \(\Delta_{gf} = \dim{E}\) and \(\Delta_{fg} = \dim{F}\).
  \end{enumerate}
\end{rem}

Consider the affine variety \(\sA = \Hom(E, F) \times \Hom(F, E)\) and its coordinate ring
\begin{equation}
  A =
  \Sym(E \otimes F^\ast) \otimes \Sym(F \otimes E^\ast) \cong
  \CC[x_{i, j}]_{\substack{1 \leq i  \leq m \\ 1 \leq j \leq n}} \otimes
  \CC[y_{j', i'}]_{\substack{1 \leq i'  \leq m \\ 1 \leq j' \leq n}}.
\end{equation}
We respectively identify the indeterminates \(x_{i, j}\) and \(y_{j', i'}\) with the \((i, j)\)-th and \((j', i')\)-th coordinate functions on \(\Hom(E, F)\) and \(\Hom(F, E)\).
Let \(\Phi^f = (x_{i,j})_{\substack{1 \leq i \leq n \\ 1 \leq j \leq m}}\) and \(\Phi^g = (y_{j, i})_{\substack{1 \leq i \leq n \\ 1 \leq j \leq m}}\) be the generic linear maps in \(\Hom(E, F)\) and \(\Hom(F, E)\).
Note that any \(k\)-fold composition \(h\) of \(f\) and \(g\) can be identified with a \(k\)-fold composition \(\Phi^h\) of \(\Phi^f\) and \(\Phi^g\).

\begin{prop}
  \label{comp:algset-gens}
  Let \(\Delta_\bullet\) be the set of parameters defining a compositional variety \(Y\) and let
  \begin{equation*}
    r_h =
    \begin{cases}
      \dim{E} - \Delta_h &\text{if } h \in \Hom(E, E) \text{ or } h \in \Hom(E, F) \\
      \dim{F} - \Delta_h &\text{if } h \in \Hom(F, E) \text{ or } h \in \Hom(F, F)
    \end{cases}
  \end{equation*}
  for each \(k\)-fold composition \(h\).
  If \(I_h \subseteq A\) is the ideal generated by the \((r_h + 1) \times (r_h + 1)\) minors of \(\Phi^h\),  then \(Y\) is defined set theoretically by the ideal \(I = I_\Delta = \sum I_h\).
\end{prop}

Since the ambient space is Noetherian, we immediately have the following:

\begin{cor}
  \label{comp:finite-Delta-enough}
  A compositional variety is determined by a finite subset of the parameters \(\Delta_\bullet\).
\end{cor}

As the example below shows, the ideal \(I\) given above in Proposition~\ref{comp:algset-gens} need not be radical in general.
Later in Proposition~\ref{simp-2:reduced}, we determine some new cases in which \(I\) is indeed radical.

\begin{eg}
  \label{eg:comp:non-radical-I}
  Suppose \(\dim{E} = \dim{F} = 3\) and consider the simple compositional variety defined by \(\Delta_{fgf} = 2\).
  Take \(h = fgf\) and note that we have \(r_h = 1\) and so \(I = I_h\) is generated by \(2 \times 2\) minors of \(\Phi^h = \Phi^f \Phi^g \Phi^f\).
  A calculation in Macaulay2~\cite{gs:macaulay2} shows that \(I\) fails to be radical:
\begin{verbatim}
i1 : m = 3; -- dim E
i2 : n = 3; -- dim F
i3 : A = ZZ/1307[x_(1,1)..x_(n,m), y_(1,1)..y_(m,n)];
i4 : Phi0 = matrix apply(n, i -> apply(m, j -> x_(i+1, j+1))); -- Phi^f
i5 : Phi1 = matrix apply(m, i -> apply(n, j -> y_(i+1, j+1))); -- Phi^g
i6 : r = 1; -- rank bdd
i7 : Phi = Phi0 * Phi1 * Phi0; -- Phi^h for h = fgf
i8 : I = minors(r + 1, Phi); -- I_h
i9 : I' = radical(I);
i10 : I == I'
o10 = false
\end{verbatim}

  Since the terms of \(\Phi^h\) consist of degree three elements, each \(2 \times 2\) minor has degree \(6\) and there are \(9\) such minors; the ideal \(I\) is generated by \(9\) degree \(6\) elements.
  Continuing the Macaulay2 session above, we can compute the Betti tables for \(I\) and its radical \(I' = \sqrt{I}\):
  \begin{equation*}
    \begin{array}[t]{r|ccccc}
      I & 0 & 1 & 2 & 3 & 4 \\ \midrule
      0 & 1 & \cdot & \cdot & \cdot & \cdot \\
      1 & \cdot & \cdot & \cdot & \cdot & \cdot \\
      2 & \cdot & \cdot & \cdot & \cdot & \cdot \\
      3 & \cdot & \cdot & \cdot & \cdot & \cdot \\
      4 & \cdot & \cdot & \cdot & \cdot & \cdot \\
      5 & \cdot & 9 & \cdot & \cdot & \cdot \\
      6 & \cdot & \cdot & 9 & 1 & \cdot \\
      7 & \cdot & \cdot & 16 & 19 & \cdot \\
      8 & \cdot & \cdot & \cdot & \cdot & 2 \\
      9 & \cdot & \cdot & \cdot & \cdot & 1
    \end{array}
    \qquad
    \begin{array}[t]{r|ccccc}
      \sqrt{I} & 0 & 1 & 2 & 3 & 4 \\ \midrule
      0 & 1 & \cdot & \cdot & \cdot & \cdot \\
      1 & \cdot & \cdot & \cdot & \cdot & \cdot \\
      2 & \cdot & \cdot & \cdot & \cdot & \cdot \\
      3 & \cdot & 1 & \cdot & \cdot & \cdot \\
      4 & \cdot & 9 & 16 & 9 & \cdot \\
      5 & \cdot & \cdot & 1 & \cdot & 1 \\
    \end{array}
  \end{equation*}
  We see that the radical is generated  by \(9\) degree \(5\) generators and one additional generator in degree \(4\).
  If \(M\) is a matrix and \(a_1 < \dots < a_k\), \(b_1 < \dots < b_k\), then let \(M_{[a_1, \dots, a_k \mid b_1, \dots, b_k]}\) denote the minor of \(M\) computed using entries from rows \(a_1, \dots, a_k\) and columns \(b_1, \dots, b_k\).
  Let \(X = \Phi_0\) and \(Y = \Phi_1\).
  The only quartic generator of \(\sqrt{I}\) is given by sums of the \(2 \times 2\) minors
  \begin{equation*}
    g_0 = (XY)_{[1, 2 \mid 1, 2]} +  (XY)_{[1, 3 \mid 1, 3]} + (XY)_{[2, 3 \mid 2, 3]}
  \end{equation*}
  of \(XY\); the nine quintics are products of \(\det(X)\) and the \(2 \times 2\) minors of \(Y\), that is,
  \begin{equation*}
    g_i = \det(X) \cdot Y_{[a_1^i, a_2^i \mid b_1^i, b_2^i]}
  \end{equation*}
  where \(1 \leq a_1^i < a_2^i \leq 3\) and \(1 \leq b_1^i < b_2^i \leq 3\) for \(i = 1, \dots, 9\) account for all possible such pairs. 
\end{eg}

\subsection{Lascoux resolution and \(1\)-fold compositional varieties}
In some cases where the ideal defining a compositional variety is radical, we  will construct a minimal free resolution of the corresponding compositional variety using resolutions of determinantal varieties.
So we briefly recall the main result in the case of the determinantal variety and set the relevant notation. 

If \(r, s\) are nonnegative integers and \(\alpha\) and \(\beta\) are (integer) partitions satisfying \(\ell(\alpha) \leq s\) and \(\ell(\beta^\top) \leq s\), define partitions
\begin{align*}
  P_{r, s}(\alpha, \beta) &= (s + \alpha_1, \dots, s + \alpha_b, s^r, \beta_1, \dots, \beta_{\ell(\beta)}), \\
  Q_{r, s}(\alpha, \beta) &= (s + \beta^\top_1, \dots, s + \beta^\top_b, s^r, \alpha^\top_1, \dots, \alpha^\top_{\ell(\alpha^\top)}),
\end{align*}
whose Young diagrams can be visualized by 
\begin{equation*}
  \begin{tikzpicture}
  \begin{scope}[cm={1, 0, 0, -1, (0, 0)}, scale=0.6]
    \draw
    (0, 2.5) node[left] {$P_{r, s}(\alpha, \beta) = $}
    (1, 1) node {$s \times s$}
    (1, 2.5) node {$r \times s$}
    (3, 1) node {$\alpha$}
    (1, 4) node {$\beta$}
    ;
    \begin{scope}[scale=0.5, cm={1, 0, 0, 1, (0, 6)}]
      \draw
      (0, 0) -- (0, 7) -- (1, 7) -- (1, 5) -- (2, 5) -- (2, 3) -- (4, 3) -- (4, 0) -- (0, 0);
    \end{scope}

    \begin{scope}[scale=0.5, cm={1, 0, 0, 1, (4, 0)}]
      \draw
      (0, 0) -- (0, 4) -- (1, 4) -- (1, 3) -- (4, 3) -- (4, 0) -- (0, 0); 
    \end{scope}

    \draw
    (0, 0) rectangle (2, 2)
    (0, 2) rectangle (2, 3);
  \end{scope}

  \begin{scope}[cm={1, 0, 0, -1, (5, 0)}, scale=0.6]
    \draw
    (0, 2.5) node[left] {$Q_{r, s}(\alpha, \beta) = $}
    (1, 1) node {$s \times s$}
    (1, 2.5) node {$r \times s$}
    (3, 1) node {$\beta^\top$}
    (1, 4) node {$\alpha^\top$}
    ;
    \begin{scope}[scale=0.5, cm={0, 1, 1, 0, (4, 0)}]
      \draw
      (0, 0) -- (0, 7) -- (1, 7) -- (1, 5) -- (2, 5) -- (2, 3) -- (4, 3) -- (4, 0) -- (0, 0);
    \end{scope}

    \begin{scope}[scale=0.5, cm={0, 1, 1, 0, (0, 6)}]
      \draw
      (0, 0) -- (0, 4) -- (1, 4) -- (1, 3) -- (4, 3) -- (4, 0) -- (0, 0); 
    \end{scope}

    \draw
    (0, 0) rectangle (2, 2)
    (0, 2) rectangle (2, 3);
  \end{scope}
\end{tikzpicture}
.
\end{equation*}
If \(m = \dim{E}\) and \(n = \dim{F}\) then set \(L(r, s)\) to be the set of pairs of partitions
\begin{equation}
  \label{def:lascoux-a-b-parts}
  L(r, s) = \left\{ (\alpha, \beta)
    \;\middle\vert
    \begin{array}{ll}
      \ell(\alpha) \leq s, & \ell(\beta) \leq m - r - s, \\
      \ell(\alpha^\top) \leq n - r - s, & \ell(\beta^\top) \leq s
    \end{array}
  \right\}.
\end{equation}

\begin{thm}[Lascoux]
  \label{t:lascoux}
  Let \(Y_r = \left\{ f \in \Hom(E, F) \mid \rank{f} \leq r\right\}\) be a determinantal variety.
  If \(S = \Sym(E \otimes F^\ast)\) then a graded minimal \(S\)-free resolution of \(\CC[Y_r]\) is given by 
  \begin{equation*}
    \mathbf L_i
    = \bigoplus_{s \geq 0} \bigoplus_{\substack{(\alpha, \beta) \in L(r, s) \\ i = s^2 + |\alpha| + |\beta|}}
    \Schur_{P_{r, s}(\alpha, \beta)}(E) \otimes \Schur_{Q_{r, s}(\alpha, \beta)}(F^\ast) \otimes_{\CC} S.
  \end{equation*}
  If \(r\) divides \(q\) and \(q = rs\) for some \(s \geq 0\) then the graded \(\Tor\)-groups are given by 
  \begin{equation*}
    \Tor_p^S(\sO_{Y_r}, \CC)_{p + q}
    =
    \bigoplus_{\substack{(\alpha, \beta) \in L(r, s) \\ p = s^2 + |\alpha| + |\beta|}}
    \Schur_{P_{r, s}(\alpha, \beta)}(E) \otimes \Schur_{Q_{r, s}(\alpha, \beta)}(F^\ast).
  \end{equation*}
  If \(r\) does not divide \(q\), then \(\Tor^S_p(\sO_{Y_r}, \CC)_{p + q} = 0\).
\end{thm}
\begin{proof}
  The result originally appears in \cite{las1978:syzygies}.
  Our notation is closest to the proof presented as part of \cite[Proposition~6.1.3]{Wey:CohomVect}.
\end{proof}

\begin{prop}
  \label{simp-1:res}
  Let \(A = \Sym(E \otimes F^\ast) \otimes \Sym(E^\ast \otimes F)\).
  \begin{enumerate}
  \item The minimal graded \(A\)-free resolution \(\bF_\bullet\) of the simple compositional variety defined by \(\Delta_f\) is given by \(\bF_\bullet = \bL_\bullet^0 \otimes \Sym(E^\ast \otimes F)\) where \(\bL_\bullet^0\) is the Lascoux resolution of the determinantal variety
    \begin{equation*}
      Y_0 = \left\{ f \in \Hom(E, F) \mid \rank{f} \leq \dim{E} - \Delta_f \right\}
    \end{equation*}
    as a \(\Sym(E \otimes F^\ast)\)-module.

  \item The minimal graded \(A\)-free resolution \(\bG_\bullet\) of the simple compositional variety defined by \(\Delta_g\) is given by \(\bG_\bullet = \Sym(E \otimes F^\ast) \otimes \bL_\bullet^1\) where \(\bL_\bullet^1\) is the Lascoux resolution of the determinantal variety
    \begin{equation*}
      Y_1 = \left\{ g \in \Hom(F, E) \mid \rank{g} \leq \dim{F} - \Delta_g \right\}
    \end{equation*}
    as a \(\Sym(E^\ast \otimes F)\)-module.
  \end{enumerate}
  In both cases the graded \(\Tor\)-groups agree with those of the respective determinantal varieties.
\end{prop}
\begin{proof}
  Both claims are a direct consequence of Theorem~\ref{t:lascoux}.
  Note that the simple compositional variety \(Y\) defined by \(\Delta_f\) is isomorphic to the product \(Y_0 \times \Hom(F, E)\) and so a minimal free resolution of \(\CC[Y]\) as an \(A\)-module is obtained by tensoring the minimal free resolution of \(\CC[Y_0]\) as a \(\Sym(E \otimes F^\ast)\)-module with \(\Sym(E^\ast \otimes F)\).
  This operation preserves the \(\Tor\)-groups and the grading.
  The second claim follows similarly.
\end{proof}

\subsection{Auxiliary spaces for \(2\)-fold compositional varieties}
\label{ss:simp-2-aux}
We now introduce the setup for the geometric method relevant for calculating geometric invariants of some compositional varieties.
We also prove some of the main results that will be used in calculation of syzygies for \(2\)-fold compositional varieties.

We assume without loss of generality that \(\dim{E} \geq \dim{F}\) and, as in Definition~\ref{def:comp}, set \(\delta = \dim{E} - \dim{F}\).
Recall the affine variety \(\sA = \Hom(E, F) \times \Hom(F, E)\) and its coordinate ring \(A = \Sym(E \otimes F^\ast) \otimes \Sym(E^\ast \otimes F)\).
Let \(\sX = \Gr(\delta, E)\) and let \(0 \to \sR \to \sO_{\sX} \otimes E \to \sQ \to 0\) be its tautological sequence where \(\sR\) and \(\sQ\) are the bundles given by
\begin{align*}
  \sR &= \left\{ (v, R) \mid R \in \Gr(\delta, E), v \in R \right\} \subseteq E \times \Gr(\delta, E), \\
  \sQ &= \left\{ (v, R) \mid R \in \Gr(\delta, E), v \in E/R \right\} = (E \times \Gr(\delta, E)) / \sR.
\end{align*}

Consider the trivial vector bundle on \(\sX\) whose fibers are isomorphic to \(\sA\), the total space of this bundle is simply \(\sX \times \sA\).
Fix \(0 \leq r \leq \dim{F}\) and define the subvariety 
\begin{equation*}
  \sZ =
  \sZ_r =
  \left\{
    (K, f, g) \in \sX \times \sA
    \;\middle|
    \begin{array}{l}
      f(K) = 0 \\
      \rank(g \colon F \to E / K) \leq r
    \end{array}
  \right\}.
\end{equation*}
Now consider the natural diagram
\begin{equation}
  \label{diag:gm:comp-2}
  \begin{tikzcd}
    \sZ_r \ar["\subseteq"]{r} \ar["\Pi' = \Pi\big\vert_{\sZ}",swap]{d} & \sX \times \sA \ar["\Pi_{\sX}"]{r} \ar["\Pi = \Pi_{\sA}"]{d} & \sX \\
    \sY_r \ar{r} & \sA &
  \end{tikzcd}
\end{equation}
involving the projections \(\Pi = \Pi_\sA\) and \(\Pi_\sX\), the restriction \(\Pi'\) of \(\Pi\) to \(\sZ\), and \(\sY = \Pi'(\sZ)\).

\begin{prop}
  \label{aux:comp-2:upstairs}
  Let \(m = \dim{E}\) and \(n = \dim{F}\).
  The minimal free resolution of \(\sO_{\sZ}\) as an \(\sO_{\sX \times \sA}\)-module is given by a complex \(\bF_\bullet\) where 
  \begin{equation*}
    \bF_t =
    \bigoplus_{\substack{|\lambda| \leq t,\; 0 \leq s\\(\alpha, \beta) \in L(r, s)\\ t - |\lambda| = s^2 + |\alpha| + |\beta|}}
    \Schur_\lambda(\sR) \otimes
    \Schur_{Q_{r, s}(\alpha, \beta)}(\sQ^\ast) \otimes
    \Schur_{\lambda^\top}(F^\ast) \otimes
    \Schur_{P_{r, s}(\alpha, \beta)}(F) \otimes_\CC A
  \end{equation*}

  Furthermore \(\dim{\sZ_r} = 2mn - (n - r)^2\).
\end{prop}
\begin{proof}
  Let \(\sZ_0\) and \(\sZ_1\) be varieties defined by
  \begin{align*}
    \sZ_0 &= \left\{ (K, f) \mid K \subseteq \ker{f} \right\} \subseteq \sX \times \Hom(E, F), \\
    \sZ_1 &= \left\{ (K, f) \mid \rank(g \colon F \to E / K) \leq r \right\} \subseteq \sX \times \Hom(F, E).
  \end{align*}
  Note that we have \(\sZ = \sZ_0 \times_\sX \sZ_1\).

  Since \(\sZ_0\) consists of pairs \((K, f)\) satisfying \(K \subseteq \ker{f}\), we identify \(\sO_{\sZ_0}\) with the subbundle \(\Hom(\sQ, F)\) of \(\sX \times \sA\).
  Therefore, we obtain 
  \begin{equation*}
    \dim(\sZ_0)
    = \rank{(\sQ^\ast \otimes F)} + \dim{\sX} 
    = (m - \delta) n + (m - \delta) \delta
    = (m - \delta)(n + \delta) = mn.
  \end{equation*}
  As an \(\sO_{\sX \times \sA}\)-module, \(\sO_{\sZ_0}\) is defined by the vanishing of the sections of the subbundle \(\sR \otimes F^\ast\) and thus its resolution as an \(\sO_{\sX \times \sA}\)-module is given by the Koszul complex \(\bK_\bullet = \bigwedge^\bullet (\sR \otimes F^\ast)\).
  

  For the variety \(\sZ_1\), first consider \(\sO_{\sZ_1}\) as an \(\sO_{\sX \times \Hom(F, E)}\)-module.
  The fiber of \(\sZ_1\) above a fixed \(K \in \sX\) is given by \(Z_1 = \left\{ g \in \Hom(F, E) \mid \rank(g \colon F \to E / K) \leq r \right\}\) whose defining ideal \(I_{r + 1} \subseteq \Sym(F \otimes (E / K)^\ast)\) is generated by the \((r + 1) \times (r + 1)\) minors of the \(n \times (m - \delta) = n \times n\) matrix of indeterminates in \(\Sym(F \otimes (E / K)^\ast)\).
  By Theorem~\ref{t:lascoux}, the quotient ring is resolved by the Lascoux resolution and tensoring first by \(\Sym(F \otimes K^\ast)\) and then by \(\Sym(E \otimes F^\ast)\), we see that an \(\sO_{\sX \times \sA}\)-module resolution of \(\sO_{\sZ_1}\) is given by a resolution \(\bL_\bullet\) whose terms are
  \begin{equation*}
    \bL_i = \bigoplus_{s \geq 0} \bigoplus_{\substack{(\alpha, \beta) \in L(r, s) \\ i = s^2 + |\alpha| + |\beta|}}  \Schur_{P_{r, s}}(F) \otimes \Schur_{Q_{r, s}(\alpha, \beta)}(\sQ^\ast).
  \end{equation*}
  Since  \(\dim{Z_1} = \dim{\left(\Sym(F \otimes (E / K)) / I_{r + 1}\right)} + \dim{\Hom(F, K)}\), 
  \begin{align*}
    \dim{\sZ_1}
    &=  \dim(\sX) + \dim(Z_1) =  n\delta + (2nr - r^2 + n \delta)
      = 2mn - 2n^2 + 2nr - r^2 \\
    &= mn + n\delta - n^2 + 2nr - r^2
      = mn + n\delta - (n - r)^2.
  \end{align*}

  We can thus conclude that \(\sO_{\sZ}\) is resolved by \(\bF_\bullet = \mathbf{K}_\bullet \otimes \mathbf{L}_\bullet\) where \(\bF_t = \bigoplus_{a + b = t} \bK_a \otimes \bL_b\) and
  \begin{equation*}
    \bK_a = \bigwedge^a (\sR \otimes F^\ast) = \bigoplus_{|\lambda| = a} \Schur_\lambda(\sR) \otimes \Schur_{\lambda^\top}(F^\ast).
  \end{equation*}
  Simplifying and reindexing using \(a = |\lambda| \leq t\) and \(b = t - |\lambda|\) gives us the claimed form for \(\bF_t\).
  Finally, note that \(\sZ = \sZ_0 \times_{\sX} \sZ_1\) and the dimension calculations for \(\sZ_0\) and \(\sZ_1\) imply 
  \begin{align*}
    \dim{\sZ}
    &=  \dim{\sZ_0} + \dim{\sZ_1} - \dim{\sX} \\
    &=  mn + mn + n\delta - (n - r)^2 - n\delta
      = 2mn - (n - r)^2. \qedhere
  \end{align*}
\end{proof}

Recall that Theorem~\ref{t:geom-method} implies that the sequence of graded free \(A\)-modules
\begin{equation*}
  \bG_i = \bigoplus_{j \geq 0} \rH^j(\sX, \bF_{i + j}) \otimes A(-i - j)
\end{equation*}
is a complex and \(\rH_{-i}(\mathbf{G}_\bullet) = \rR^i \Pi'_\ast \sO_{\sZ}\).
Furthermore when \(\Pi'\) is birational and \(\rR^i\Pi'_\ast \sO_{\sZ} = 0\) for all \(i \geq 0\) then \(\sY\) is normal, has rational singularities, and \(\bG_\bullet\) is a minimal free resolution of \(\sY\).
Assuming birationality of \(\Pi'\) and using the diagram~\eqref{diag:gm:comp-2} we will now show that \(\sY\) is normal, Cohen--Macaulay, and has rational singularities.
In our case, the terms of \(\bG_\bullet\) are obtained by computing cohomology groups over a Grassmannian \(\sX\), so the general strategy is to use Borel--Weil--Bott to show that the required cohomology groups vanish.
The main result here is Proposition~\ref{aux:comp-2:rat-sing} but we first prove some combinatorial properties associated to the partitions appearing in \(\bF_\bullet\). 

For partitions \(\lambda\) and \(\mu\), set
\begin{equation}
  C(\lambda, \mu) =
  \left\{
    (i, j) \mid
    (\lambda_i - i) + (\mu_j - j) \geq 0
  \right\}.
\end{equation}
Let \(u\) and \(v\) be tuples whose terms are given by
\begin{equation*}
  u_i = \Bigl| \{ j \mid (i, j) \in C(\lambda, \mu) \} \Bigr|,
  \quad
  v_j = \Bigl| \{ i \mid (i, j) \in C(\lambda, \mu) \} \Bigr|.
\end{equation*}
Our goal is to show that \(|C(\lambda, \mu)|\) calculates cohomology degree for summands of \(\bF_\bullet\) associated to weights \(\lambda\) and \(\mu\) and so, in a sequence of lemmas, we determine the necessary bounds.

\begin{lem}
  \label{lem:c-props}
  The following holds for the set \(C(\lambda, \mu)\) and the tuples \(u\) and \(v\):
  \begin{enumerate}
  \item \((i, j) \in C(\lambda, \mu)\) and \(i' \leq i\) and
    \(j' \leq j\) then \((i', j') \in C(\lambda, \mu)\),
  \item \(u_i \geq u_{i + 1}\) and \(v_j \geq v_{j + 1}\).
  \end{enumerate}
  In particular, \(C(\lambda,\mu)\) can be identified with the Young diagram of a partition \(\theta\).
\end{lem}
\begin{proof}
  Since \(\lambda\) and \(\mu\) are partitions
  \(\lambda_{i'} \geq \lambda_{i}\) and \(\mu_{j'} \geq \mu_{j}\), so
  \begin{align*}
    \lambda_{i'} - i' + \mu_{j'} - j' \geq \lambda_i - i + \mu_j - j \geq 0
  \end{align*}
  and \((i', j') \in C(\lambda, \mu)\). Observe that each \(u_i\) is maximal such
  that \((i, u_i) \in C(\lambda, \mu)\). If \(u_a < u_b\) for some
  \(a < b\) then
  \((b, u_b) \in C(\lambda, \mu) \implies (a, u_b) \in C(\lambda, \mu)\) contradicts maximality
  of \(u_a\). A similar argument shows \(v_j \geq v_{j + 1}\).
\end{proof}

\begin{eg}
  Lemma~\ref{lem:c-props} shows that \(C(\lambda, \mu)\) is the Young diagram of a partition \(\theta\) whose rows and columns are given by the tuples \(u\) and \(v\) respectively. For example, if \(\lambda = (4, 3, 3, 3, 3, 2, 2)\) and \(\mu = (6, 6, 4, 4, 3, 3, 2, 2)\) then \(\theta = (6, 4, 4, 3, 2, 2, 1)\).
  We can visualize the relationship between the partitions in the diagram below:
  \begin{equation*}
    \ytableausetup{boxsize=0.8em}
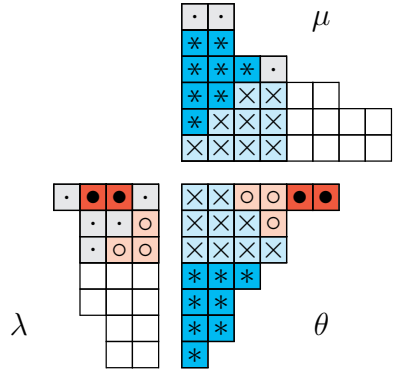
\begin{tikzpicture}

  \draw (0, 0) node[rotate=-90, anchor=north west]{%
    \ydiagram{7, 7, 5, 1}
    *[*(Red!80) \bullet]{7 + 0, 1, 1}
    *[*(Red!20) \circ]{1 + 2, 2 + 1}
    *[*(Gray!20) \cdot]{3, 3, 3, 1}
  };
  \draw (-2, -2) node {$\lambda$};

  \draw (0, 0) node[rotate=90, anchor=north west]{%
    \ydiagram{6, 6, 4, 4, 3, 3, 2, 2}
    *[*(Cyan!80) \ast]{1 + 4, 2 + 3, 3 + 1}
    *[*(Cyan!20) \times]{1, 2, 3, 3}
    *[*(Gray!20) \cdot]{6, 6, 4, 4}
  };
  \draw (2, 2) node {$\mu$};

  \draw (0, 0)  node[anchor=north west] {%
    \ydiagram{6, 4, 4, 3, 2, 2, 1}
    *[*(Red!20) \circ]{2 + 2, 3 +1}
    *[*(Cyan!20) \times]{2, 3, 4}
    *[*(Red!80) \bullet]{4 + 2}
    *[*(Cyan!80) \ast]{0, 0, 0, 3, 2, 2, 1}
  };
  \draw (2, -2) node {$\theta$};
\end{tikzpicture}

  \end{equation*}

  Recall that the rank of a partition is the length of the largest square contained in its Young diagram \cite[\S1.8]{Sta:EC1}.
  The ranks \(a = \rank{\lambda}\) and \(b = \rank{\mu}\) determine a partition of the boxes of \(\theta\) into three disjoint sets: a rectangle \(a \times b\) and partitions \(\theta^\lambda = (\theta_1 - b, \dots, \theta_a - b)\) and \(\theta^\mu = (\theta^\top_1 - a, \dots, \theta^\top_b - a)\).
  These can be identified with some of the boxes in the first \(a\) rows of \(\lambda\) and first \(b\) rows of \(\mu\) according to the coloring scheme shown above.
  For instance the boxes in \(\theta^\mu_j\) map to the boxes \((j, j + 1), \dots, (j, j + \theta^\mu_j)\) of \(\mu\).
  In Lemma~\ref{lem:c-bdd} we show that such a decomposition is always possible and use it to determine a bound on \(|C(\lambda, \mu)|\).
\end{eg}

\begin{lem}
  \label{lem:c-bdd}
  If \(\lambda\) and \(\mu\) are partitions then
  \begin{equation*}
    |C(\lambda, \mu)| \leq \sum_{i = 1}^{\rank{\lambda}} \lambda_i + \sum_{j = 1}^{\rank{\mu}}\mu_j.
  \end{equation*}
  If \(\rank{\lambda} + \rank{\mu} > 0\) then the inequality above is strict.
\end{lem}
\begin{proof}
  Set \(a = \rank{\lambda}\), \(b = \rank{\mu}\), and \(\theta\) be the partition whose Young diagram is given by \(C(\lambda, \mu)\).
  When \(i > a \) and \(j > b\), \(\lambda_i - i < 0\) and \(\mu_j - j < 0\) so \((i, j) \notin C(\lambda, \mu)\); on the other hand, if \(i \leq a\) and \(j \leq b\) we have \(\lambda_i \geq i\) and \(\mu_j \geq j\) so \((i, j) \in C(\lambda, \mu)\).
  Therefore \(\theta\) contains the rectangle \(a \times b\) Since \(b \leq u_i\) for \(i \leq a\) and \(a \leq v_j\) for \(j \leq b\), the rest of the boxes of \(\theta\) determine partitions
  \begin{align*}
    \theta^\lambda &=  (u_1 - b, \dots, u_a - b)  = (\theta_1 - b, \dots, \theta_a - b), \\
    \theta^\mu &= (v_1 - a, \dots, v_b - a) = (\theta^\top_1 - a, \dots, \theta^\top_b - a)
  \end{align*}
  satisfying \(|\theta| = ab + |\theta^\lambda| + |\theta^\mu|\).

  If \(u_i \geq b\) then \(\mu_{u_i} \leq b\) and \(u_i \leq \lambda_i - i + \mu_{u_i} \leq  \lambda_i - i + b\).
  Therefore
  \begin{equation*}
    |\theta^\lambda|
    = \sum_{i \leq a} (u_i - b)
    \leq \sum_{i \leq a} (\lambda_i - i)
    = \sum_{i \leq a} \lambda_i - \frac{a (a + 1)}{2}.
  \end{equation*}
  Similarly if \(j \leq b\) then \((a, j) \in C(\lambda, \mu)\) and \(v_j \leq \lambda_{v_j} + \mu_j - j \leq a + \mu_j - j\).
  In this case,
  \begin{equation*}
    |\theta^\mu|
    = \sum_{j \leq b} (v_j - a)
    \leq \sum_{j \leq b} (\mu_j - j)
    \leq \sum_{j \leq b} \mu_j - \frac{b (b + 1)}{2}.
  \end{equation*}


  Thus we can now conclude
  \begin{align*}
    |C(\lambda, \mu)| &= |\theta| = ab + |\theta^\lambda| + |\theta^\mu| \\
                      &\leq ab +
                        \left(\sum_{i = 1}^{\rank \lambda} \lambda_i - \frac{a(a + 1)}{2}\right) +
                        \left(\sum_{j = 1}^{\rank \mu} \mu_j - \frac{b(b + 1)}{2}\right) \\
                      &= \sum_{i = 1}^{\rank \lambda} \lambda_i +  \sum_{j = 1}^{\rank \mu} \mu_j
                        - \frac{1}{2} [(a - b)^2 + (a + b)] 
                        \leq \sum_{i = 1}^{\rank \lambda} \lambda_i +  \sum_{j = 1}^{\rank \mu} \mu_j.
  \end{align*}
  If \(a + b > 0\) then \( \frac12 [(a - b)^2 + (a + b)] \geq 1\) so in this case final inequality above is strict.
\end{proof}

\begin{lem}
  \label{lem:c-len-formula}
  Let \(\mu = (\mu_1, \dots, \mu_n)\) and \(\lambda = (\lambda_1, \dots, \lambda_\delta)\) be partitions, \(m = n + \delta\), and
  \begin{equation*}
    \gamma = (-\mu_n, \dots, -\mu_1, \lambda_1, \dots, \lambda_\delta),
    \qquad
    \rho = (m - 1, \dots, 1, 0).
  \end{equation*}
  If \(\gamma + \rho\) has no repeated entries and \(\sigma \in \mathfrak{S}_m\) is a permutation such that entries of \(\sigma(\gamma + \rho)\) appear in decreasing order, then \(\ell(\sigma) = |C(\lambda, \mu)| = \sum_{i \in [\delta]} u_i = \sum_{j \in [n]} v_j\).
\end{lem}
\begin{proof}
  The entries of \(\gamma + \rho\) are given by
  \begin{equation*}
    (\gamma + \rho)_i =
    \begin{cases}
      -\mu_{n + 1 - i} + m - i, &i \in [n] \\
      \lambda_{i - n} + m - i & (i - n) \in [\delta]
    \end{cases}.
  \end{equation*}
  Since the first \(n\) and the last \(\delta\) entries of \(\gamma\) are weakly
  decreasing and all entries of \(\rho\) are strictly decreasing, we have
  \((\gamma + \rho)_i < (\gamma + \rho)_j\) whenever \(i < j \leq n\) or
  \(n < i < j \leq m\). The total number of inversions required to sort
  \(\gamma + \rho\) is given by the cardinality of the set
  \begin{equation*}
    \mathrm{inv}(\sigma) = \left\{
      (i, j) \in [n] \times [\delta] \mid
      (\gamma + \rho)_i < (\gamma + \rho)_{j + n}
    \right\}.
  \end{equation*}
  Since \(i \mapsto (n + 1 - i)\) is an involution on \([n]\), the set above
  has the same cardinality as
  \begin{equation*}
    C' = \left\{
      (i, j) \in [n] \times [\delta] \mid
      (\gamma + \rho)_{n + 1 - i} < (\gamma + \rho)_{j + n}
    \right\}.
  \end{equation*}
  However
  \begin{align*}
    (\gamma + \rho)_{n + 1 - i}
    &=  - \mu_i + m - (n + 1 - i)  =  (- \mu_i + i) + \delta - 1, \\
    (\gamma + \rho)_{n + j}
    &= \lambda_j + m - (n + j)  =  \lambda_j+ \delta - j,
  \end{align*}
  and \((\gamma + \rho)_{n + 1 -i} < (\gamma + \rho)_{j + n}\) if and only if
  \begin{align*}
    \lambda_j + \delta - j &> (-\mu_i + i) + \delta - 1 \\
    \iff (\lambda - j) +  (\mu_i - i) &> (-1) \\
    \iff (\lambda - j) +  (\mu_i - i) &\geq 0.
  \end{align*}
  Thus we have shown \(C(\lambda, \mu) = C'\) and \(|\ell(\sigma)| = |\mathrm{inv}(\sigma)| = |C(\lambda, \mu)|\).
  Since each nonzero \(u_i\) counts the total number of \(j\) such that \((i, j) \in C(\lambda, \mu)\), \(|C(\lambda, \mu)| = \sum_{i \in [\delta]}u_i\).
  Similarly, \(|C(\lambda, \mu)| = \sum_{j \in [n]}v_j\).
\end{proof}

The combinatorial calculations above can now be used to prove our final result:

\begin{prop}
  \label{aux:comp-2:rat-sing}
  If \(\Pi'\) is birational then the variety \(\sY\) is normal and has rational singularities.
  In particular, \(\sY\) is Cohen--Macaulay.
\end{prop}
\begin{proof}
  It suffices to show that \(\rR^i\Pi'_\ast \sO_{\sZ} = 0\) for all \(i > 0\) that is, the modules \(\rH^i(\Gr(\delta, V_0), \bF_t)\) vanish for \(i \geq t\).
  When \(t > 0\) the summands of \(\bF_t\) are of the form \(\sE_{\lambda, \mu} = \Schur_\mu \sQ^\ast \otimes \Schur_\lambda \sR\) where \(\mu = Q_{r, s}(\beta^\top, \alpha^\top) = P_{r, s}(\alpha, \beta)\) for some \((\beta^\top, \alpha^\top) \in L(r, s)\) satisfying \(t - |\lambda| = s^2 + |\alpha| + |\beta|\).
  Now \(\rH^i(\sX, \sE_{\lambda, \mu})\) can be computed using Borel--Weil--Bott so we set \(\gamma = (-\mu_n, \dots, -\mu_1, \lambda_1, \dots, \lambda_\delta)\).
  Borel--Weil--Bott implies that \(\rH^i(\Gr(\delta, V_0), \sE_{\lambda, \mu}) \neq 0\) only when there is a nontrivial permutation \(\sigma \in \mathfrak{S}_m\) for which \(\sigma \bullet \gamma \neq \gamma\) and \(\sigma \bullet \gamma\) is dominant.
  In the latter case, we have \(i = \ell(\sigma)\).
  Let \(\sigma\) be any such permutation.
  If \(\sigma \bullet \gamma\) is dominant then \(\sigma(\gamma + \rho)\) is strictly decreasing and contains no repeats.
  By Lemma~\ref{lem:c-len-formula} the length is given by \(\ell(\sigma) = |C(\lambda, \mu)|\).
  Finally as \(t > 0\), \(|\lambda| + |\mu| > 0\) and using Lemma~\ref{lem:c-bdd} obtain the required bound
  \begin{equation*}
    \ell(\sigma)
    < \sum_{i = 1}^{\rank{\lambda}} \lambda_i +  \sum_{j = 1}^{\rank{\mu}}\mu_j
    \leq |\lambda| + s^2 + |\alpha| + |\beta| = t.
  \end{equation*}
  To obtain the penultimate inequality, we note that \(\mu = P_{r, s}(\alpha, \beta)\) implies \(\rank{\mu} = s\) and \(\sum_{j = 1}^s \mu_i = \sum_{j = 1}^s s + \alpha_j = s^2 + |\alpha|\).
\end{proof}

\subsection{Simple \(2\)-fold compositional varieties}
\label{ss:simp-2-comp}
In this subsection we use properties established about the auxiliary spaces to determine properties of the simple \(2\)-fold compositional variety
\begin{equation*}
  Y = \left\{ (f, g) \in \Hom(E, F) \times \Hom(F, E) \mid \dim{\ker{fg}} \geq \Delta_{fg} \right\}.
\end{equation*}
Without loss of generality suppose \(m = \dim{E} \geq n = \dim{F}\).
Unless stated otherwise, in this subsection we take \(r = n - \Delta_{fg}\).

\begin{thm}
  \label{simp-2:norm+rat-sing}
  Let \(Y\) be the simple \(2\)-fold compositional variety
  \begin{equation*}
    Y = \left\{ (f, g) \in \Hom(E, F) \times \Hom(F, E) \mid \dim{\ker{fg}} \geq \Delta_{fg} \right\}
  \end{equation*}
  and let \(r = n - \Delta_{fg}\).
  Then \(\sZ_r\) is a desingularization of \(Y\) and the variety \(Y\) is normal, Cohen--Macaulay, and has rational singularities.
\end{thm}
\begin{proof}
  We show that the auxiliary space \(\sZ_r\) given in \eqref{diag:gm:comp-2} is a desingularization of \(Y\) and we have \(Y = \sY_r\).
  Consider the open subset
  \begin{equation*}
    \sU = \left\{ (K, f, g) \in \sZ_r
      \;\middle|
      \begin{array}{l}
        \dim{\ker{f}} = \delta \\
        \rank(g \colon V_1 \to V_0 / K) = r
      \end{array}
    \right\}
  \end{equation*}
  of \(\sZ = \sZ_r\).
  By the definition of \(\sZ\) we see that if \((K, f, g) \in \sU\) then \(K = \ker{f}\).
  Moreover \(\rank(f \colon V_0/K \to V_1) = m - \delta = n\) so \(\dim{\ker{fg}} = \dim{\ker{g}} = \Delta_{fg}\).
  Therefore the image \(\Pi'(\sU)\) is given by an open set
  \begin{equation*}
    \Pi'(\sU) = \left\{ (f, g) \in Y \middle|
      \begin{array}{l}
        \dim{\ker{f}} = \delta \\
        \dim{\ker{fg}} = \Delta_{fg}
      \end{array}
    \right\}
  \end{equation*}
  and the algebraic map \((f, g) \mapsto (\ker{f}, f, g)\) is an inverse of
  the map \(\Pi'|_{\sU}\).
  The set \(\sU\) is seen to be nonempty as it contains any \(f\) and \(g\) that have a nontrivial \((n - \Delta_{fg}) \times (n - \Delta_{fg})\) matrix on the upper left and zeroes everywhere else.
  Therefore \(\Pi'|_{\sU}\) is an isomorphism and \(\Pi'\) is a birational isomorphism.
  Proposition~\ref{aux:comp-2:rat-sing} now implies that \(Y = \sY_r\) is normal, Cohen--Macaulay, and has rational singularities.
\end{proof}

\begin{cor}
  \label{simp-2:codim}
  If \(Y\) is simple \(2\)-fold compositional defined by \(\Delta_{fg}\) then we have
  \begin{equation*}
    \codim_{\sA}Y = \Delta_{fg}^2.
  \end{equation*}
\end{cor}
\begin{proof}
  If \(r = n - \Delta_{fg}\) then Theorem~\ref{simp-2:norm+rat-sing} implies that \(\sZ_r\) is a desingularization of \(Y\).
  Using the dimension calculation from Proposition~\ref{aux:comp-2:upstairs},
  \begin{equation*}
    \dim{Y} = \dim{\sZ} = 2mn - (n - r)^2 = 2mn - \Delta_{fg}^2.
  \end{equation*}
  Since \(\dim{\sA} = 2mn\), \(\codim_{\sA}{Y} = \Delta_{fg}^2\).
\end{proof}

We recall that the ideal generated by the obvious set theoretic generators given in Proposition~\ref{comp:algset-gens} need not be radical.
In particular, as Example~\ref{eg:comp:non-radical-I} shows, the ideal \(I_\Delta\) of the simple \(3\)-fold compositional variety fails to be radical.
We use the calculation above to show that the situation is better for simple \(2\)-fold compositional varieties. 

\begin{prop}
  \label{simp-2:reduced}
  Let \(\Phi^f = (x_{i,j})_{\substack{1 \leq i \leq n \\ 1 \leq j \leq m}}\) and \(\Phi^g = (y_{j, i})_{\substack{1 \leq i \leq n \\ 1 \leq j \leq m}}\) respectively be matrices of indeterminates in \(A\).
  If \(I = I_{r +1}\) is the ideal generated by the \((r+1) \times (r+1)\) minors of \(\Phi^f\Phi^g\) then \(A/I\) is reduced and Cohen--Macaulay.
\end{prop}
\begin{proof}
  Set theoretically, the zero set of \(I\) coincides with \(Y\) so \(\codim_A{I} = \codim_{\sA}{Y} = \Delta_{fg}^2\) from Corollary~\ref{simp-2:codim}.
  So the quotient ring \(A / I\) is determinantal and therefore Cohen--Macaulay by \cite[Theorem~18.18]{Eis:CommAlg}.
  Serre's criteria for normality states that \(A/I\) is reduced if and only if it satisfies properties \((\rR_0)\) and \((\rS_1)\), see \stacks{031R}.
  A Cohen--Macaulay ring satisfies \((\rS_k)\) for all \(k\) so we only need to establish \((\rR_0)\).
  From the Jacobian criteria given in \cite[Theorem~16.19 and cf. Theorem~18.15]{Eis:CommAlg} it suffices to find a regular point in \(Y\).
  If \(m_i\) is one of the generators of \(I\) and \(t_j\) is an indeterminate in \(A\) taken with respect some monomial order, we need to ensure that
  \begin{equation*}
    \rank{\Jac(I)_{(f, g)}} = \codim_AI = \Delta_{fg}^2
  \end{equation*}
  for some \((f, g) \in Y\) where
  \begin{equation*}
    \Jac_{(f,g)}
    = \Jac(I)_{(f, g)}
    = \left( \frac{\partial m_i}{\partial t_j} \right)_{(f, g)}
  \end{equation*}
  is the Jacobian localized at \((f, g)\).
  The matrices \(f, g\) from Lemma~\ref{lem:simp-2:jac-nonzero} satisfy \(\rank{fg} = r\) so we have \((f, g) \in Y\).
  Moreover all nonzero entries of \(\Jac_{(f, g)}\) are given by terms of the form \((\partial[1, \dots, r, i \mid 1, \dots, r, j]/\partial x_{i, j} )_{(f, g)}\) where \(r + 1 \leq i, j \leq n\).
  Since there are \((n - r)^2 = \Delta_{fg}^2\) such pairs of indices, \(\rank{\Jac_{(f, g)}} = \Delta_{fg}^2\) as required.
\end{proof}

If \(\varphi\) is a matrix and \(\ul a = (a_1,\dots,a_\ell)\) and \(\ul b = (b_1,\dots,b_{\ell'})\) are tuples of strictly decreasing indices then define the submatrix \(\varphi_{\ul a,\ul b} = (\varphi_{a_i, b_j})\) consisting of rows and columns indexed by \(\ul a\) and \(\ul b\) respectively.
If \(\ul a\) and \(\ul b\) have the same length then let  \([\varphi]_{\ul a, \ul b} = \det{\varphi_{\ul a, \ul b}}\) denote the minor.
When there is no ambiguity and when \(\varphi\) is clear from context, this minor will be simply denoted by \([\ul a \mid \ul b ]\).
Finally let \(\ul{a, b}\) be the tuple of pairs \(( (a_1, b_1), \dots, (a_\ell, b_\ell))\). We say \((i,j) \in \ul{a, b}\) if and only if \(i = a_k\) and \(j = b_k\)  for some \(1 \leq k \leq \ell\); in this case we set
\begin{equation*}
  \ul{a - i, b - j} =
  \left((a_1, b_1), \dots, (a_{k - 1}, b_{k - 1}), (a_{k + 1}, b_{k + 1}), \dots, (a_\ell, b_\ell)\right).
\end{equation*}

\begin{lem}
  \label{lem:simp-2:jac-nonzero}
  Let \(\Phi^f\) and \(\Phi^g\) be the usual matrices of indeterminates in \(A\).
  Let \(\Phi = \Phi^f \Phi^g\) and fix tuples \(\ul a = (a_1,\dots,a_{r + 1})\) and \(\ul b = (b_1,\dots,b_{r + 1})\).
  \begin{enumerate}
  \item
    \label{lem:jac:monomial}
    The monomial
    \begin{equation*}
      x_{\ul{i, j}} \cdot y_{\ul{i', j'}} =
      x_{i_1, j_1} \dots x_{i_{r + 1}, j_{r + 1}} y_{i'_1, j'_1} \dots y_{i'_{r + 1}, j'_{r +1}}
    \end{equation*}
    appears as a term in \([\Phi]_{\ul a, \ul b}\) if and only if \(\{\ul a\} = \{\ul i\}\), \(\{\ul b\} = \{\ul j'\}\), and \(\{\ul j\} = \{\ul i'\}\) as sets.

  \item
    \label{lem:jac:partial}
    Suppose \(t = x_{i, j}\) or \(y_{i, j}\) then
    \begin{equation*}
      \frac{\partial {x_{\ul{a, c}} \cdot y_{\ul{c', b}}}}{\partial t}  =
      \begin{cases}
        x_{\ul{a - i, c - j} } \cdot y_{\ul{c', b}} & t = x_{i, j} \text{ and } (i, j) \in \ul{a, c}\\
        x_{\ul{a, c} } \cdot y_{\ul{c' - i, b - j}} & t = y_{i, j} \text{ and } (i, j) \in \ul{c', b}\\
        0 &\text{otherwise}
      \end{cases}.
    \end{equation*}
  \item
    \label{lem:jac:fg}
    If \(f\) and \(g\) are block matrices
    \begin{equation*}
      f =
      \begin{bNiceArray}{ccc|cc|cc}
        \Hdotsfor{3}^{r} & \Hdotsfor{2}^{n - r} & \Hdotsfor{2}^{m - n} \\
        1 &   &  & \Block{3-2}{0} &  & \Block{3-2}{0} & & \Vdotsfor{3}^{r} \\
        & \Ddots[line-style=standard]  &  &                &  & & & \\
        &  & 1 &         &  &  & & \\ \midrule
        \Block{2-3}{0} &  &  & \Block{2-2}{0} &  & \Block{2-2}{0} & & \Vdotsfor{2}^{n - r} \\
        & & & & & & & \\
      \end{bNiceArray}
      \qquad,\qquad
      g =
      \begin{bNiceArray}{ccc}
        \Hdotsfor{3}^{n} \\
        1 & & & \Vdotsfor{3}^{n} \\
        & \Ddots[line-style=standard] & & \\
        & & 1 & \\ \midrule
        \Block{1-3}{0}& & & \Vdotsfor{1}^{m - n}\\
      \end{bNiceArray}
      \quad
    \end{equation*}
    then we have
    \begin{equation*}
      \left( \frac{\partial [\ul a \mid \ul b]}{\partial x_{i, j}} \right)_{(f, g)} = 1
    \end{equation*}
    if and only if \(\ul a = (1, 2, \dots, r, i)\), \(\ul b = (1, 2, \dots, r, j)\) for \(i, j \in \{r + 1, \dots, n \}\).
    All other partial derivatives vanish.
  \end{enumerate}
\end{lem}
\begin{proof}
  We recall that \(\Phi^f\) and \(\Phi^g\) are \(n \times m\) and \(m \times n\) matrices. We also denote the entries of the \(n \times n\) matrix \(\Phi\) by \((z_{i, j})\).
  Since \(z_{i, j} = \sum_{k \in [m]} x_{i,k}y_{k,j}\) as matrices we have \(\Phi_{\ul a, \ul b} = \Phi^f_{\ul a, [m]} \Phi^g_{[m], \ul a}\) so the claim now directly follows using the Cauchy--Binet formula for the determinant of a product of matrices:
  \begin{align*}
    \det \Phi_{\ul a, \ul b}
    &=
      \sum_{\ul c \subseteq [m]}
      [\Phi^f]_{\ul a, \ul c} [\Phi^g]_{\ul c, \ul b} \\
    &=
      \sum_{\ul c \subseteq [m]}
      \sum_{\substack{\sigma \in \mathfrak{S}_{r +{1}} \\ \tau \in \mathfrak{S}_{r +{1}}}}
    \sign{(\sigma \tau)}
    x_{a_1, c_{\sigma(1)}} \dots x_{a_{r + 1}, c_{\sigma(r + 1)}}
    y_{c_{\tau(1)}, b_1} \dots y_{c_{\tau(r + 1)}, b_{r + 1}} \\
    &=
      \sum_{\ul c \subseteq [m]}
      \sum_{\substack{\sigma \in \mathfrak{S}_{r +{1}} \\ \tau \in \mathfrak{S}_{r +{1}}}}
    \sign{(\sigma \tau)} x_{\ul{a, \sigma(c)}} \cdot y_{\ul{\tau(c), b}}
    =
    \sum_{\substack{\ul c \subseteq [m] \\ \sigma \in \mathfrak{S}_{r +{1}}}}
    \sign{(\sigma)} x_{\ul{a, \sigma(c)}} \cdot y_{\ul{c, b}}
  \end{align*}
  This proves \ref{lem:jac:monomial} and we immediately get \ref{lem:jac:partial} from the form of the monomial and properties of partial derivatives.
  For \ref{lem:jac:fg}, note that  if \([\ul a \mid \ul b]\) is an arbitrary minor then each monomial appearing in \((\partial[\ul a \mid \ul b]/\partial y_{i, j})\) will have \(r + 1\) indeterminates \(x_{i', j'}\).
  Since \(f\) has exactly \(r\) nonzero entries, each of the monomials will have at least one term \(x_{i',j'}\) that is set to zero and thus \(\left(\partial[\ul a | \ul b] / \partial y_{i,j}\right)_{(f,g)} = 0\) for all \(y_{i,j}\).
  Now consider the indeterminate \(x_{i, j}\) when \(1 \leq i \leq r\) or \(1 \leq j \leq r\).
  Every nonzero monomial of \((\partial[\ul a \mid \ul b] / \partial x_{i,j})\) must contain at least one term \(x_{i', j'}\) satisfying \(i' > r\) or \(j' > r\) and for this pair of indices we have \(f_{i', j'} = 0\).
  So in this case \(\left(\partial[\ul a|\ul b] / \partial x_{i,j}\right)_{(f, g)}\) vanishes.
  Finally let \(i, j \in \{r + 1, \dots, n\}\) and consider the monomial \(m = x_{\ul{a, c}} \cdot y_{\ul{c', b}}\) in \(\partial[\ul a \mid \ul b]/\partial x_{i, j}\).
  Since \(\partial m / \partial x_{i, j}\) is nonzero if and only if \((i, j) \in \ul{a, c} \), 
  its localization \((\partial m/\partial x_{i,j})_{(f, g)}\) is nonzero if and only if \(f_{\ul{a - i, c - j}} \cdot g_{\ul{c', b}} \neq 0\).
  This is only possible when \(\ul{a - i} = \ul{c - j} = (1, \dots, r)\) and, as \(\ul b\) is decreasing and \(\{\ul c\} = \{\ul{c'}\}\), \(\ul{c'} = \ul b\) we must have \(\ul c = \ul{c'} = \ul b\).
  So whenever \((\partial [\ul a | \ul b] / \partial x_{i,j})_{(f,g)}\) is nonzero, 
  \begin{equation*}
    \left( \frac{\partial [\ul a|\ul b]}{\partial x_{i, j}} \right)_{(f, g)} =
    \sign(\mathrm{id}) \cdot{f}_{\ul{a - i, c - i}} \cdot g_{\ul{c, b}} =
    f_{1, 1} \dots f_{r,r} \cdot g_{(1, 1)} \dots g_{(r, r)} g_{(j,j)} =
    1.
  \end{equation*}
  We conclude that \(\ul a = (1, \dots, r, i)\) and \(\ul b = \ul c = (1, \dots, r, j)\) as required.
\end{proof}

A free resolution of \(Y\) can now be determined using the so-called Eagon--Northcott generic perfection.
In particular, we determine the grade of \(\CC[Y]\) and then compute the free resolution as a slight modification of the resolution of a determinantal variety.

\begin{thm}
  \label{simp-2:res}
  Let \(B = \CC[z_{i,j}]_{1 \leq i, j \leq n}\), \(I'_{r + 1}\)be the ideal generated by the \((r + 1) \times (r + 1)\) minors of the \(n \times n\) matrix \((z_{i,j})\) of indeterminates, and \(M = B / I'_{r + 1}\).
  If \(\bF_\bullet\) is a \(B\)-free resolution of \(M\) then \(\CC[Y] = M \otimes B\) is resolved by \(\bF_\bullet \otimes B\).
\end{thm}
\begin{proof}
  This is a direct application of \cite[Theorem~3.5]{BV:DeterminantalRings} and Proposition~\ref{simp-2:reduced}.
  Since \(M\) is a determinantal ring it is perfect of grade \(g = \Delta_{fg}^2\).
  From Proposition~\ref{simp-2:reduced}, \(\CC[Y] = M \otimes B\).
  Since \(\CC[Y]\) is Cohen--Macaulay,
  \begin{equation*}
    \grade{\CC[Y]} = \codim_\sA{\CC[A]} = \Delta_{fg}^2.
  \end{equation*}
  The map \(z_{i, j} \mapsto \sum_k x_{i, k} y_{k, j}\), induced by matrix multiplication makes \(A\) into a \(B\)-algebra.
  As \(\CC[Y] \neq 0\), \cite[Theorem~3.5]{BV:DeterminantalRings} now implies \(\CC[Y]\) is perfect of grade \(g\) and is resolved by \(\bF_\bullet \otimes B\).
\end{proof}

\section{Preliminaries II}
\label{s:prelim-ii}

\subsection{Normal forms of a pair of maps}
\label{ss:norm-forms}
If \(V_0\) and \(V_1\) are finite dimensional complex vector spaces and \(f \colon V_0 \to V_1\) and \(g \colon V_1 \to V_0\) are linear maps, then \(V_0 \oplus V_1\) can be thought of as a \(\ZZ_2\)-graded module over \(\CC[x]\) where \(\deg{x} = 1\) and \(x\) acts on \(V_0\) by \(f\) and on \(V_1\) by \(g\).
Let \(\fA\) denote the category whose objects are tuples \((V_0, V_1, f, g)\) and morphisms determined by this \(\CC[x]\)-module structure.
The indecomposable objects in \(\fA\) determine the forms in which matrices of  \(f\) and \(g\) can be put in.

\begin{prop}[]
  \label{fg-normal-form}
  An indecomposable object in \(\fA\) is isomorphic to one of the following:
  \begin{enumerate}
  \item the object \(J_n(\lambda)\) with \(V_0 = V_1 = \CC^n\), \(f\) the identity matrix, and \(g\) a single Jordan block with eigenvalue \(\lambda\);
  \item the object \(J_n(\infty)\) with \(V_0 = V_1 = \CC^n\), \(f\) a single nilpotent Jordan block, and \(g\) the identity;
  \item \(K_n = \CC[x] / (x^{2n + 1})\)  with basis \(v_0, \dots, v_{2n}\) where \(v_i\) has parity \(i\) and \(xv_i = v_{i + 1}\) for \(i < 2n\) and \(xv_{2n} = 0\);
  \item  the object \(K_n[1]\) obtained by swapping  the even and odd parts of \(K_n\).
  \end{enumerate}
\end{prop}
\begin{proof}
  See \cite[Proposition~4.3]{ss2024:cohomology} and discussion at the beginning of \cite[\S4.2]{ss2024:cohomology}.
\end{proof}

\subsection{Splitting and factorization rings}
\label{ss:split-fact}
Here we compile basic properties of splitting and factorization rings relevant for the cohomology calculation.
Details are available in \cite[\S 3]{ss2024:cohomology} and \cite[\S 2.3]{sam2026:borel}.

Let \(A\) be a commutative ring and \(f(u) = u^N + \sum_{i = 0}^{N - 1} a_{N - i}u^i \in A[u]\) be a degree \(N\) monic polynomial.
The splitting ring \(\Split_A(f)\) is a quotient of \(A[\xi_1, \dots, \xi_N]\) obtained by identifying coefficients of the polynomials \(f(u) = (u - \xi_1) \dots (u - \xi_N)\).
If \(A\) is graded by \(\deg{a_i} = i\deg{a_1}\), then the splitting ring inherits a grading with \(\deg{\xi_i} = \deg{a_1}\) for all \(i\).
The symmetric group \(\fS_N\) acts on \(\Split_A(f)\) by permuting the \(\xi_i\) and fixing elements of \(A\).
\begin{prop}
  \label{split:facts}
  Let \(f \in A[u]\) be a degree \(N\) monic polynomial, \(\Delta\) be its discriminant, and \(B = \Split_A(f)\).
  Set
  \begin{equation*}
    V(\Delta, \partial\Delta) = \left\{ \fp \in \Spec{A} \mid \Delta \in \fp^2 A_\fp \right\}.
  \end{equation*}

  \begin{enumerate}
  \item \label{split:rank}
    \(B\) is an \(A\)-module of rank \(N!\).

  \item \label{split:syntomic+splitting}
    The map \(A \to B\) is syntomic (that is, flat, of finite presentations, and all fibers are local complete intersections).
    Furthermore, this map admits an \(A\)-linear splitting.

  \item \label{split:serre-cm}
    If \(A\) satisfies Serre's condition \((\rS_k)\) then so does \(B\).
    If \(A\) is Cohen--Macaulay, then so is \(B\).

  \item \label{split:etale}
    If \(\Delta\)is a unit of \(A\), then \(A \to B\) is étale.

  \item \label{split:red}
    If \(A\) is reduced and \(\Delta\) is a nonzerodivisor then \(B\) is reduced.

  \item  \label{split:norm-criteria}
    If \(A\) is normal, \(\Delta \in A\) is a nonzero divisor, and \(V(\Delta, \partial \Delta)\) has codimension at least \(2\) in \(\Spec(A)\), then the splitting ring \(\Split_A(f)\) is normal.
  \end{enumerate}
\end{prop}
\begin{proof}
  See \cite[Proposition~3.10]{ss2024:cohomology} for \ref{split:norm-criteria} and \cite[Proposition~3.1]{ss2024:cohomology} for the rest.
\end{proof}

Factorization rings are obtained by considering rings of invariants of \(\Split_A(f)\).
Let \(f  \in A[u]\) be as before and \(\mathbf{d} = (d_1, \dots, d_r)\) be a tuple of nonnegative integers satisfying \(\sum d_i = N = \deg{f}\).
If \(g_1, \dots, g_r\) are monic polynomials of degrees \(\deg{g_i} = d_i\), then the \(\bd\)-factorization ring \(\Fact_A^{\bd}(f)\) is obtained by considering the polynomial ring \(A[b_{ij} \mid 1 \leq i \leq r, 1 \leq j \leq d_i]\) and taking the quotient that identifies the coefficients  of the polynomials \(f(u) = g_1(u) \dots g_r(u)\).
The factorization ring \(\Fact_A^\bd(f)\) inherits a grading \(\deg{b_{i, j}} = j \deg{a_1}\) for all \(i, j\).

\begin{prop}
  \label{fact:facts}
  Let \(f \in A[u]\) be a degree \(N\) monic polynomial and suppose \(f(u) = g_1(u) \dots g_r(u)\) in the factorization ring \(C = \Fact^{\mathbf{d}}_A(f)\).

  \begin{enumerate}
  \item There is a natural isomorphism
    \begin{equation*}
      \Split_A(f) = \Split_C(g_1) \otimes_C \dots \otimes_C \Split_C(g_r)
    \end{equation*}
    of \(A\)-modules.
  \item \label{fact:quot-of-split}
    \(\Fact^{\mathbf{d}}_A(f)\) is a free \(A\)-module of rank \({N \choose d_1, \dots, d_r} = \frac{N!}{d_1! \dots d_r!}\).
  \item \label{fact:inv-ring}
    If \(N!\) is invertible then \(\Fact_A^{\bd}(f) = \Split_A(f)^G\) is the ring of \(G\)-invariants with respect to the group \(G \cong \mathfrak{S}_{d_1} \times \dots \times \mathfrak{S}_{d_r}\).
    Here \(\mathfrak{S}_{d_1}\) permutes the first \(d_1\) roots, \(\mathfrak{S}_{d_2}\) the next \(d_2\) roots, etc.
  \item \label{fact:red}
    If \(\Split_A(f)\) is reduced (respectively, integral, normal) then so is \(\Fact_A^{\bd}(f)\).
  \end{enumerate}
\end{prop}
\begin{proof}
  See \cite[Proposition~2.3]{sam2026:borel} and \cite[Proposition~3.11]{ss2024:cohomology}
\end{proof}

If \(X\) is a smooth variety and \(A\) its chow ring, then the factorization rings \(\Fact^\bd_A(f)\) can be realized as Chow rings of partial flag bundles over \(X\).
In particular, when \(X\) is a point, the Chow rings agree with singular cohomology groups of certain partial flag varieties.

\begin{prop}
  \label{iso:fact=singular-hom}
  Let \(A\) be a graded ring concentrated in degree \(0\), \(\mathbf{d} = (d_1, \dots, d_r)\) and \(d = \sum d_i\).
  There is a natural isomorphism of graded rings
  \begin{equation*}
    \Fact^{\mathbf{d}}_A(u^d)
    = A \otimes \rH^\bullet_{\mathrm{sing}}(\Flag(d_1, d_1 + d_2, \dots, d_1 + \dots + d_{r - 1};\;\CC^d), \ZZ)
  \end{equation*}
  where the generators \(b_{i, j}\) of \(\Fact^{\mathbf{d}}_A(u^n)\) have degrees given by \(\deg{b_{i, j}} = 2j\).
\end{prop}
\begin{proof}
  See \cite[Proposition~2.4]{sam2026:borel}.
  For the variant of the statement for Chow rings see \cite[Theorem~1]{gro1958:quelques} or \cite[Thm~5.1]{gss2022:computations}.
\end{proof}

We will be interested in the schemes associated to splitting and factorization rings.
If \(Y\) is an algebraic variety and \(\chi \in \sO_Y[u]\) we set
\begin{equation*}
  \sSplit_{Y}(\chi)  = \Spec(\Split_{\sO_Y}(\chi)),
  \qquad
  \sFact^{\mathbf d}_Y(\chi)  = \Spec(\Fact^{\mathbf{d}}_{\sO_Y}(\chi)).
\end{equation*}
These spaces will be called, respectively, the splitting and the \(\bd\)-factorization schemes of \(\chi\) over \(Y\).
If clear from context, we will omit \(Y\) and \(\mathbf d\) in the notation above.

\section{Grothendieck--Springer ensembles}
\label{s:ges}

In this section we calculate the Grothendieck--Springer ensemble (GSE) associated to five infinite families of flag supervarieties.
We will work with the basic diagram~\eqref{eq:basic-diag} and determine the image \(Y\) of the projection \(\ord{X} \times \gone \to \gone\)  when restricted to the subbundle \(Z\).
We will then introduce the relevant splitting and factorization schemes in the basic diagram and prove the technical properties required for our cohomology calculation.

\subsection{Compositional varieties of superflags}
\label{ss:gse-comp}
In this section we determine the compositional variety \(Y\) associated to six of the eight type classes of 2-step flags supervarieties given in Example~\ref{eg:2-step-types}.
We then describe a general procedure for extending these results to flags of arbitrary length and determine the five main cases of interest.
For the remaining two types of \(2\)-step flags, see Remark~\ref{r:missing-flags}.

\begin{prop}
  \label{type:trivial}
  If \(X \colon [\ftyp{\ast t \dots 1 0}] = [\delta \geq \delta_t \geq \dots \geq \delta_1 \geq 0]\) then
  \begin{equation*}
    Y = \Hom(V_0, V_1) \times \Hom(V_1, V_0).
  \end{equation*}
\end{prop}
\begin{proof}
  Any nonempty open set in an irreducible algebraic variety is dense.
  Since \(\Hom(V_0, V_1) \times \Hom(V_1, V_0)\) is irreducible, it suffices to show that \(\pi\) surjects onto a nonempty open subset of \(\Hom(V_0, V_1) \times \Hom(V_1, V_0)\).
  Let \(U\) be the open set
  \begin{equation*}
    U = \left\{
      (f, g) \;\middle\vert\;
      fg \text{ has } n \text{ distinct nonzero eigenvalues}
    \right\}.
  \end{equation*}
  The set \(U\) is nonempty as it contains the block matrices
  \begin{equation*}
    f =
    \begin{bNiceArray}{ccc|cc}%
      \Hdotsfor{3}^{n} & \Hdotsfor{2}^{\delta} \\
      e_1 & & & \Block{3-2}{0} & & \Vdotsfor{3}^{n} \\
      &  \Ddots[line-style=standard] & & & & \\
      & & e_n & & & \\
    \end{bNiceArray}
    \qquad,\qquad
    g =
    \begin{bNiceArray}{ccc}
      \Hdotsfor{3}^{n} \\
      \Block{3-3}{1_{n}} & & & \Vdotsfor{3}^{n} \\
      & & & \\
      & & & \\ \hline
      \Block{2-3}{0} & & & \Vdotsfor{2}^{\delta} \\
      & & &  \\
    \end{bNiceArray}
    \quad
  \end{equation*}
  for  some choice of nonzero and distinct \(e_1, \dots, e_n\).
  It is enough to show that the preimage of \((f, g) \in U\) in \(Z\) is nonempty.
  If \((f, g) \in U\) then \(g\) has full rank and so \(\dim \ker f = \delta\).
  As \(\delta_1 \leq \dots \leq \delta_t \leq \delta\), we may pick a nontrivial flag \(K_\bullet \in \Flag(\delta_1, \dots, \delta_t; \ker{f})\).
  Let \(E_1, \dots, E_n\) be the one dimensional eigenspaces corresponding to the \(n\) distinct nonzero eigenvalues of \(fg\).
  Setting \(S_i = E_1 \oplus \dots \oplus E_{q_i}\) and \(R_i = g(S_i) \oplus K_i\) we see that \((R_\bullet, S_\bullet, f, g) \in \pi^{-1}(f, g)\).
\end{proof}

\subsubsection{Simple \(1\)-fold compositional varieties}
\begin{prop}
  \label{type:d201}
  If \(X \colon [\ftyp{\ast 2 0 1}] = [\delta \geq \delta_2 \geq 0 > \delta_1]\) then \(Y\) is the determinantal variety
  \begin{equation*}
    Y = \{ (f, g) \mid \dim{\ker{g}} \geq -\delta_1 \}.
  \end{equation*}
\end{prop}

\begin{proof}
  Consider the open set
  \begin{equation*}
    U =
    \left\{
      (f, g)
      \;
      \middle|
      \begin{array}{l}
        \dim{\ker{f}} = \delta, \dim{\ker{g}} = -\delta_1 \\
        fg \text{ has } n + \delta_1 \text{ distinct eigenvalues}
      \end{array}
    \right\}.
  \end{equation*}
  The open set \(U\) is nonempty as it contains block matrices
  \begin{equation*}
    f =
    \begin{bNiceArray}{ccc|cc|cc}
      \Hdotsfor{3}^{n + \delta_1} & \Hdotsfor{2}^{-\delta_1} & \Hdotsfor{2}^{\delta} \\
      e_1 &   &  & \Block{3-2}{\ast} &  & \Block{3-2}{\ast} & & \Vdotsfor{3}^{n + \delta_1} \\
      & \Ddots[line-style=standard]  &  &                &  & & & \\
      &  & e_{n + \delta_1} &         &  &  & & \\ \midrule
      \Block{2-3}{0} &  &  & \Block{2-2}{1} &  & \Block{2-2}{\ast} & & \Vdotsfor{2}^{-\delta_1} \\
      & & & & & & & \\
    \end{bNiceArray}
    \qquad,\qquad
    g =
    \begin{bNiceArray}{ccc|cc}
      \Hdotsfor{3}^{n + \delta_1} & \Hdotsfor{2}^{-\delta_1} \\
      \Block{3-3}{1_{n + \delta_1}} &  & & \Block{3-2}{0} & & \Vdotsfor{3}^{n + \delta_1} \\
      & & & & & \\
      & & & & & \\ \midrule
      \Block{4-3}{0} & & & \Block{4-2}{0} & & \Vdotsfor{4}^{\delta - \delta_1} \\
      & & & & & \\
      & & & & & \\
      & & & & & \\
    \end{bNiceArray}
    \quad
  \end{equation*}
  for some nonzero and distinct \(e_1, \dots, e_{n + \delta_1}\).
  Let \(E_i\) denote the eigenspace of \(e_i\), \(K = \ker{f}\), and \(L = \ker{g}\).
  Our assumptions imply \(\dim{E_i} = 1\) for all \(i\), \(\dim{K} = \delta\), and  \(\dim{L} = -\delta_1\).
  Define \(S_\bullet = S_1 \subseteq S_2 \subseteq V_1\) by 
  \begin{align*}
    S_1 &= L \oplus E_1 \oplus \dots \oplus E_{p_1}, & \dim{S_1} &=  (-\delta_1) + p_1 = q_1, \\
    S_2 &= L \oplus E_1 \oplus \dots \oplus E_{q_2 + \delta_1}, & \dim{S_2} &= (-\delta_1) + (q_2 + \delta_1) = q_2.
  \end{align*}
  As \(E_i\) is an eigenspace of \(fg\) corresponding to a nonzero eigenvalue, \(K \cap g(E_i) = 0\) for all \(i\).
  In particular, \(K \cap g(S_i) = 0\), and so
  \begin{align*}
    \dim{V_0 \big/ (K \oplus g(S_2))}
    &=  m - (\delta + q_2 + \delta_1) = (n - q_2) - \delta_1,\\
    \dim{V_1 \big/ S_2}
    &= n - q_2.
  \end{align*}
  Since \(fg(S_2) \subseteq S_2\), \(f\) induces a map \(f' \colon V_0 \big/ (K \oplus g(S_2)) \to V_1 \big/ S_2\) whose kernel \(K'\) satisfies \(\dim{K'} \ge -\delta_1 > 0\).
  We can therefore pick \(K'_0 \in \Gr(-\delta_1, V_0 \big/ (K \oplus g(S_2)))\) such that  \(K'_0 = K_0 \oplus K \oplus g(S_2) \subseteq V_0\)  and \(\dim{K_0} = -\delta_1\).
  Finally, since \(\delta_2 \leq \delta\), pick \(K_1 \in \Gr(\delta_2, K)\) and define the flag \(R_\bullet = R_1 \subseteq R_2 \subseteq V_0\) by
  \begin{align*}
    R_1 &= g(S_1),  & \dim{R_1} & = q_1 + \delta_1 = p_1, \\
    R_2 &= K_0 \oplus K_1 \oplus g(S_2), & \dim{R_2} &= -\delta_1 + \delta_2 + (q_2 + \delta_1) = q_2 + \delta_2 = p_2.
  \end{align*}
  By construction we have \(f(R_\bullet) \subseteq S_\bullet\) and \(g(S_\bullet) \subseteq R_\bullet\) and so \((R_\bullet, S_\bullet, f, g) \in Z\) is in the preimage of \((f, g)\).
\end{proof}

\begin{prop}
  \label{type:d021}
  If \(X \colon [\ftyp{\ast 0 2 1}] = [\delta \geq 0 \geq \delta_2 \geq \delta_1]\) then \(Y\) is the determinantal variety
  \begin{equation*}
    Y = \{ (f, g) \mid \dim{\ker{g}} \geq -\delta_1 \}.
  \end{equation*}
\end{prop}
\begin{proof}
  Most of this case is identical to the proof of type \([\ftyp{\ast 2 0 1}]\) given in Proposition~\ref{type:d201}.
  Let \(U\), \(K = \ker{f}\),  \(L = \ker{g}\), the spaces \(E_i\)  for \(i = 1, \dots, n + \delta_1\), and \(S_\bullet\) be the same as the ones in proof of Proposition~\ref{type:d201}.
  For reference we have
  \begin{equation*}
    S_1 = L \oplus E_1 \oplus \dots \oplus E_{p_1},
    \qquad
    S_2 = L \oplus E_{1} \oplus \dots \oplus E_{q_2 + \delta_1}.
  \end{equation*}
  As \(E_i\) is an eigenspace of \(fg\), \(fg(S_2)) \subseteq S_2\) and we have an induced map \(V_0 \big/ g(S_2) \to V_1 / S_2\) whose kernel we denote by \(K'\).
  Since
  \begin{align*}
    \dim{V_0 \big/ g(S_2)} &= m - (q_2 + \delta_1) = m - q_2 - \delta_1 \\
    \dim{V_1 \big/ S_2} &= n - q_2
  \end{align*}
  and so \(\dim{V_0 \big/ g(S_2)} \geq \dim{V_1 \big/ S_2}\) as \(\delta = m - n \geq 0 \geq \delta_1\).
  Therefore,
  \begin{equation*}
    \dim{K'}
    \geq \dim{V_0 \big/ g(S_2)} - \dim{V_1 \big/S_2}
    = (m - n) - \delta_1 = \delta - \delta_1.
  \end{equation*}
  Since \(\delta \geq \delta_2 \geq \delta_1 \implies \delta - \delta_1 \geq \delta_2 - \delta_1 \geq 0 \) and we can pick \(K_0 \oplus g(S_2) \subseteq K'\) satisfying \(\dim{K_0} = \delta_2 - \delta_1\).
  We can now define
  \begin{align*}
    R_1 &=  g(S_1), & \dim{R_1} &= q_1 - q_0 = p_1, \\
    R_2 &= K_0 \oplus g(S_2), & \dim{R_1} &= (\delta_2 - \delta_1) + (q_2 - q_0) = \delta_2 - \delta_1 + q_2 + \delta_1 = p_2.
  \end{align*}
  As before, we get \((R_\bullet, S_\bullet, f, g) \in Z\) in the preimage of \(\pi\).
\end{proof}

\begin{prop}
  \label{type:d012}
  If \(X \colon [\ftyp{\ast 0 1 2}] = [\delta \geq 0 \geq \delta_1 \geq \delta_2]\) then \(Y\) is the determinantal variety
  \begin{equation*}
    Y = \left\{ (f, g) \mid \dim{\ker{g}} \geq (-\delta_2) \right\}.
  \end{equation*}
\end{prop}
\begin{proof}
  Consider the open set
  \begin{equation*}
    U = \left\{
      (f, g) \;\middle\vert
      \begin{array}{l}
        \dim{\ker{f}} = \delta, \dim{\ker{g}} = -\delta_2 \\
        fg \text{ has } n - q_0 \text{ distinct eigenvalues}
      \end{array}
    \right\}.
  \end{equation*}
  Since \(n + \delta_2 = n - q_2 + p_2 \geq 0\), the set \(U\) is well defined and it is  nonempty because it contains block matrices
  \begin{equation*}
    f =
    \begin{bNiceArray}{ccc|cc|c}
      \Hdotsfor{3}^{n + \delta_2} & \Hdotsfor{2}^{-\delta_2} & \Hdotsfor{1}^{\delta} \\
      e_1 &  & & \Block{3-2}{0} & & \Block{3-1}{0} & \Vdotsfor{3}^{n + \delta_2}\\
      & \Ddots[line-style=standard] & & & & \\
      & & e_{n + \delta_2} & & &  \\ \hline
      \Block{2-3}{0}& & & \Block{2-2}{1} & & \Block{2-1}{0} & \Vdotsfor{2}^{-\delta_2}\\
      & & & & &
    \end{bNiceArray}
    \qquad,\qquad
    g =
    \begin{bNiceArray}{ccc|cc}
      \Hdotsfor{3}^{n + \delta_2} & \Hdotsfor{2}^{-\delta_2} \\
      1 & & & \Block{3-2}{0} & & \Vdotsfor{3}^{n + \delta_2} \\
      & \Ddots[line-style=standard] & & & \\
      & & 1 & & \\ \hline
      \Block{2-3}{0} & & & \Block{2-2}{0} & &  \Vdotsfor{2}^{-\delta_2} \\
      & & & & \\ \hline
      \Block{1-3}{0} & & & \Block{1-2}{0} & & \Vdotsfor{1}^{\delta}
    \end{bNiceArray}
  \end{equation*}
  for some nonzero and distinct \(e_1, \dots, e_{n + \delta_2}\).

  Let \((f, g) \in U\) and set \(K = \ker{f}\), \(L = \ker{g}\), and \(E_i\)  the eigenspace of \(e_i\).
  Since \(0 \geq \delta_1 \geq \delta_2 \implies (-\delta_2) \geq (-\delta_1) \geq 0\), pick \(L_0 \in \Gr(-\delta_1, L)\).
  We now define the flag \(S_\bullet\) by
  \begin{align*}
    S_1 &=  L_0 \oplus E_1 \oplus \dots \oplus E_{p_1}, & \dim{S_1} &= (-\delta_1) + p_1 = q_1, \\
    S_2 &= L \oplus E_1 \oplus \dots \oplus E_{p_2}, & \dim{S_2} &= q_0 + p_2 = q_2.
  \end{align*}
  Our assumptions imply \(n + \delta_2 = (n - q_2) + p_2 \geq p_2 \geq p_1 \geq 0\) so \(S_\bullet\) is well defined.
  Finally define \(R_\bullet\) by \(R_i = g(S_i)\), \(\dim{R_i} = p_i\).
  By construction \(f(R_\bullet) \subseteq S_\bullet\), \(g(S_\bullet) \subseteq R_\bullet\) and so \((R_\bullet, S_\bullet, f, g)\) is in the preimage of \((f, g)\).
\end{proof}

\subsubsection{Simple \(2\)-fold compositional varieties}
\begin{prop}
  \label{type:d120}
  Suppose \(X \colon [\ftyp{\ast 1 2 0}]\).
  If \(\delta \geq \delta_1 \geq \delta_2 \geq 0\) then \(Y\) is the simple compositional variety
  \begin{equation*}
    Y = \left\{ (f, g) \mid \dim \ker fg \geq \delta_1 - \delta_2 \right\}.
  \end{equation*}
\end{prop}
\begin{proof}
  Let \(U \subseteq Y\) be open set consisting of morphisms
  \begin{equation*}
    U = \left\{ (f, g)
      \;\middle|
      \begin{array}{l}
        \dim \ker fg = (\delta_1 - \delta_2), \dim \ker f = \delta, \\
        fg \text{ has } n - (\delta_1 - \delta_2) \text{ distinct nonzero eigenvalues}
      \end{array}
    \right\}
  \end{equation*}
  Let \(E_1, \dots, E_{n - (\delta_1 - \delta_2)}\) be the one dimensional eigenspaces of
  \(fg\) and let \(L = \ker{fg}\).
  We note that \(L \cap E_i = 0\) for all \(i\) so define \(S_\bullet\) by 
  \begin{align*}
    S_1 &= E_1 \oplus \dots \oplus E_{q_1} & \dim{S_1} &= q_1  \\
    S_2 &= E_1 \oplus \dots \oplus E_{q_1 + (p_2 - p_1)} \oplus L & \dim{S_2} &= q_1 + (p_2 - p_1) + (\delta_1 - \delta_2) = q_2 
  \end{align*}

  As \(\delta_1 \leq \delta\) and \(g(L) \subseteq \ker{f}\) with \(\dim{g(L) \leq \dim{L} = \delta_1 - \delta_2}\), we may pick \(K \subseteq \ker{f}\) with \(\dim{K} = \delta_1\) such that \(g(L) \subseteq K \subseteq \ker{f}\) as \(\delta_1 \leq \delta = \dim{\ker{f}}\).

  By construction we have \(g(S_2) \cap K = g(L)\)  so let \(R_\bullet\) be the flag defined by
  \begin{align*}
    R_1 &= g(S_1) \oplus K & \dim{R_1} &= q_1 + \delta_1  = p_1 \\
    R_2 &= g(S_2) + (K / g(S_2)) & \dim{R_2} &= (q_1 + (p_2 - p_1) + \dim{g(L)}) + \delta_1 - \dim{g(L)} = p_2
  \end{align*}

  If \(e_1, \dots, e_{n - (\delta_1 - \delta_2)}\) are distinct and nonzero, then \(U\) contains the matrices
  \begin{equation*}
    f = 
    \begin{bNiceArray}{ccc|cc|cc}
      \Hdotsfor{3}^{n - (\delta_1 - \delta_2)} & \Hdotsfor{2}^{\delta} & \Hdotsfor{2}^{\delta_1 - \delta_2} \\
      e_1 & & & \Block{3-2}{0} & & \Block{3-2}{0} & & \Vdotsfor{3}^{n - (\delta_1 - \delta_2)} \\
      & \Ddots[line-style=standard] & & & & & & \\
      & & e_{n - (\delta_1 - \delta_2)} & & & & & \\ \hline
      \Block{2-3}{0} & & & \Block{2-2}{0} & & \Block{2-2}{1} & & \Vdotsfor{2}^{\delta_1 - \delta_2} \\
      & & & & & & & \\
    \end{bNiceArray}
    \qquad\quad,\quad
    g = 
    \begin{bNiceArray}{ccc|cc}
      \Hdotsfor{3}^{n - (\delta_1 - \delta_2)} & \Hdotsfor{2}^{\delta_1 - \delta_2} \\
      1 & & & \Block{3-2}{0} &  & \Vdotsfor{3}^{n - (\delta_1 - \delta_2)} \\
      & \Ddots[line-style=standard] & & &  \\
      & & 1 & &  \\ \hline
      \Block{5-3}{0} & & & \ddots & & \Vdotsfor{2}^{\delta_1 - \delta_2} \\
      & & & & 1 & \\ \cline{4-5}
      & & & \Block{3-2}{0} & & \Vdotsfor{3}^{\delta} \\
      & & & & & \\
      & & & & & \\ 
    \end{bNiceArray}
  \end{equation*}
  and is therefore nonempty.
  Note that an arbitrary pair of maps in \(Y\) is in a neighborhood of a pair of maps in \(U\) in the Euclidean topology: after moving to an appropriate normal form and picking bases, one only needs to ensure that \(fg\) has the appropriate number of distinct nonzero eigenvalues. 

\end{proof}

\subsubsection{Five families of flags}
Here we describe a general method of extending the results from flags of fixed length to flags of arbitrary length.
Roughly speaking, we show that any finite type can be ``extended'' by the trivial type of Proposition~\ref{type:trivial} without affecting the final compositional variety.

\begin{prop}
  \label{p:type-extension}
  Let \(X' = \Flag(p_1|q_1, \dots, p_{t'}|q_{t'};\; V')\) be a flag supervariety, \(V' = V'_0 \oplus V'_1\), and \(\gone' = \Hom(V'_0, V'_1) \times \Hom(V'_1, V'_0)\).
  Let \(Z'\) be the subbundle 
  \begin{equation*}
    Z'
    = \left\{ (R_\bullet, S_\bullet, f, g) \;\vert\; f(R_\bullet) \subseteq S_\bullet, g(S_\bullet) \subseteq R_\bullet \right\}
    \subseteq \ord{X'} \times \gone', 
  \end{equation*}
  and \(\pi' \colon \ord{X'} \times \gone' \to \gone'\) the projection.
  Consider 
  \begin{equation*}
    X = \Flag(p_1|q_1, \dots, p_{t'}|q_{t'}, p_{t' + 1}|q_{t' + 1}, \dots, p_{t}|q_{t};\;V)
  \end{equation*}
  where \(V = V_0 \oplus V_1\) is a vector superspace satisfying \(p_{t' + 1} = \dim{V'_0}\), \(q_{t' + 1} = \dim{V'_1}\), and 
  \begin{equation*}
    \sdim{V'} = p_{t' + 1} - q_{t' + 1} \leq \dots \leq p_t - q_t \leq \sdim{V}
  \end{equation*}
  Let \(\gone\), \(Z\), \(\pi\),  respectively, be the analogues of \(\gone'\), \(Z'\), and \(\pi'\).
  If \(\pi'(Z') \subseteq \gone'\) is a compositional variety determined by parameters \(\Delta_\bullet\) then \(\pi(Z) \subseteq \gone\) is also compositional and determined by the same set of parameters. 
\end{prop}
\begin{proof}
  Let \(Y \subseteq \gone\) be the compositional variety defined by \(\Delta_\bullet\).
  It is enough to show that \(\pi \colon Z \to Y\) is surjective.
  Pick \((f, g) \in Y\).
  Let \(X'' = \Flag(p_{t' + 1}|q_{t' + 1}, \dots, p_{t}|q_{t};\;V)\).
  Since \(p_{t' + 1} - q_{t' + 1} \leq \dots \leq p_t - q_t \leq \sdim(V)\), \(X''\) is of trivial type and by Proposition~\ref{type:trivial}, there exists a pair \((R_{t' + 1} \subseteq \dots \subseteq R_{t} \subseteq V_0, S_{t' + 1} \subseteq \dots \subseteq S_t \subseteq V_1)\) of flags satisfying \(\dim{R_i} = p_i\), \(\dim{S_i} = q_i\), \(f(R_i) \subseteq S_i\), and \(g(S_i) \subseteq R_i\) for all \(t' + 1 \leq i \leq t\).
  
  Let \(f'\) and \(g'\) denote, respectively, the restrictions of \(f\) and \(g\) to \(R_{t' + 1}\) and \(S_{t' + 1}\).
  Since \(p_{t' + 1} = \dim{R_{t' + 1}} = \dim{V_0'}\) and \(q_{t' + 1} = \dim{S_{t' + 1}} = \dim{V_1'}\), \((f', g') \in \gone'\).
  Since \(\Delta_{\cdots f} \leq \dim{V_0'} = p_{t' + 1} \leq \dim{V_0}\) and \(\Delta_{\cdots g} \leq \dim{V_1'} = q_{t' + 1} \leq \dim{V_1}\), following construction of the flag in proof of Proposition~\ref{type:trivial}, we may assume that the restriction of the appropriate kernels to \(R_{t' + 1}\) and \(S_{t' + 1}\) satisfy the required lower bounds.
  Kernels of various compositions of \(f'\) and \(g'\) satisfy the conditions specified by \(\Delta_\bullet\) and thus, using the identifications \(V' = R_{t' + 1} \oplus S_{t' + 1}\) and \(\gone' = \Hom(R_{t' + 1}, S_{t' + 1}) \times \Hom(S_{t' + 1}, R_{t' + 1})\), we have \((f', g') \in \pi'(Z')\).
  Since \(\pi' \colon Z' \to \pi'(Z)\) is surjective, pick \((R_1 \subseteq \dots \subseteq R_{t' + 1}, S_1 \subseteq \dots \subseteq S_{t' + 1}, f', g')\) in the preimage of \((f', g')\).
  We conclude that \((R_1 \subseteq \dots \subseteq R_t \subseteq V_0, S_1 \subseteq \dots \subseteq S_t \subseteq V_1, f, g)\) is in the preimage \(\pi^{-1}(f, g)\) and \(\pi \colon Z \to Y\)  is surjective.
\end{proof}

\begin{cor}
  \label{five-families}
  Let \(X \colon \tau\)  be a partial flag supervariety of length \(t\).
  Let \(Y\) be the restriction of the projection \(\pi \colon \ord{X} \times \gone \to \gone\) to the subbundle \(Z\).
  If \(\tau\) belongs to one of the five families listed below, then \(Y\) is compositional and its defining parameters depend on \(\tau\) in the following way:
  \begin{align}
    \tag{\(\sT\)} \label{fam:T}
    \tau =  [\ftyp{\ast t \dots 3 2 1 0}] &\implies Y = \gone,
    \\
    \tag{\(\sD_1\)} \label{fam:D1}
    \tau = [\ftyp{\ast t \dots 3 2 0 1}] &\implies Y = \{ (f, g) \mid \dim{\ker{g}} \geq - \delta_1 \},
    \\
    \tag{\(\sD_2\)} \label{fam:D2}
    \tau = [\ftyp{\ast t \dots 3 0 2 1}] &\implies Y = \{ (f, g) \mid \dim{\ker{g}} \geq - \delta_1 \},
    \\
    \tag{\(\sD_3\)} \label{fam:D3}
    \tau = [\ftyp{\ast t \dots 3 0 1 2}] &\implies Y = \{ (f, g) \mid \dim{\ker{g}} \geq - \delta_2 \},
    \\
    \tag{\(\sC\)} \label{fam:C}
    \tau = [\ftyp{\ast t \dots 3 1 2 0}] &\implies Y = \{ (f, g) \mid \dim{\ker{fg}} \geq \delta_1 - \delta_2 \}.
  \end{align}
\end{cor}
\begin{proof}
  First statement is Proposition~\ref{type:trivial} the rest follow by using Proposition~\ref{p:type-extension} with Proposition~\ref{type:d201}, Proposition~\ref{type:d021}, Proposition~\ref{type:d012}, and Proposition~\ref{type:d120} respectively.
\end{proof}

\begin{rem}
  \label{rem:stability}
  We suspect that the assumption that \(Y\)  is compositional can be removed from Proposition~\ref{p:type-extension} using techniques from representation stability.
  If \(\gone' = \Hom(V_0', V_1') \times \Hom(V_1', V_0')\) then, as \(Z'\) is equivariant with respect to the group \(\GL(V_0') \times \GL(V_1')\), the space of generators of \(\pi(Z')\) given by \(M' = \Tor_1^{A'}(\sO_{\pi(Z')}, \bC)_\bullet\) is naturally a rational \(\GL(V_0') \times \GL(V_1')\)-representation.
  By varying \(X\) one obtains inclusions \(\GL(V_0') \times \GL(V_1') \hookrightarrow \GL(V_0) \times \GL(V_1)\) and a sequence of representations \(M' \to M\); it will be interesting to know if this sequence is representation stable in the sense of \cite{ss2016:grobner, ps2017:representation}.
\end{rem}

\begin{rem}[The ``missing'' 2-step flags]
  \label{r:missing-flags}
  In Corollary~\ref{five-families} above, we only list the types for which we are able to completely determine cohomology.
  Note that all the families arise from five of the eight types of two step flags given in Example~\ref{eg:2-step-types}.
  Of the three remaining types, \([\ftyp{\ast 1 0 2}]\) results in a compound compositional variety as we show in Proposition~\ref{type:d102} below.
  In the types \([\ftyp{2 \ast 0 1}]\) and \([\ftyp{1 \ast 0 2}]\), we haven't identified all the defining parameters.
  However we suspect that the resulting compositional varieties will be given by \(\Delta_f  = \delta_2\), \(\Delta_g = -\delta_1\) in type \([\ftyp{2 \ast 0 1}]\) and \(\Delta_{f} = \delta_1\), \(\Delta_g = - \delta_2\) in type \([\ftyp{1 \ast 0 2}]\).
  Additionally, these cases might require a further restriction on \(\Delta_{fg}\) or \(\Delta_{gf}\).
  In all these missing cases, the resulting compositional variety is not simple and currently we don't have a good description of their \(\Tor\)-groups.
  Since the latter appear in the cohomology spectral sequence in Theorem~\ref{t:super-geom-method}, lacking a good description of the syzygies, we are unable to show the spectral sequence degenerates.
\end{rem}

\begin{prop}
  \label{type:d102}
  If \(X \colon [\ftyp{\ast 1 0 2}] = [\delta \geq \delta_1 \geq 0 > \delta_2]\) then \(Y\) is the compositional variety
  \begin{equation*}
    Y =  \left\{ (f, g) \mid \dim{\ker{g}} \geq (-\delta_2), \dim{\ker{fg}} \geq \delta_1 - \delta_2 \right\}.
  \end{equation*}
\end{prop}
\begin{proof}
  Consider the dense open set
  \begin{equation*}
    U = \left\{ (f, g) \;\middle\vert
      \begin{array}{l}
        \dim{\ker{f}} = \delta, \dim{\ker{g}} = -\delta_2, \\
        \dim{\ker{fg}} = \delta_1 - \delta_2, \\
        fg \text{ has } n - (\delta_1 - \delta_2) \text{ distinct eigenvalues}
      \end{array}
    \right\}.
  \end{equation*}
  Since \(n - (\delta_1 - \delta_2) = (n - q_2) + q_1 - (p_2 - p_1) \geq 0 \), \(U\) is well defined.
  Let \(f, g \in Y\), \(e_1, \dots, e_{n - (\delta_1 - \delta_2)}\) be the distinct eigenvalues of \(fg\), \(E_i\) the eigenspace of \(e_i\), \(K_0 = \ker{f}\), \(L_0 = \ker{g}\), \(L_1 = \ker{fg}\).
  Let \(S_\bullet\) be the flag
  \begin{align*}
    S_1 &= E_1 \oplus \dots \oplus E_{q_1} & \dim{S_1} &= q_1, \\
    S_2 &= E_1 \oplus \dots \oplus E_{p_2 - \delta_1} \oplus L_1 & \dim{S_1} &= (p_2 - \delta_1) + (\delta_1 - \delta_2) = q_2.
  \end{align*}
  Since \(\dim{L_1} = \delta_1 - \delta_2 \geq \dim{L_0} = - \delta_2\), we must have \(\dim{g(L_1) \cap K_0} = \delta_1\).
  Letting \(g(L_1) \cap K_0 = K_1\), define flag \(R_\bullet\) by
  \begin{align*}
    R_1 &= g(S_1) \oplus K_1 & \dim{R_1} &= q_1 + \delta_1 = p_1, \\
    R_2 &= g(S_2) & \dim{R_2} &=  (p_2 - \delta_1) + \delta_1 = p_2.
  \end{align*}
  Since \(K_1 \subseteq f(L_1)\), by construction we get \(f(R_\bullet) \subseteq S_\bullet\) and \(g(S_\bullet) \subseteq R_\bullet\).
  We conclude that \((R_\bullet, S_\bullet, f, g)\) is in the preimage of \((f, g)\).
\end{proof}

Based on calculations in the \(2\)-step case and the extension result, we suspect that \(Y\) is compositional for all partial flag supervarieties.
Suppose \(X = \Flag(p_1|q_1, \dots, p_t|q_t;\;V)\).
Pick \((R_\bullet, S_\bullet, f, g) \in Z\) and let \(f_i\) and \(g_i\) denote respectively the restrictions of \(f\) and \(g\) to \(R_i\) and \(S_i\).
Since \(f(R_i) \subseteq S_i\) and \(g(S_i) \subseteq R_i\), we have \((f_i, g_i) \in \Hom(R_i, S_i) \times \Hom(S_i, R_i)\).
Thus \(\dim{\ker{f_i}} \geq \delta_i\) and \(\dim{\ker{g_i}} \geq -\delta_i\).
Since \(\ker{f_i} \subseteq \ker{f}\) and \(\ker{g_i} \subseteq \ker{g}\), \(f\) and \(g\) must satisfy
\begin{equation*}
  \dim{\ker{f}} \geq \max\{0, \delta_1, \dots, \delta_t \},
  \qquad
  \dim{\ker{g}} \geq \max\{0, -\delta_1, \dots, -\delta_t \}.
\end{equation*}
A similar analysis goes through for arbitrary compositions of \(f\) and \(g\).
Since \(f_i g_i\) induces a map \(S_i / S_{i - 1} \to S_i / S_{i - 1}\) that factors through \(R_i / R_{i - 1}\); thus we have
\begin{equation*}
  \rank{f_i g_i} \leq \min(p_i - p_{i - 1}, q_i - q_{i - 1})
\end{equation*}
and by writing \(V_0 = \bigoplus R_i / R_{i - 1}\) and \(V_1 = \bigoplus S_i / S_{i - 1}\),  we obtain the bound
\begin{equation*}
  \dim{\ker{fg}} \geq  n - \sum \min(p_i - p_{i - 1}, q_i - q_{i -1}).
\end{equation*}
Continuing in this manner, it is possible that these rank conditions completely determine \(Y\).

\begin{conj}
  \label{conj:Y-always-comp}
  If \(X = \Flag(p_1|q_1, \dots, p_t|q_t;\;V)\) is a flag supervariety of type \(\tau\) then \(Y\) is a compositional variety whose defining parameters \(\Delta_\bullet\) are given by \(\Delta_h = \sum c^h_i p_i + d^h_i q_i\) where \(c^h_i, d^h_i \in \bZ\) only depend \(\tau\) and are independent of the \(p_i\) and \(q_i\).
\end{conj}

\subsection{Splitting rings and their spectra}
\label{ss:gse-split}
Here we study properties of splitting rings over simple \(1\)-fold and simple \(2\)-fold compositional varieties.
The simple \(1\)-fold case is effectively the case of the determinantal varieties and appears in the cohomology calculation of the structure sheaf of the supergrassmannian in \cite{ss2024:cohomology}.
For reference, we restate the result below in our notation. 

\begin{thm}
  \label{simp-1:split-integ-ratsing}
  Let \(\dim{V_0} \geq \dim{V_1}\) and let \(Y \subseteq \gone\) be the simple \(1\)-fold compositional (that is, determinantal) variety defined by \(\Delta_g\). 
  Let \(f \colon V_0 \otimes  \sO_Y \to V_1 \otimes \sO_Y\) and \(g \colon V_0 \otimes \sO_Y \to V_1 \otimes \sO_Y\) be the universal linear maps, \(\chi(u) \in \sO_Y[u]\) the characteristic polynomial of \(fg\), and \(\bar \chi(u) = u^{- \Delta_{g}} \chi(u)\)
  The splitting scheme \(\sSplit_{Y}(\bar\chi)\) is integral and has rational singularities.
\end{thm}
\begin{proof}
  By \cite[Theorem~5.4]{ss2024:cohomology} spectra of splitting rings over determinantal varieties are integral and have rational singularities.
\end{proof}

In the rest of this subsection, we prove the analogous result for the simple \(2\)-fold compositional variety.
From Corollary~\ref{five-families}, the only case of interest for us are flag supervarieties \(X\) of type \([\ftyp{\ast t \dots 3 1 2 0}]\).
So we let \(Y\) be the simple \(2\)-fold compositional variety determined by \(\Delta_{fg} = \delta_1 - \delta_2\).
We first prove the results for the \(t = 2\) case and fix \(X \colon [\ftyp{\ast 1 2 0}]\).
Let \(\sZ\) be the total space of a subbundle of a trivial bundle over a product 
\begin{equation*}
  \Flag(\delta_1 + 1, \dots, \delta_1 + q_1, \delta_1 + q_1, \dots n + \delta_2;\; V_0)
  \times
  \Flag(1, \dots, q_1, q_1 + (\delta_1 - \delta_2), \dots, n;\; V_1)
\end{equation*}
of partial flag varieties with fibers isomorphic to \(\gone\).
The \(T\)-points of \(\sZ\) are tuples \((R_\bullet, S_\bullet, f, g)\) where 
\begin{align*}
  R_\bullet &= R_0^{\delta_1 + 1} \subseteq \dots \subseteq R_{q_1 - 1}^{\delta_1 + q_1} \subseteq R^{\delta_1 + q_1}_{q_1} \subseteq \dots \subseteq R_{n - (\delta_1 - \delta_2)}^{n + \delta_2} \subseteq (V_0)_T, \\
  S_\bullet &= S_0^1 \subseteq \dots \subseteq S_{q_1 - 1}^{q_1} \subseteq S_{q_1}^{q_1 + (\delta_1 - \delta_2)} \subseteq \dots \subseteq S_{n - (\delta_1 - \delta_2)}^n \subseteq (V_1)_T
\end{align*}
are flags of \(T\)-submodules with ranks prescribed by the superscripts and  \(f \colon (V_0)_T \to (V_1)_T\), \(g \colon (V_1)_T \to (V_0)_T\) are \(T\)-module morphisms satisfying \(f(R_\bullet) \subseteq S_\bullet\) and \(g(S_\bullet) \subseteq R_\bullet\).

This results in a natural surjection \(\sZ \to Z\) given by 
\begin{equation*}
  (R_\bullet, S_\bullet, f, g) \mapsto
  (R_{q_1 - 1}^{p_1} \subseteq R_{p_2 - p_1 + q_1}^{p_2} \subseteq V_0,
  S_{q_1 - 1}^{q_1} \subseteq S_{p_2 - p_1 + q_1}^{q_2} \subseteq V_1, f, g).
\end{equation*}
Let \(f \colon V_0 \otimes \sO_Y \to V_1 \otimes \sO_Y\) and \(g \colon V_1 \otimes \sO_Y \to V_0 \otimes \sO_Y\) be the universal linear maps, \(\chi\)  the characteristic polynomial of \(fg\), and \(\bar\chi = u^{-\Delta_{fg}} \chi\) the reduced characteristic polynomial.
Let \(\Phi \colon \sZ \to \sSplit_{Y}(\bar\chi)\) be the morphism sending \((R_\bullet, S_\bullet, f, g)\) to \((f, g, \prod_{i = 0, i \neq q_1}^{n  - \Delta_{fg}}(u - \lambda_i))\) where \(\lambda_i\) is the eigenvalue of \(fg\) on \(S_i / S_{i - 1}\).
We work with the extension 
\begin{equation}
  \label{diag:split:simp-2}
  \begin{tikzcd}
    \sZ
    \ar["\Phi"]{d}
    \ar[two heads]{r}
    &
    Z
    \ar["\pi' = \pi\big\vert_Z", two heads]{d}
    \ar[hook]{r}
    &
    \ord{X} \times \gone
    \ar["\pi", two heads]{d}
    \\
    \sSplit_{Y}(\bar\chi)
    \ar["\widetilde \Phi"]{r}
    &
    Y
    \ar[hook]{r}
    &
    \gone
  \end{tikzcd}
\end{equation}
of the basic diagram of the geometric method.

\begin{prop}
  \label{simp-2:split:maps-well-def}
  The morphisms \(\Phi\) and \(\widetilde \Phi\) in \eqref{diag:split:simp-2} are well defined and surjective.
\end{prop}
\begin{proof}
  Since \(\widetilde\Phi\) is induced from the ring map used in the construction of the splitting ring, it is automatically well defined whenever \(\bar\chi\) is.
  Since \(\dim{\ker{fg}} \geq \Delta_{fg}\), \(u^{\Delta_{fg}}\) divides the characteristic polynomial \(\chi\) of \(fg\) and so \(\bar\chi = u^{-\Delta_{fg}} \chi\)  is well defined.
  Since any such \(\bar\chi\) splits into linear terms, \(\widetilde\Phi\) is surjective.

  If \(i \neq q_1\) then \(\dim{S_i / S_{i - 1}} = 1\) so the characteristic polynomial of \(fg\) restricted to \(S_i / S_{i - 1}\) is linear and of the form \((u - \lambda_i)\).
  The product \(\prod_{i = 0, i \neq q_1}^{n - (\delta_1 - \delta_2)}(u - \lambda_i)\) is a factor of the characteristic polynomial \(\chi\) of \(fg\) on \(V_1\) of degree \(n - (\delta_1 - \delta_2)\). 
  Since \(\sZ \to Z\) and \(\pi'\) are both surjective, \((f, g) \in Y\).
  Thus \((f, g, \prod_{i = 0, i \neq q_1}^{n - (\delta_1 - \delta_2)} (u - \lambda_i)) \in \sSplit_{Y}(\bar\chi)\) and \(\Phi\) is well defined.
  To show \(\Phi\) surjective, it is enough to work over the open set in \(\sSplit_Y(\bar\chi)\) where \(\dim{\ker{fg}} = \delta_1 - \delta_2\) and \(\dim{\ker{f}} = \delta\) (cf. Proposition~\ref{type:d120}).
  Let \(E_i\) be the one dimensional eigenspace of \(\lambda_i\), \(K \subseteq \ker{f}\) a \(\delta_1\)-dimensional subspace, and \(L = \ker{fg}\).
  If \(1 \leq i < q_1\) then let \(S_{i - 1} = E_1 \oplus \dots \oplus E_i\) and if \(i \geq q_1\) let \(S_i = L \oplus E_1 \oplus \dots \oplus E_i\).
  If \(R_i = K \oplus g(S_i)\) then we see that \((f, g, R_\bullet, S_\bullet) \in \sZ\) and conclude that \(\Phi\) is surjective.
\end{proof}

\begin{prop}
  \label{simp-2:split:reduced}
  The scheme \(\sSplit_Y(\bar\chi)\) is reduced.
\end{prop}
\begin{proof}
  Let \(\Delta \in \sO_Y\) be the discriminant of \(\bar\chi\).
  The maps \(f \colon V_0 \to V_1\) and \(g \colon V_1 \to V_0\) given in proof of Proposition~\ref{type:d120} determine a \(\CC\)-point of \(Y\) where the characteristic polynomial of \(fg\) is given by \(u^{\delta_1 - \delta_2} \prod_{i = 1}^{n - \delta_1 + \delta_2}(u - e_i)\)  for distinct and nonzero \(e_1, \dots, e_{n - \delta_1 + \delta_2}\).
  Therefore \(\Delta \neq 0\) at \((f, g)\).
  Since \(Y\) is the surjective image of a vector bundle \(Z\) over a projective variety \(\ord{X}\), it is irreducible.
  Form Proposition~\ref{simp-2:reduced} we see that \(Y\) is reduced.
  The claim now follows from Proposition~\ref{split:facts}\ref{split:red} by noting that \(\Delta \neq 0\) and \(Y\) is integral.
\end{proof}

\begin{prop}
  \label{simp-2:normal-form}
  If \((f, g)\) is a \(\bC\)-point of \(Y\) then we can choose bases of \(V_0\) and \(V_1\) such that the matrices \(f\) and \(g\) have the form
  \begin{equation*}
    f =
    \begin{bmatrix}
      F & 0 & \ast \\
      0 & 0 & \ast
    \end{bmatrix},
    \quad
    g =
    \begin{bmatrix}
      G & \ast \\
      0 & \ast \\
      0 & 0
    \end{bmatrix}.
  \end{equation*}
  Where \(F\)  and \(G\) are upper triangular \((n - \Delta_{fg}) \times (n - \Delta_{fg})\) matrices and the matrices denoted by \(\ast\) are arbitrary and have \(\delta\)-columns in \(f\) and \(\Delta_{fg}\)-columns in \(g\).
\end{prop}
\begin{proof}
  We note that in our case of type \([\ftyp{\ast 1 2 0}]\) we have \(\delta \geq \delta_1 \geq \delta_2 \geq 0\) and \(\Delta_{fg} = \delta_1 - \delta_2\).
  Since \(\dim{\ker{fg}} \geq \Delta_{fg}\), pick \(L \subseteq \ker{fg}\) such that \(\dim{L} = \Delta_{fg}\).
  Now pick a basis \(y_1, \dots, y_n\) of \(V_1\) where the last \(\Delta_{fg}\) vectors \(y_{n - \Delta_{fg} + 1}, \dots, y_n\) form a basis of \(L\).
  Note that \(\rank{gf} \leq n - \Delta_{fg}\) so \(\dim{\ker{gf}} \geq m - n + \Delta_{fg} = \delta + \Delta_{fg}\)
  Since \(\dim{\ker{f}} \geq \delta\), we can pick \(K_0 \subseteq \ker{f}\) of dimension \(\delta\) and \(K_1\) of dimension \(\Delta_{fg}\), such that \(K_0 \oplus K_1 \subseteq \ker{gf}\).
  Now pick a basis \(x_1, \dots, x_m\) of \(V_0\) such that the last \(\delta + \Delta_{fg}\) vectors form a basis of \(K_0 \oplus K_1\) and the last \(\Delta_{fg}\) vectors form a basis of \(K_1\).
  By construction of \(L\), \(K_0\), \(K_1\), and \(g(L) \subseteq \ker{f}\) we see that the matrices have the required form but don't necessarily satisfy the requirement that \(F\) and \(G\) are upper triangular.
  Let \(W_0 \subseteq V_0\)  be the span of \(x_1, \dots, x_{n - \Delta_{fg}}\) and \(W_1 \subseteq V_1\) be the span of \(y_1, \dots, y_{n - \Delta_{fg}}\).
  Since \((W_0, W_1, F, G)\) is an object in the category \(\fA\) of Proposition~\ref{fg-normal-form}, we can make \(F\) and \(G\) upper triangular by writing in terms of indecomposables and picking bases.
  The claim now follows.
\end{proof}

\begin{prop}
  \label{simp-2:split:normal}
  The variety \(\sSplit_Y(\bar\chi)\) is normal.
\end{prop}
\begin{proof}
  The proof of this is exactly the same as proof of normality in the determinantal case given in \cite[Proposition~5.11]{ss2024:cohomology}
  So we summarize the main parts of the proof and give references to \cite{ss2024:cohomology} whenever necessary.

  Suppose \((f, g)\) is written in terms of the normal forms in Proposition~\ref{simp-2:normal-form}.
  Let \(\Delta\) be the discriminant of the reduced characteristic polynomial \(\bar\chi\).
  As in \cite[Propostiion~5.9]{ss2024:cohomology}, if \((f, g) \in V(\Delta)\) then we may assume that the first two diagonal entries of \(FG\) are equal.
  If \(\lambda\) is a nonzero repeated root of \(\bar\chi\), then in the decomposition determined by Proposition~\ref{fg-normal-form}, we may assume that the indecomposable object \(J_{n'}(\lambda)\) appears first in the construction of the normal form of the pair \((F, G)\).
  In this case, the first two diagonal entries of \(FG\) are \(\lambda\).
  If not, then the repeated root is \(0\) and we can pick any of the indecomposables \(J_{n'}(0)\), \(J_{n'}(\infty)\), \(K_{n'}\), or \(K_{n'}[-1]\) first to make the first two diagonal entries of \(FG\) zero.

  The closed set \(V(\Delta) \subseteq Y\) is irreducible by the argument given in \cite[Propostiion~5.10]{ss2024:cohomology}; the only change is to work with the form of the matrices coming from Proposition~\ref{simp-2:normal-form}.
  Using the same form of matrices, the proof of \(\codim(V(\Delta, \partial \Delta)) \geq 2\) in \(Y\) follows directly from  \cite[Proposition~5.11]{ss2024:cohomology}.
  Since \(Y\) is normal by Theorem~\ref{simp-2:norm+rat-sing} and \(\codim{V(\Delta, \partial\Delta)} \geq 2\) in \(Y\), normality of \(\sSplit_Y(\bar\chi)\) follows from Proposition~\ref{split:facts}\ref{split:norm-criteria}.
\end{proof}

The proof strategy for rational singularities is very similar to the ones used in \cite[\S5.6]{ss2024:cohomology} and \cite[\S5.3]{sam2026:borel}.
Whenever possible, in this subsection, we try to keep our notation consistent with Sam--Snowden and Sam's. 
Consider the following closed sets of \(Y\):
\begin{itemize}
\item \(D_1\) the locus where \(\bar\chi(0) = 0\);
\item \(D_2\) the locus where \(\bar\chi(u)\) has a repeated root;
\item \(D_3 \subseteq D_2\) the locus where \(\bar\chi(u)\) has a triple root, two repeated roots, or a unique repeated root whose corresponding Jordan block of \(fg\) is a scalar;
\item \(D_4 = D_1 \cap D_2\);
\item \(D_5 = D_1 \cup D_3 \).
\end{itemize}
Let \(U_i = Y \setminus D_i\) for all \(i\).
For any space \(W\) over \(Y\), we will let \(D_i(W)\) and \(U_i(W)\) denote the preimages of the \(D_i\) and \(U_i\) in \(W\).

\begin{prop}
  \label{simp-2:aux:D34-codim}
  If \(i = 3, 4\) then  \(\codim_{\sZ} (D_i(\sZ)) \geq 2\).
\end{prop}
\begin{proof}
  Since \(\sZ\) is an equivariant vector bundle, it is enough to show that over any given fiber, the restrictions have the specified codimension.
  From Proposition~\ref{simp-2:normal-form}, by picking bases of \(V_0\) and \(V_1\) we  can put any \((f, g) \in Y\) in the form
  \begin{equation*}
    f =
    \begin{bmatrix}
      F & 0 & \ast \\
      0 & 0 & \ast
    \end{bmatrix},
    \quad
    g =
    \begin{bmatrix}
      G & \ast \\
      0 & \ast  \\
      0 & 0
    \end{bmatrix}
  \end{equation*}
  where \(F\), \(G\) are upper triangular \((n - \Delta_{fg}) \times (n - \Delta_{fg})\) matrices (cf. matrices given in proof of Proposition~\ref{type:d120}).

  We see that the characteristic polynomial of \(fg\) is determined by the product \(FG\).
  Let \(x_1, \dots, x_{n - \Delta_{fg}}\) and \(y_1, \dots, y_{n - \Delta_{fg}}\) denote, respectively, the diagonal entries of \(F\) and \(G\).

  A repeated root corresponds to one component for the fiber for each equation \(x_iy_i = x_j y_j\) for \(i < j\) and \(D_4(\sZ)\) is cut out in this component by \(\prod_i x_iy_i\).
  We see that \(D_4\) has codimension \(2\).
  Similarly, a triple root imposes two equations \(x_iy_i = x_j y_j\) and \(x_ky_k = x_\ell y_\ell\) and so has codimension \(2\).
  Finally if there is a unique repeated root where the Jordan block  of \(fg\) is a scalar imposes a nontrivial condition on one of the irreducible components and thus has codimension \(\geq 2\).
\end{proof}

\begin{prop}
  \label{simp-2:split:birational}
  The morphism \(\Phi\) is birational.
\end{prop}
\begin{proof}
  We will show that \(\Phi \colon U_5(\sZ) \to U_5(\sSplit_Y(\bar\chi))\) is an isomorphism.
  Pick \((f, g, p) \in U_5(\sSplit_Y(\bar\chi))\) where \(p = \prod_{i = 1}^{n - \Delta_{fg}} (u - \lambda_i)\) and \((R_\bullet, S_\bullet, f, g)\) is in the preimage of \(\Phi\). 
  We have \((f, g) \in U_5\), so \(\lambda_i \neq 0\) for all \(i\) and \(L = \ker{fg}\) is \(\Delta_{fg}\)-dimensional.
  We may reindex the product by \(p = \prod_{i = 0, i \neq q_1}^{n - \Delta_{fg}} (u - \lambda_i)\) and then we see that \(S_i\) must be the sum of the eigenspace \(\lambda_1, \dots, \lambda_i\) when \(1 \leq i < q_1\).
  First suppose all \(\lambda_i\) are unique.
  Since \(g(S_{q_1}) \subseteq R_{q_1} = R_{q_1 - 1}\), and \(f(R_{q_1 - 1}) \subseteq S_{q_1 - 1}\), the induced map \(fg \colon (S_{q_1} / S_{q_1 - 1}) \to (S_{q_1} / S_{q_1 - 1})\) has a \(\Delta_{fg}\) dimensional kernel so if \(i \geq q_1\), \(S_i\) is the span of \(L\) and the eigenspaces \(\lambda_1, \dots, \lambda_i\).
  We have \(\rank{fg} = n - (\delta_1 - \delta_2) = \rank{gf}\).
  If \(K = \ker{gf}\), then  \(\dim{K} = m - n + \delta_1 - \delta_2 = \delta + \delta_1 - \delta_2\).
  As \(\dim{V_1} = \dim{S_{n - (\delta_1 - \delta_2)}} = n\), \(gf\) induces the zero map on \(V_0 / R_{n - (\delta_1 - \delta_2)}\).
  Since \(\dim{R_{n - (\delta_1 - \delta_2)}} = n + \delta_2\), \(V_0 = K_0 \oplus R_{n - (\delta_1 - \delta_2)}\) where \(\dim{K_0} = m - n - \delta_2 =  \delta - \delta_2\), so \(K = K_0 \oplus K_1\) where \(\dim{K_1} = \delta_1\).
  Similar to the \(S_i\), we see that the \(R_i\) are defined by \(K_1\) and the eigenvectors of the \(\lambda_j\).
  So in this case \((R_\bullet, S_\bullet, f, g)\) is unique.

  Now suppose there is only one repeated eigenvalue.
  Since the Jordan block is not a scalar \(R_i\) and \(S_i\) can again be uniquely defined; for the repeated eigenvalue, we simply pick the eigenvector first followed by the generalized eigenvector.

  We have established an isomorphism of open sets and conclude that \(\Phi\) is birational.
\end{proof}

\begin{prop}
  \label{simp-2:split:ratsing}
  The variety \(\sSplit_Y(\bar\chi)\) has rational singularities.
\end{prop}
\begin{proof}
  We showed in Proposition~\ref{simp-2:split:birational} that \(U_5(\sSplit_Y(\bar\chi))\) is smooth and has rational singularities.
  Recall that Theorem~\ref{simp-2:norm+rat-sing} implies \(Y\) has rational singularities.
  Over \(U_2\), \(\bar\chi\) has no repeated roots and \(\Delta \neq 0\).
  Over this open set \(\sSplit_Y(\bar\chi) \to Y\) is étale by Proposition~\ref{split:facts}\ref{split:etale}.
  Since \(Y\) has rational singularities, so does \(U_2(\sSplit_Y(\bar\chi))\).
  The complement of these two sets is the intersection of  \(D_3(\sSplit_Y(\bar\chi))\) and  \(D_4(\sSplit_Y(\bar\chi))\), the inverse image of this set in \(\sZ\) has codimension at least \(2\)  by Proposition~\ref{simp-2:aux:D34-codim}.

  The variety \(Y\) is Cohen--Macaulay by Theorem~\ref{simp-2:norm+rat-sing}.
  Since \(U_2 \cup U_5\)  has rational singularities and the codimension of the preimage \(\widetilde\Phi ^{-1}(D_3 \cap D_4)\) in \(\sSplit_Y(\bar\chi)\) is bounded below by \(2\), we conclude that \(\sSplit_Y(\bar\chi)\) has rational singularities by \cite[Theorem~5.10]{KM:BirationalGeometry} and \cite[Proposition~1.11]{har1994:generalized} (also see \cite[Proposition~4.1]{ss2024:cohomology}).
\end{proof}

\begin{thm}
  \label{simp-2:split-integ-ratsing}
  Let \(X \colon [\ftyp{\ast t \dots 3 1 2 0}]\) be a flag supervariety and \(Y\) the simple \(2\)-fold compositional variety defined by \(\Delta_{fg} = \delta_1 - \delta_2\).
  Let \(f \colon V_0 \otimes  \sO_Y \to V_1 \otimes \sO_Y\) and \(g \colon V_0 \otimes \sO_Y \to V_1 \otimes \sO_Y\) be the universal linear maps, \(\chi(u) \in \sO_Y[u]\) the characteristic polynomial of \(fg\), and \(\bar\chi(u) = u^{- \Delta_{fg}} \chi(u)\)
  The splitting scheme \(\sSplit_{Y}(\bar\chi)\) is integral and has rational singularities.
\end{thm}
\begin{proof}
  The flag \(X' = \Flag(p_1|q_1, p_2|q_2;\;V)\) has type \([\ast 1 2 0]\) and defines the same compositional variety \(Y \subseteq \gone\) as \(X\) (cf. Proposition~\ref{p:type-extension}).
  Since the construction of the splitting ring only depends on \(V_0\), \(V_1\), \(Y\), and \(\bar\chi\), all the properties proven in the \(t = 2\) case apply in the case of arbitrary \(t\).
  Proposition~\ref{simp-2:split:reduced} and Proposition~\ref{simp-2:split:normal} imply that \(\sSplit_Y(\bar\chi)\) is integral and Proposition~\ref{simp-2:split:ratsing} implies it has rational singularities.
\end{proof}

\begin{rem}
  If \(Y\) is an arbitrary simple \(2\)-fold compositional variety defined by \(\Delta_{fg}\), then, after possibly imposing mild restrictions on \(\Delta_{fg}\), the same result holds for \(\sSplit_Y(\bar\chi)\).
  Given \(\Delta_{fg}\) one simply needs to construct a two-step flag supervariety of type \([\ftyp{\ast 1 2 0}]\)  such that \(\Delta_{fg} = \delta_1 - \delta_2 = p_1 - p_2 + q_2 - q_1\).
\end{rem}

\subsection{Factorization rings and their spectra}
\label{ss:gse-fact}
Recall that in the basic diagram \eqref{eq:basic-diag} of the geometric method, the projection \(\pi' \colon Z \to Y\) admits a Stein factorization through the affinization \(\widetilde Y\) of \(Z\).
In this subsection, we construct this space as a ring of invariants of the splitting rings from \S\ref{ss:gse-split}.
As before, \(f\) and \(g\) will denote the usual universal linear maps, \(\chi\) the characteristic polynomial of \(fg\), and \(\bar\chi\) the reduced characteristic polynomial.
We will identify parameters \(\bd = (d_1, \dots, d_t)\)  satisfying \(\sum d_i = \deg{\bar\chi}\) and work with the extended basic diagram 
\begin{equation}
  \label{eq:gse}
  \begin{tikzcd}
    &
    &
    Z
    \ar["\subseteq"]{rr}
    \ar["\pi' = \pi\big\vert_Z"]{dd}
    \ar["\varphi"]{dl}
    &
    &
    \ord{X} \times \gone
    \ar["\pi = \pi_{\gone}"]{dd}
    \\
    \sSplit_{Y}(\bar\chi)
    \ar["\widetilde \Phi", end anchor=west, start anchor=south east, bend right]{drr}
    \ar{r}
    &
    \sFact^{\bd}_{Y}(\bar\chi) \ar["\widetilde\varphi"]{dr}
    &
    \\
    &
    &
    Y
    \ar["\subseteq"]{rr}
    &
    & \gone
  \end{tikzcd}
\end{equation}
where the maps \(\widetilde \varphi\) and \(\widetilde\Phi\) are induced by the relevant ring maps and \(\varphi \colon Z \to \sFact_Y^\bd(\bar\chi)\) sends \((R_\bullet, S_\bullet, f, g)\) to \((f, g, p)\) where \(p = \prod_{i = 1}^t \chi'_i(u)\) is a factor of \(\chi\) and \(\chi'_i = u^{-k_i}(u - \lambda_{i, 1}) \dots (u - \lambda_{i, (q_i - q_{i- 1})})\) is the reduced characteristic polynomial of \(fg\) on \(S_i / S_{i - 1}\) and the \(k_i \geq 0\)  depend on \(Y\).
We will first determine the exact parameters \(\bd\) and \(\bar\chi\), prove that \(\varphi\) exists, and show that all objects appearing in \eqref{eq:gse} are well defined.

\begin{prop}
  \label{gse:def-fact}
  Let \(X\) be a partial flag supervariety belonging to one of the families \eqref{fam:T}, \eqref{fam:D1}, \eqref{fam:D2}, \eqref{fam:D3}, or \eqref{fam:C}.
  Let \(\bd = (d_1, \dots, d_{t + 1})\)  where \(d_i = \min(q_i - q_{i - 1}, p_i - p_{i - 1})\), and \(\bar\chi(u) = u^{\sum d_i - n} \chi(u)\).
  The projection \(\pi' \colon Z \to Y\) in \eqref{eq:gse} factors through \(\sFact^\bd_Y(\bar\chi)\) and there exists a surjective map \(\varphi \colon Z \to \sFact^\bd_Y(\bar\chi)\) satisfying \(\pi' = \widetilde \varphi \circ \varphi\).
\end{prop}
\begin{proof}
  We see that only \(d_1\) and \(d_2\) differ for the given families and the parameters can be computed as follows:
  \begin{equation}
    \label{table:fact-ring-params}
    \begin{array}{lcccc}
      & d_1 & d_2 & d_{i \geq 3} & \sum_i d_i = \deg{\bar\chi} \\
      \midrule
      \eqref{fam:T}  & q_1 & q_2-q_1 & q_i-q_{i-1} & n \\
      \eqref{fam:D1} & p_1 & q_2-q_1 & q_i-q_{i-1} & n + \delta_1 \\
      \eqref{fam:D2} & p_1 & q_2-q_1 & q_i-q_{i-1} & n + \delta_1 \\
      \eqref{fam:D3} & p_1 & p_2-p_1 & q_i-q_{i-1} & n + \delta_2 \\
      \eqref{fam:C} & q_1 & p_2-p_1 & q_i-q_{i-1} & n - \delta_1 + \delta_2 \\
      \bottomrule
    \end{array}
  \end{equation}
  Most of this proof essentially boils down to constructing the required preimages using the strategy showing in Corollary~\ref{five-families} and the constructions for the specific families mentioned therein.

  A point of \(\sFact_Y^\bd(\bar\chi)\) is a triple \((f, g, p)\)  where \((f, g) \in Y\) and \(p = \prod_{i = 1}^t\chi_i\) is an ordered factorization of \(\bar\chi\) and \(d_i = \deg{\chi_i}\).
  Since \(\widetilde \varphi\) is the morphism that sends such a triple to the point \((f, g)\), it is automatically surjective.
  So it is enough to describe the construction of \(\varphi\) in detail and the use surjectivity of \(\pi'\) to construct preimages of \(\varphi\).

  \textbf{Case \(\sT\):}
  This is a direct consequence of \cite[Theorem~6.3]{sam2026:borel}.

  \textbf{Case \(\sD_1\):}
  From Theorem~\ref{five-families}, \(Y\) is a determinantal variety determined by \(\Delta_{g} = -\delta_1\).
  Since \(\dim{\ker{g}} \geq (-\delta_1)\), \(u^{-\delta_1}\) divides \(\chi(u)\) and \(\bar\chi\) is well defined. 
  Now if \((R_\bullet, S_\bullet, f, g) \in Z\) then \(\dim{R_1} = p_1 \leq q_1 = \dim{S_1}\) and \(\dim{S_1} = q_1\) so the restriction of \(g\) to \(S_1\) satisfies \(\dim{\ker g\vert_{S_1}} \geq q_1 - p_1 = -\delta_1\).
  If \(L \subseteq \ker{g\vert_{S_1}}\) is \((-\delta_1)\)-dimensional then the characteristic polynomial \(\bar\chi\) of \(h' \colon V_1 / L \to V_1 / L\) induced by \(fg\) admits an ordered factorization \(\bar\chi = \chi_1 \cdots \chi_t\) where \(\chi_i\) is the characteristic polynomial of \(h'\) on \((S_i / L) \big/ (S_{i - 1} / L)\).
  Since
  \begin{align*}
    \deg{\chi_1} &= \dim{S_1 / L} = p_1, \\
    \deg{\chi_i} &= \dim{(S_{i} / L) \big/ (S_{i - 1} / L)} = \dim{S_i} - \dim{S_{i - 1}} = q_i - q_{i - 1} \quad\text{ for } i > 1,
  \end{align*}
  we see that the factorization is independent of \(L\) and \(\varphi\) is the morphism taking \((R_\bullet, S_\bullet, f, g)\) to the ordered factorization of \(\bar\chi\).
  If \((f, g, p) \in \sFact_Y^\bd(\bar \chi)\)  then write \(p = \prod_{i = 1}^t (u - \lambda_{i, 1}) \dots (u - \lambda_{i, \ell_i})\).
  Using the \(\lambda_{i, j_i}\) as the eigenvalues of \(fg\) in the construction of  \((R_\bullet, S_\bullet, f, g) \in \pi'^{-1}(f, g)\) given in  Proposition~\ref{p:type-extension}, Proposition~\ref{type:d201} shows that \((R_\bullet, S_\bullet, f, g) \in \varphi^{-1}(f, g, p)\).
  Thus \(\varphi\) is surjective.
  
  \textbf{Case \(\sD_2\) and \(\sD_3\):}
  In this case the resulting variety \(Y\) is again determinantal and the proofs are obtained, \emph{mutatis mutandis}, from the proof for in the case  \(\sD_1\).
  The map \(\varphi\) is obtained by quotient by an appropriate kernel and taking eigenvalues of \(fg\) and the construction of the preimages using the strategy given in Proposition~\ref{type:d021}, Proposition~\ref{type:d012}, and Proposition~\ref{p:type-extension} shows that \(\varphi\) is surjective.

  \textbf{Case \(\sC\):}
  Let \((R_\bullet, S_\bullet, f, g) \in Z\) and consider the induced map \(g' \colon S_2 / S_1 \to R_2 / R_1\).
  Since \(\dim{S_2 / S_1} = q_2 - q_1 \geq p_2 - p_1 = \dim{R_2 / R_1}\), \(\dim{\ker{g'}} \geq \delta_1 - \delta_2\).
  Picking \(L \oplus S_1 \subseteq S_2\) with \(\dim{L} = \delta_1 - \delta_2\), we consider the map induced by \(fg\) on \(h' \colon V_1 / L \to V_1 / L\).
  The flag \(S_1 / L \subseteq S_2 / L  \subseteq \dots \subseteq V_1 / L\) determines an ordered factorization of the characteristic polynomial \(\bar\chi\) of \(h'\).
  As in the case \(\sD_1\), calculating the degrees of the \(\chi_i\) shows that this construction is independent of \(L\) and \(\varphi\) is well defined.
  An element in the preimage of \((f, g, p) \in \sFact_Y^\bd(\bar\chi)\) can now be constructed using the roots of the \(\chi_i\) along with the procedure showed in Proposition~\ref{type:d120} and Proposition~\ref{p:type-extension}.
\end{proof}

We now state our main technical result on the spectra of factorization rings.
As the factorization rings are quotients of the splitting ring by a finite group, many of the results follow directly from analogous statements in \S\ref{ss:gse-split}.

\begin{thm}
  \label{t:gse-fact-rings}
  Let \(X\) be a flag supervariety belonging to one of the families \(\sT\), \(\sD_1\), \(\sD_2\), \(\sD_3\), or \(\sC\).
  Let \(\bd\) and \(\bar\chi\) be defined as in Proposition~\ref{gse:def-fact}.
  Let \(r_i = d_1 + \dots + d_i\).
  Then we have the following
  \begin{enumerate}
  \item \label{gse-fact:rank}
    The map \(\widetilde\varphi \colon \sFact^\bd_Y(\bar \chi) \to Y\) is finite and flat.
    The structure sheaf \(\sO_{\sFact^\bd_Y(\bar\chi)}\) is a free \(\sO_Y\)-module of rank \({\deg{\bar\chi} \choose d_1, \dots, d_t}\).

  \item \label{gse-fact:sing-iso}
    There is an isomorphism
    \begin{equation*}
      \sO_{\sFact^\bd_Y(\bar \chi)} \otimes_{\sO_Y} \bC
      \cong \rH_{\mathrm{sing}}^\bullet(\Flag(r_1, \dots, r_t;\; r_{t + 1}), \bC)
    \end{equation*}
    of graded rings.

  \item \label{gse-fact:rat-norm}
    The scheme \(\sFact^{\bd}_Y(\bar \chi)\) is integral and has rational singularities.
    In particular it is normal and Cohen--Macaulay.

  \item \label{gse-fact:higher-direct}
    The map \(\varphi\) satisfies \(\varphi_\ast \sO_Z = \sO_{\sFact^\bd_Y (\bar\chi)}\) and \(\rR^i\pi_\ast \sO_Z = 0\) for all \(i > 0\).
  \end{enumerate}
\end{thm}
\begin{proof}
  The statement~\ref{gse-fact:rank} is a direct consequence of Proposition~\ref{fact:facts}\ref{fact:quot-of-split}.
  Statement~\ref{gse-fact:sing-iso} follows from Proposition~\ref{iso:fact=singular-hom}.

  When \(X\) belongs to \(\sT\) then \ref{gse-fact:rat-norm} and \ref{gse-fact:higher-direct} follow from \cite[Theorem~6.3]{sam2026:borel}.

  When \(X\) belongs to \(\sD_i\) or to \(\sC\), Theorem~\ref{simp-1:split-integ-ratsing} and Theorem~\ref{simp-2:split-integ-ratsing} respectively imply that the spectra \(\sSplit_Y(\bar\chi)\) of the splitting rings over \(Y\) are integral and have rational singularities.
  By Proposition~\ref{fact:facts}\ref{fact:red} the factorization rings over \(Y\) are reduced.
  Since \(Z\) is irreducible and \(\varphi\) surjective by Proposition~\ref{gse:def-fact},  \(\sFact_Y^\bd(\bar\chi)\) is integral.
  From Proposition~\ref{fact:facts}\ref{fact:quot-of-split}, the factorization rings are quotients of the splitting rings by a finite group.
  Since the spectra of the splitting rings have rational singularities,  \cite{bou1987:singularits} now implies \ref{gse-fact:rat-norm}.

  If \(X\) belongs to \(\sD_i\), then \ref{gse-fact:higher-direct} follows from \cite[Proposition~5.17]{ss2024:cohomology}.
  Suppose \(X\) belongs to \(\sC\).
  In this case, the birationality argument of Proposition~\ref{simp-2:split:birational} shows that \(\Phi \colon U_5(\sZ) \to U_5(\sSplit_Y(\bar\chi))\) is an isomorphism; if \((f, g, p) \in U_5(\sFact_Y^\bd(\bar\chi))\) then \((R_\bullet, S_\bullet, f, g) \in \varphi ^{-1}(f, g, p)\) is unique by the same argument.
  Essentially, any preimage of \((f, g, p)\) in \(\sSplit_Y(\bar\chi)\) will corresponds to a factorization of \(p = \prod_{i = 1}^t \chi_i\) into linear terms.
  The uniquely determined point \((R_\bullet, S_\bullet, f, g) \in U_5(\sZ)\) maps to a point in  \(Z\) by recombining the eigenspaces of the linear terms into the eigenspaces of \(\chi_i\) and so every \((f, g, p) \in U_5(\sFact_Y^\bd(\bar\chi))\) has a unique preimage in \(U_5(Z)\).
  We conclude that \(\varphi\) is birational.
  Now \ref{gse-fact:higher-direct} follows in this case by noting that \(\varphi\) arises from the Stein factorization of \(\pi'\) so is proper \cite[III.4.3.1]{EGA3}.
\end{proof}

\section{Cohomology of flag supervarieties}
\label{s:cohomology}

For each of the five families given in Corollary~\ref{five-families}, we now completely describe the sheaf cohomology of their structure sheaf.
Note that the calculation for the family \eqref{fam:T} appears in \cite{sam2026:borel} but we restate it here in our notation.
We will prove the following main theorem.

\begin{thm}
  \label{main-thm}
  Suppose \(X\) is a flag supervariety belonging to one of the families \(\sT\), \(\sD_i\), or \(\sC\).
  There is an isomorphism of graded algebras
  \begin{equation*}
    \rH^\bullet(X, \sO_X) = \rH^\bullet_{\mathrm{sing}}(\Flag(r_1, \dots, r_t;\; r), \bC) \otimes E^\bullet
  \end{equation*}
  where \(E^\bullet\) is given by the graded \(\Tor\)-groups 
  \begin{equation*}
    E^i = \bigoplus_{j \geq 0} \Tor^A_j (\sO_Y, \bC)_{i + j}
  \end{equation*}
  of a compositional variety \(Y\).
  The type of \(X\), the parameters \(r_i\), and the rank conditions on \((f, g) \in Y\) depend on the family in the following way:
  \begin{equation*}
    \begin{array}{cclllll}
      & \tau & r_1 & r_2 & r_{i \geq 3} & r & \text{Conditions on } (f, g) \\
      \midrule
      \eqref{fam:T} & [\ftyp{\ast t \dots 3 2 1 0}]
          & q_1 & q_2    & q_i    & n & \emptyset
      \\
      \eqref{fam:D1} & [\ftyp{\ast t \dots 3 2 0 1}]
          & p_1 & q_2+\delta_1 & q_i+\delta_1 & n+\delta_1 &  \rank{g} \leq n + \delta_1
      \\
      \eqref{fam:D2} & [\ftyp{\ast t \dots 3 0 2 1}]
          & p_1 & q_2+\delta_1 & q_i+\delta_1 & n+\delta_1 & \rank{g} \leq n + \delta_1
      \\
      \eqref{fam:D3} & [\ftyp{\ast t \dots 3 0 1 2}]
          & p_1 & p_2    & q_i+\delta_2 & n+\delta_2 &  \rank{g} \leq n + \delta_2
      \\
      \eqref{fam:C} & [\ftyp{\ast t \dots 3 1 2 0}]
          & q_1 & p_2-\delta_1 & q_i+\delta_2-\delta_1 & n-\delta_1+\delta_2 & \rank{fg} \leq n - \delta_1 + \delta_2
      \\
      \bottomrule
    \end{array}
  \end{equation*}
\end{thm}

For the rest of this section, unless specified otherwise, \(X\) will be a partial flag supervariety belonging to one of the families \eqref{fam:T}, \eqref{fam:D1}, \eqref{fam:D2}, \eqref{fam:D3}, or \eqref{fam:C}.
We will be working with the diagram 
\begin{equation}
  \label{diag:cohom}
  \begin{tikzcd}
    &
    Z
    \ar["\subseteq"]{rr}
    \ar["\pi' = \pi\big\vert_Z"]{dd}
    \ar["\varphi", end anchor=north east]{dl}
    &
    &
    \ord{X} \times \gone
    \ar["\pi = \pi_{\gone}"]{dd}
    \\
    \widetilde Y = \sFact^{\bd}_{Y}(\bar\chi)
    \ar["\widetilde\varphi", start anchor=south east]{dr}
    &
    \\
    &
    Y
    \ar["\subseteq"]{rr}
    &
    & \gone
  \end{tikzcd}
\end{equation}
where \(\bd\) and \(\bar\chi\) are defined as in Proposition~\ref{gse:def-fact}.
Let \(A = \Sym(\gone^\ast)\) be the coordinate ring of the affine variety \(\gone = \Hom(V_0, V_1) \times \Hom(V_1, V_0)\).
Let \(r_i = d_1 + \dots + d_i\) and let \(H_\rs = \rH^\bullet_{\mathrm{sing}}(\Flag(r_1, \dots, r_t;\;r_{t + 1}), \bC)\) be the singular cohomology ring appearing in Theorem~\ref{t:gse-fact-rings}\ref{gse-fact:sing-iso}.

\subsection{Cohomology spectral sequence}
\label{ss:cohom-ss}
A flag supervariety \(X\) is a quotient \(\bG / \bP\) of the algebraic supergroup \(\bG = \GL(V)\) by an appropriate parabolic supersubgroup \(\bP\).
Since the bosonic reduction \(\ord{X}\) of \(X\) is product of two ordinary flag varieties, there is a short exact sequence
\begin{equation*}
  0 \to \sJ / \sJ^2 \to \sO_{\ord{X}} \otimes \gone^\ast \to \eta \to 0
\end{equation*}
where \(\sJ \subseteq \sO_X\) is the ideal sheaf generated by the odd elements.
We see that \(Y\) in \eqref{diag:cohom} can be identified with \(\Spec(\Sym(\eta))\), see Theorem~\ref{t:super-geom-method}.
Since the ideal sheaf \(\sJ\) defines the bosonic reduction \(\ord{X}\) as a subscheme of \(X\), the natural \(\sJ\)-adic filtration on \(\sO_X\) results in a spectral sequence relating cohomology of \(\ord{X}\) to that of \(X\), see \S\ref{s:prelim-i}.

\begin{prop}
  \label{p:flag-ss}
  There exists a natural isomorphism
  \begin{equation*}
    \rH^q(X, \gr^{p + q}\sO_X) = \Tor^A_p(\sO_{\widetilde Y}, \bC)_{p + q}
  \end{equation*}
  and a spectral sequence
  \begin{equation*}
    \rE^{p, q}_1 = \Tor^A_{-q}(\sO_{\widetilde Y}, \bC)_p \implies \rH^{p + q}(X, \sO_X).
  \end{equation*}
\end{prop}
\begin{proof}
  Since \(\widetilde Y = \sFact_Y^\bd(\bar\chi)\), the higher cohomology vanishes by Theorem~\ref{t:gse-fact-rings}\ref{gse-fact:higher-direct}.
  The claim now follows from Theorem~\ref{t:super-geom-method}.
\end{proof}

\begin{prop}
  \label{p:hs-twidle-iso}
  The graded vector space \(\Tor_p^A(\sO_{\widetilde Y}, \CC)\) is naturally an \(H_\rs\)-module and the induced map
  \begin{equation*}
    H_\rs \otimes _{\bC} \Tor^A_p(\sO_Y, \bC) \to \Tor^A_p(\sO_{\widetilde Y}, \bC)
  \end{equation*}
  is an isomorphism of graded \(H_\rs\)-modules.
\end{prop}
\begin{proof}
  By construction, \(H_\rs\) is isomorphic to \(\sO_{\widetilde Y} \otimes_{\sO_Y} \bC\) by Theorem~\ref{t:gse-fact-rings}\ref{gse-fact:sing-iso}.
  Since \(\sO_{\widetilde Y}\) is a free \(\sO_Y\)-module by Theorem~\ref{t:gse-fact-rings}\ref{gse-fact:rank}, the claim now follows from \cite[Proposition~6.8]{ss2024:cohomology}.
\end{proof}

We are now ready to analyze the spectral sequence of Proposition~\ref{p:flag-ss}.
Let 
\begin{equation*}
  \widetilde L_k = \bigoplus_{p \geq 0} \Tor^A_{p}(\sO_{\widetilde Y}, \bC)_{p + k},
  \quad
  L_k = \bigoplus_{p \geq 0} \Tor^A_{p}(\sO_{Y}, \bC)_{p + k}.
\end{equation*}

\begin{prop}
  \label{p:ss-degen}
  The spectral sequence of Proposition~\ref{p:flag-ss} degenerates at the first page.
\end{prop}
\begin{proof}
  From Proposition~\ref{p:flag-ss}, we have \(\rE^{p, q}_1 = \Tor^A_{-q}(\sO_{\widetilde Y}, \bC)_p\).
  Let \(\rE_r^k = \bigoplus_{p + q = k} \rE_r^{p, q}\) so the differential is a map \(d_r \colon \rE_r^k \to \rE_r^{k + 1}\) and we have \(\rE_1^k =  \widetilde L_k\).

  If \(X\) belongs to \(\sT\) then the result appears as part of \cite[Theorem~6.3]{sam2026:borel}.

  If \(X\) belongs to \(\sD_i\) then the spectral sequence degenerates by \cite[Proposition~6.10]{ss2024:cohomology}.
  The main observation in this case is that \(\widetilde L_k\) is related to the Lascoux resolution and decomposes as multiplicity-free \(\ord{\bG}\)-representations.
  Since \(\widetilde L_k\) and \(\widetilde L_{k + 1}\) have no simple factors in common, the differential is forced to be zero as it is \(\ord{\bG}\)-equivariant.

  Now assume \(X\) belongs to \(\sC\).
  Theorem~\ref{simp-2:res} implies that \(\sO_Y\) is resolved by \(\bF_\bullet \otimes B\) where \(B = \bC[z_{i, j}]_{1 \leq i, j \leq n}\) and \(\bF_\bullet\) is the Lascoux resolution of a determinantal variety.
  Since the \(z_{i, j} \in B\) are related to \(x_{i, j}, y_{i', j'} \in A\) via  \(z_{i, j} = \sum_{k} x_{i, k} y_{k, j}\).
  Thus \(\deg{x_{i, j}} = \deg{y_{i', j'}} = 1\) implies \(\deg{z_{i, j}} = 2\).
  In terms of the \(A\)-module structure on the \(\Tor\)-groups, the generators all appear in even degrees.
  In other words, if \(\Tor_p(\sO_Y, \bC)_{p + k} \to \Tor_{p'}(\sO_Y, \bC)_{p' + k + 1}\)  is the map coming from the differential \(d_1\), then one of the \(\Tor\)-groups is necessarily zero and thus \(d_1\) is the zero map.
  We conclude that the spectral sequence degenerates on the first page.
\end{proof}

\subsection{Proof of the main theorem}
\label{ss:mainthm}
We now combine all the ingredients to prove our main theorem.

\begin{proof}[Proof of Theorem~\ref{main-thm}]
  Let \(X\) belong to one of the families \(\sT\), \(\sD_i\), or \(\sC\).
  Then Corollary~\ref{five-families} shows that \(Y\) has the stated form.
  Proposition~\ref{gse:def-fact} and Theorem~\ref{t:gse-fact-rings} show  \(\widetilde Y = \sFact^\bd_Y(\bar\chi)\) and thus \(H_\rs = \rH^\bullet_{\mathrm{sing}}(\Flag(r_1, \dots, r_t;\;r), \bC)\) is defined by the given parameters.
  In all cases, the cohomology spectral sequence of Proposition~\ref{p:flag-ss} degenerates by Proposition~\ref{p:ss-degen}; we obtain isomorphisms
  \begin{equation*}
    \rH^i(X \sO_X) = \gr(\rH^i(X, \sO_X)) = \widetilde L_i
  \end{equation*}
  for all \(i\).
  Moving to the algebra \(\rH^\bullet(X, \sO_X)\),  using the definition of \(\widetilde L_i\) and the isomorphism from Proposition~\ref{p:hs-twidle-iso}, 
  \begin{equation*}
    \gr(\rH^\bullet(X, \sO_X))
    = \bigoplus_{p, q} \Tor^A_p(\sO_{\widetilde Y} \bC)_{p + q}
    = H_\rs \otimes \bigoplus_{p, q} \Tor_p^A(\sO_Y, \bC)_{p + q}.
  \end{equation*}
  Note that \(L_i\) in fact belongs to \(\gr(\rH^i(X, \sO_X))\) and that \(\gr(\rH^\bullet(X, \sO_X))\) is a free \(H_\rs\)-module.
  The latter is a consequence of Theorem~\ref{t:super-geom-method} (description of \(\Tor\)-groups), Theorem~\ref{t:gse-fact-rings}\ref{gse-fact:rank} (\(\sO_{\widetilde Y}\) a free \(\sO_Y\)-module), and generalities of cohomology base change; see \cite[Remark~2.3]{ss2024:cohomology} and \stacks{08IB}.
  The associated graded \(\gr(E^i)\) corresponds to the copy of \(L_i\) in the final direct sum above, so \(H_\rs \otimes \gr(E^\bullet) \to \gr(\rH^\bullet(X, \sO_X))\) is an isomorphism.
  Thus, \(\rH^\bullet(X, \sO_X)\) and \(H_\rs \otimes E^\bullet\) are isomorphic as algebras and the result follows.
\end{proof}

\begin{cor}
  \label{c:eulerchar}
  The super Euler characteristic of \(X\) is given by
  \begin{equation*}
    \chi(X)
    = \sum_{i \ge 0} (-1)^i\sdim{\rH^i(X, \sO_X)}
    = {\sum d_i \choose d_1, \dots, d_{t}}.
  \end{equation*}
\end{cor}
\begin{proof}
  From \cite[Corollary~2.5]{ss2024:cohomology}, \(\chi(X)\) is the degree of the map \(\widetilde \varphi\) over a generic point of \(Y\).
  Using the spectral sequence from Proposition~\ref{p:flag-ss}, we see that
  \begin{equation*}
    \chi(X)
    = \sum_{i \ge 0} \sdim{\rH^i(X, \sO_X)}
    = \sum_{i \geq 0} (-1)^p \Tor_p^A(\sO_{\widetilde Y}, \bC)
  \end{equation*}
  where the latter coincides with \(\dim{\mathrm{Frac}(A) \otimes_A \sO_{\widetilde Y}}\).
  By Theorem~\ref{t:gse-fact-rings}\ref{gse-fact:rank}, \(\widetilde \varphi \colon \widetilde Y \to Y\) is finite and \(\sO_{\widetilde Y}\) is a free \(\sO_Y\)-module of rank \({ \sum d_i \choose d_1, \dots, d_{t}}\).
  The claim now follows by comparing dimensions and using Proposition~\ref{p:flag-ss}.
\end{proof}

\begin{rem}
  If \(\sV\) is a super vector bundle on a superscheme \(X\) such that \(\ord{X}\) is projective, then the argument in \cite[Theorem~6.2]{sam2026:borel} can be adapted to prove an analogue of Theorem~\ref{main-thm} for the relative partial flag supervariety \(\Flag_X(p_1|q_1, \dots, p_t|q_t;\;\sV)\).
  We omit it here for simplicity.
\end{rem}

\bibliographystyle{alphaurl}
\bibliography{compositional-refs}
\end{document}